\documentclass[11pt]{article}
\usepackage[english]{babel}
\usepackage[left=1in,right=1in,top=1in,bottom=1.5in]{geometry}
\usepackage[skip=5pt, indent=10pt]{parskip}
\usepackage{amssymb,amsmath,amsthm}
\usepackage{mathtools,bm}              
\usepackage{hyperref}
\usepackage{xcolor}
\usepackage{nicefrac}
\usepackage[margin=1cm]{caption}
\usepackage{natbib}
\usepackage[makeroom]{cancel}
\usepackage{bbm}
\usepackage{mathrsfs}
\usepackage{times}
\usepackage{graphicx}
\usepackage{subcaption}
\usepackage{multirow}
\usepackage{diagbox}
\usepackage{siunitx}
\usepackage{cleveref}
\usepackage{xargs}
\usepackage{algorithm}
\usepackage{algorithmic}
\usepackage{enumitem}
\usepackage{url}

\allowdisplaybreaks
\definecolor{darkred}{HTML}{880000}
\definecolor{darkblue}{HTML}{000088}

\hypersetup{
  colorlinks,
  linkcolor={darkred},
  citecolor={darkblue}
}

\newtheorem{proposition}{Proposition}[section]
\newtheorem{lemma}[proposition]{Lemma}
\newtheorem{theorem}[proposition]{Theorem}
\newtheorem{corollary}[proposition]{Corollary}
\theoremstyle{definition}
\newtheorem{definition}[proposition]{Definition}
\newtheorem{assumption}{Assumption}
\theoremstyle{remark}
\newtheorem{remark}[proposition]{Remark}

\numberwithin{equation}{section}

\newcommand{\R}{\mathbb{R}}

\newcommand{\E}{\mathbb{E}}

\newcommand{\Var}{\operatorname{Var}}
\newcommand{\Cov}{\operatorname{Cov}}
\newcommand{\Hcal}{\mathcal{H}}

\newcommand{\norm}[1]{\left\lVert #1 \right\rVert}

\title{WEIRDO: WEak resIdual Regularized DOob's h-transform diffusion alignment}

\author{
Denis Suchkov\thanks{HSE University, Russian Federation}
}

\date{}

\begin{document}

\maketitle

\begin{abstract}
We study the problem of estimating the guidance that steers the distribution learned by a diffusion generative model toward a tilted target $q_0 \propto w\,p_0$ at inference time. Relying on the stochastic optimal control approach, we observe that the exact drift correction is the gradient of the logarithm of Doob's $h$-function, and we study the problem of estimating it from a sample. In the present paper, we assume that the score of the pretrained model is available, that the tilting weight is bounded and positive, and that the reference distribution has a bounded support, no smoothness of the weight is required. Introducing a penalized
least-squares risk in which the penalty is the residual of the space-time
harmonicity equation satisfied by the $h$-function, measured in a dual Sobolev norm, we derive high-probability bounds on the squared error of the resulting guidance estimate. Since the penalty vanishes at the target, the estimator is free of regularization bias, and in favourable scenarios its rate of convergence is faster than the minimax rate of estimating first-order derivatives of a smooth regression function. Assuming that $w$ is bounded and positive with $\mathbb{E}_{p_0}[w^{-\mathrm{s}}] < \infty$ for some $\mathrm{s} \in (0,\infty]$, and that the reference data are compactly supported, we prove that the guidance is estimable in squared $L^2$ at rate $\varepsilon_n^{\mathrm{s}/(\mathrm{s}+4)}$, where $\varepsilon_n = n^{-2(\beta-1)/(2(\beta-1)+d)}.$ We also transfer the obtained bounds to the total variation distance between the marginals of the estimated and the exactly guided samplers, and illustrate the performance of the suggested approach with numerical experiments.
\end{abstract}

\section{Introduction} \label{sec:intro}

Score-based diffusion models \citep{song2021} have become a standard tool of
generative modelling. Such a model transforms data into noise by a forward
stochastic differential equation and produces new samples by running the
corresponding time reversal. Let $p_0$ denote the data distribution on $\R^d$ and let the forward dynamics be the variance-preserving Ornstein--Uhlenbeck process
\begin{equation} \label{eq:intro_forward}
    dX_\tau = -X_\tau \, d\tau + \sqrt{2} \, dB_\tau, \qquad 0 \leq \tau \leq T,
    \qquad X_0 \sim p_0 ,
\end{equation}
where $B$ is a standard $d$-dimensional Brownian motion and $T > 0$ is a fixed time horizon. We denote the law of $X_\tau$ by $p_\tau,$ which for $ \tau > 0 $ has a smooth positive density denoted in the same way, and note that $p_\tau$ approaches the standard Gaussian density as $\tau$ grows. It is well known \citep{anderson1982,haussmann1986} that the time reversal
$\overleftarrow{X}_t \overset{d}{=} X_{T-t}$ of the process
\eqref{eq:intro_forward} is again a diffusion, governed by the SDE
\begin{equation} \label{eq:rev_denoising} 
    d \overleftarrow{X}_t = b(t, \overleftarrow{X}_t)\,dt + \sqrt{2}\,dB_t,
    \qquad b(t,x) = x + 2\,s(t,x), \qquad s(t,x) = \nabla \log p_{T-t}(x),
\end{equation}
started at $\overleftarrow{X}_0 \sim p_T$ and terminating at
$\overleftarrow{X}_T \sim p_0$. The function $s$ is referred to as the score. It is not available in closed form and, in practice, is learned by a neural network. Throughout the paper, a pretrained diffusion model is understood as a supplier of $s$, and we treat $s$ as known.

Once such a model is available, one usually wants to adapt it to a particular task or reward, and retraining it is expensive. Let $w : \R^d \to \R_{\geq 0}$ be a weight expressing the new objective: a reward exponential $e^{r/\alpha}$, a likelihood in an inverse problem, or a class probability. The aim is then to sample not from $p_0$ but from the tilted distribution
\begin{equation} \label{eq:intro_tilt}
    q_0 \;\propto\; w \, p_0 .
\end{equation}
Several routes are available. One can fine-tune the network \citep{uehara2024}, resample from a heuristic proposal \citep{wu2023}, or add a drift to \eqref{eq:rev_denoising} \citep{dhariwal2021,ho2022cfg}. We are interested in the third kind, known as inference-time alignment, because it leaves the pretrained weights untouched: the pretrained network is never retrained, only an auxiliary estimator of the correction is fitted per reward, and there is no need to draw many samples in order to keep a few. Unlike the usual heuristics, however, we work with the drift that is exact.

That drift is determined by a single function. From the optimal control point of view, turning $p_0$ into $q_0$ reduces to finding
\begin{equation} \label{eq:intro_h}
    h^\ast(t,x) \;:=\; \E\left[ w(X_0) \mid X_{T-t} = x \right],
\end{equation}
known as Doob's $h$-function \citep{doob1984}. Adding $2\nabla_x \log h^\ast$ to the drift of \eqref{eq:rev_denoising} produces
\begin{equation} \label{eq:intro_guided}
    d\overleftarrow{Z}_t = \big( b(t,\overleftarrow{Z}_t) + 2\nabla_x\log h^\ast(t,\overleftarrow{Z}_t) \big) dt + \sqrt{2}\,dB_t ,
\end{equation}
whose terminal law is exactly $ q_0, $ when it is started from the tilted law $ h^* (0, \cdot) p_T / Z, \, Z := \int w dp_0. $ What the sampler uses is the guidance $g^\ast := \nabla \log h^\ast$, and since $p_0$ is unknown, $h^\ast$ has to be estimated from data. This is done on a window of reverse times $ I = [t_0, t_1] \subset (0, T) $ kept away from both endpoints, since as $ t \uparrow T $ the noise vanishes and recovery is ill-posed, while as $ t \downarrow 0 $ there is nothing to recover.

Our focus lies on statistical aspects of guidance estimation, and here the
following difficulty arises. By \eqref{eq:intro_h}, the function $h^\ast$ is a conditional expectation, so a natural strategy is to estimate it by nonparametric regression. The sampler \eqref{eq:intro_guided}, however, does not consume $h^\ast$ but $\nabla \log h^\ast$, and a small value error does not by itself imply a small gradient error. Estimating the derivative is also intrinsically harder than estimating the function: for a $\beta$-smooth regression function observed in noise, the optimal squared $L^2$ rate is $n^{-2\beta/(2\beta+d)}$ for the function and $n^{-2(\beta-1)/(2\beta+d)}$ for its first derivative \citep{stone1982}, one derivative's worth slower. In order to do better, one has to exploit a property of $h^\ast$ that an arbitrary conditional expectation does not possess.

The property we use is that $h^\ast$ is space-time harmonic for the generator of \eqref{eq:rev_denoising}, that is, it satisfies the PDE
\begin{equation} \label{eq:intro_harm}
    \mathcal{M} h^\ast \;:=\; \partial_t h^\ast + b(t,x) \cdot \nabla h^\ast + \Delta h^\ast \;=\; 0 .
\end{equation}
Every term in \eqref{eq:intro_harm} is either a derivative of the unknown function or, through $b$, the score of the pretrained model. The residual $\mathcal{M}h$ can therefore be evaluated on the sample and added to the regression objective as a penalty. Since it vanishes at $h^\ast$, this penalty introduces no bias at the population level, whatever weight it is given.

However, the Laplacian in \eqref{eq:intro_harm} makes residual second order in the unknown, and second derivatives are exactly what the rate we are aiming for cannot control. We therefore pass to a weak form: we test \eqref{eq:intro_harm} against a function $\psi$ and integrate by parts, which moves one derivative from $h$ onto $\psi$. This step is available to us for a specific reason. Integrating by parts against the marginal $p_{T-t}$ produces its logarithmic derivative as a weight, and that is again the score already appearing in \eqref{eq:rev_denoising}. What results is a penalty that is first order in both arguments and computable from the pretrained model. From it we obtain an estimator of $h^\ast$ whose guidance converges to $g^\ast$ in the squared $L^2$ norm at a rate we make explicit, and the same bound transfers to the marginals of \eqref{eq:intro_guided} in total variation.

Throughout we take $\beta = 2$ as the leading case, treating the approximation of $h^\ast$ as that of a $W^{\beta,\infty}$ function. Although $h^\ast$ is smooth for any bounded $w$, as we show below, $\beta = 2$ is the smallest integer for which the minimax rate $n^{-2(\beta-1)/(2\beta+d)}$ of \citet{stone1982} is informative, and hence the smallest for which a comparison against it is meaningful. This should be read not as a assumption that $h^\ast$ belong to that smoothness class, but as the primary example in which the difference between our estimate and the benchmark is clearest.

\paragraph{Contribution.}
\begin{itemize}[leftmargin=1.4em,itemsep=2pt,topsep=3pt]
    \item The main contribution of this paper is a non-asymptotic
    high-probability upper bound on the squared $L^2$-error of the guidance
    estimate (Theorem \ref{thm:rate_main}), which holds on all but an arbitrarily
    small end-portion of the estimation window. In favourable scenarios, when the
    weight $w$ is not too close to zero, the obtained rate is faster than the
    generic minimax rate of first-derivative estimation for a smooth
    regression function.
    \item The proof relies on an exact energy identity for the reverse diffusion,
    which allows us to control the gradient by the value error and the weak
    residual (Theorem \ref{thm:coercivity_main}). We would like to emphasize that
    this result requires no inf-sup condition on the trial and test classes, a
    hypothesis that a weak formulation would ordinarily need, it is replaced by a
    containment between the two classes.
    \item We transfer the obtained bound to the total variation distance between
    the marginals of the estimated and the exactly guided samplers on the
    estimation window (Corollary \ref{cor:sampling_main}), so that an error in the
    estimated guidance is controlled as an error in the law that is sampled.
\end{itemize}

\subsection*{Related work}

\textbf{Inference-time alignment.} The oldest way of steering a diffusion sampler is to add a heuristic drift: classifier guidance \citep{dhariwal2021} uses the gradient of a noise-conditioned classifier, and classifier-free guidance \citep{ho2022cfg} extrapolates between conditional and unconditional scores. Both are effective but not exact: classifier guidance reproduces the conditional $h$-transform only with an exact noise-conditioned classifier at unit scale and the matching initial law, and other scales change the terminal law by an unquantified amount. The alternatives split by what they modify. Fine-tuning changes the weights, by reinforcement learning or stochastic control, and is reviewed by \citet{uehara2024}. Sequential Monte Carlo leaves the model alone and corrects a heuristic proposal by importance weights, as in the twisted sampler of \citet{wu2023}, paying for exactness with a particle system and with weight degeneracy in high dimension. The line closest to ours keeps the sampler exact and learns the $h$-transform itself. \citet{denker2024} fit a generalised $h$-transform on top of a frozen model, without a rate, and \citet{zhu2026doit} estimate it at sampling time by Monte Carlo, with a total-variation guarantee governed by the number of look-ahead draws rather than by a sample size. \citet{chang26} derive convergence rates for a variational estimator of $h^\ast$, but their smoothness penalty does not vanish at the target, so the rate reflects a bias--variance trade-off in the penalty weight. In \citet{guo2026}, who estimate $h$ and its gradient through martingale identities, and \citet{kawata2025}, who sample dual-averaging iterates through the $h$-transform, the drift error enters the bound as a parameter rather than being estimated at a rate. The question here is accordingly not whether the drift correction admits a statistical theory, but what it costs: which rate is attainable when the penalty incurs no bias at all, and what has to be assumed to get it.

\textbf{Statistical theory for diffusion models.} The nonparametric theory of diffusion models has concentrated on the score. \citet{oko2023} prove minimax-optimal distribution estimation under empirical score matching, and a parallel line takes the score as given and asks what sampling error follows \citep{chen2023sampling,li2024towards}. Because rates in the ambient dimension degrade quickly, much recent effort targets the intrinsic dimension instead: \citet{chen2023low} under an exact manifold hypothesis, \citet{yakovlev2025} under a substantially relaxed one --- a nonparametric Gaussian mixture in which samples may leave the manifold, with rates set by the intrinsic dimension and valid even when the ambient dimension grows polynomially in $n$ --- and \citet{li2024adapting} for the sampler, showing DDPM adapts to unknown low-dimensional structure. The guidance is a log-gradient in the same geometry as the score, but is estimated from a regression whose response is the terminal weight, and the error that matters is in the gradient.

\textbf{Empirical process theory.} Localization by local Rademacher
complexities \citep{Bartlett_2005,koltchinskii2006}, in the form set out by
\citet{wainwright2019}, yields multiplicative rather than additive comparisons
between empirical and population quantities: on a star hull satisfying a
Bernstein condition $Pf^2 \lesssim Pf$, the deviation is governed by a critical
radius read off a uniform entropy bound. That theory is built for bounded
classes. For suprema with unbounded envelopes, \citet{adamczak2008tailinequalitysupremaunbounded}
supplies a Talagrand-type tail inequality that applies after truncation and
peeling. We need both, the drift coefficient $x + s$ being only sub-Gaussian
under $\rho_t$.

\section{Setup} \label{sec:setup}

\textbf{Notation and conventions.} We keep the processes of Section
\ref{sec:intro} and fix $T > 0$. The forward SDE \eqref{eq:intro_forward} has
Gaussian transition kernel $X_\tau \mid X_0 = x_0 \sim \mathcal{N}(\mu_\tau x_0,
\sigma^2_\tau I_d)$ with $\mu_\tau = e^{-\tau}$ and $\sigma_\tau^2 = 1 -
e^{-2\tau}$, so $\mu_\tau^2 + \sigma_\tau^2 = 1$, and we write $p_\tau :=
\mathrm{Law}(X_\tau)$. The reversal \eqref{eq:rev_denoising} is started at
$ \overleftarrow{X}_0 \sim p_T, $ its marginals $\rho_t :=
\mathrm{Law}(\overleftarrow{X}_t) = p_{T-t}$ are the laws from which the
estimation sample is drawn. Write
\begin{equation} \label{eq:generator}
    (\mathcal{L}_t f)(x) := b(t,x) \cdot \nabla f(x) + \Delta f(x), \qquad
    \mathcal{M} := \partial_t + \mathcal{L}_t
\end{equation}
for the generator of \eqref{eq:rev_denoising} and the associated space-time
operator, the Laplacian rather than $\tfrac12 \Delta$ appears because the
diffusion coefficient is $\sqrt{2}$. The score $s$ is supplied by the pretrained
model, and we treat it as known throughout. As is standard in the statistical analysis of guidance, we treat $s$ as exact, which separates the approximation error of the pretrained model from the error of estimating $h^\ast$. The image experiments of Section~\ref{sec:experiments} deliberately drop this idealization and use the model's learned score, to test the estimator where it would be used in practice.

\textbf{Target and guidance.} Given a known weight $w \geq 0$, the target is $q_0 = w p_0 / Z$ with $Z = \int w \, p_0$. Doob's $h$-transform adds the drift $2 \nabla \log h^\ast$ to \eqref{eq:rev_denoising}, giving the controlled process
\begin{equation} \label{eq:rewersed_process}
    d \overleftarrow{Z}_t = \big( \overleftarrow{Z}_t + 2 s(t, \overleftarrow{Z}_t) + 2 \nabla \log h^\ast(t, \overleftarrow{Z}_t) \big) dt + \sqrt{2} \, dB_t ,
\end{equation}
whose terminal law is $ q_0, $ when started from $ \rho_0^* := h^*(0, \cdot) p_T / Z $ by Lemma \ref{lem:marginals_bounded}. Here $h^\ast(t,x) := \E[w(\overleftarrow{X}_T) \mid \overleftarrow{X}_t = x] = \E[w(X_0) \mid X_{T-t} = x]$ is Doob's $h$-function and $g^\ast := \nabla \log h^\ast$ the guidance. By Bayes' rule for the Gaussian likelihood, $h^\ast(t,x) = \E_{X_0 \sim \pi_{t,x}}[w(X_0)]$ for the posterior $\pi_{t,x}(dx_0) \propto \phi_{\sigma^2_{T-t}}(x - \mu_{T-t} x_0) p_0(dx_0)$.

\textbf{Assumptions.} Exactly two hypotheses are used anywhere below.

\begin{assumption}[Weight] \label{ass:bounded}
    $w : \R^d \to \R$ is measurable with $0 < w \leq \bar B < \infty$ on $\mathrm{supp}(p_0)$, and there is $\mathrm{s} \in (0, \infty]$ with $\Xi_{\mathrm{s}} := \E_{p_0}[w^{-\mathrm{s}}] < \infty$, where $\mathrm{s} = \infty$ is read as $w \geq B > 0$ on $\mathrm{supp}(p_0)$.
\end{assumption}

\begin{assumption}[Compact reference support] \label{ass:compact}
    $p_0$ is supported in the cube $\{x_0 : |x_0|_\infty \leq 1\}$.
\end{assumption}

Reward and likelihood tilts with bounded reward are of this form, $\Xi_{\mathrm{s}}$ quantifies how close $w$ comes to zero, degenerating at $\mathrm{s} = \infty$ to a uniform lower bound. Assumption \ref{ass:compact} is the usual compact-support hypothesis, and its only role is to make the posterior compactly supported, hence to yield the regularity bounds of Appendix \ref{app:regularity}. No smoothness of $w$ is assumed anywhere: the smoothness index $\beta$ below is a free parameter of the construction, and the requisite derivative envelopes for $h^\ast$ are derived from Assumptions \ref{ass:bounded}--\ref{ass:compact} (Lemma \ref{lem:space_time_approx}, Step 1). Assumption~\ref{ass:bounded} lets $ w $ vanish, which the corresponding hypotheses of prior work do not: \citet{chang26} and \citet{kawata2025} bound $ w $ and the density ratio, respectively, from both sides, and \citet{zhu2026doit} bounds $ h $ from below. \citet{guo2026} work under conditions of a different kind, strong log-concavity of the marginals and of $ h $ with the score and the guidance error controlled in sup norm, which are not directly comparable to ours. The order $ \mathrm{s} $ is free here and enters the rate, so how far $ w $ may vanish is traded explicitly against how fast the guidance converges, a trade we measure in Appendix \ref{app:finiteS}.

\textbf{Window and function spaces.} As $t \uparrow T$ the noise level $\sigma_{T-t} \to 0$ and estimation is ill-posed, as $t \downarrow 0$ there is nothing to recover. We fix an early--late window
\begin{equation} \label{eq:window}
    I := [t_0, t_1] \subset (0,T), \quad \ell := t_1 - t_0 \leq 1, \quad \sigma_\ast := \sigma_{T - t_1} > 0,
\end{equation}
and a probability weight $\nu$ on $I$, taken to be normalized Lebesgue measure. For each $t$, $L^2(\rho_t)$ and $H^1(\rho_t) = \{ f : f, \nabla f \in L^2(\rho_t) \}$ carry the reverse marginal as reference measure, and $H^{-1}(\rho_t)$ is the dual of $H^1(\rho_t)$ under the pairing
\begin{equation} \label{eq:h_inv_norm}
    \| F \|_{H^{-1}(\rho_t)} := \sup_{0 \neq \psi \in H^1(\rho_t)} \frac{\langle F, \psi \rangle_{\rho_t}}{\| \psi \|_{H^1(\rho_t)}} \;\; \leq \;\; \| F \|_{\rho_t} .
\end{equation}
On the window we use the space-time spaces $\mathcal{V} := L^2_\nu(I; H^1(\rho_t))$ and its dual $\mathcal{V}^\ast$, with $\| F \|_{\mathcal{V}^\ast}^2 = \int_I \| F(t,\cdot) \|^2_{H^{-1}(\rho_t)} \nu(dt)$. Constants written $\sigma^{-c}$ are of the form $ C_\star \sigma_\ast^{-c}, $ where $ c $ is absolute and $ C_\star $ depends only on $ (d, \beta, \bar{B}, \ell, C_t, \kappa_\Theta) $ and not on $ n, N, \zeta, \delta $ or $ \sigma_\ast, $ the relations $ \lesssim $ and $ \asymp $ hide constants of the same form.

\section{Harmonicity and its weak form} \label{sec:weak_residual}

Write $\mathcal{R}_t[h] := (\mathcal{M}h)(t,\cdot)$ for the residual and
\begin{equation} \label{eq:functionals_main}
    \mathcal{D}(h) := \int_I \| h(t,\cdot) - h^\ast(t,\cdot) \|^2_{\rho_t} \nu(dt), \qquad \mathcal{P}^w(h) := \| \mathcal{R}[h] \|^2_{\mathcal{V}^\ast}
\end{equation}
for the value functional and the weak penalty. Here $\mathcal{D}$ is exactly the population excess least-squares risk, since at fixed $t$ the minimizer of $\E[(h(t,X_{T-t}) - w(X_0))^2]$ is the conditional mean $h^\ast(t,\cdot)$.

\begin{proposition}[Harmonicity and its weak form] \label{prop:weak_form_main}
Under Assumptions \ref{ass:bounded}--\ref{ass:compact}, $h^\ast$ is space-time harmonic for the reverse generator, $\mathcal{M}h^\ast = 0$ on $(0,T)\times\R^d$ with $h^\ast(T,\cdot) = w$. Moreover, for $t \in I$, $h \in C^{1,2}$ and $\psi \in H^1(\rho_t)$ of polynomial growth, writing $\varphi := h - h^\ast$,
\begin{equation} \label{eq:weak_main}
    \langle \mathcal{R}_t[h], \psi \rangle_{\rho_t} \;=\; a_t(h,\psi) \;:=\; \int_{\R^d} \Big[ \partial_t h \, \psi + \big((x+s) \cdot \nabla h\big)\psi - \nabla h \cdot \nabla \psi \Big] \rho_t \, dx ,
\end{equation}
and $|a_t(h,\psi)| \leq \big( \| \partial_t \varphi \|_{\rho_t} + \| \nabla \varphi \|_{\rho_t} + M_4 \| |\nabla \varphi| \|_{L^4(\rho_t)} \big) \| \psi \|_{H^1(\rho_t)}$ with $M_4 \leq C\sqrt d \sigma_\ast^{-2}$.
\end{proposition}

There are three main consequences. The form $a_t$ is computable: it needs first derivatives of $h$, first derivatives of $\psi$, and the pretrained score, because the weight generated by integrating $\Delta h$ by parts against $\rho_t$ is $\nabla \log \rho_t = s$. It is gap-free: $a_t(h^\ast,\psi) = 0$ for all $\psi$, so $h^\ast$ minimizes $\mathcal{D} + \mu\mathcal{P}^w$ for every $\mu \geq 0$ and no bias is traded against variance in the penalty weight. And --- the point that drives the rate --- the continuity bound is first order in $\varphi$, so the penalty inherits a first-order approximation floor: for $ n \geq n_0 $ and suitable budgets there is an $h_N \in \Hcal_N$ with 
\begin{align} \label{eq:first_order_approx_floor}
    & \| h_N - h^\ast \|_{L^4} + \underset{t \in I}{\sup} \| h_N - h^\ast \|_{\rho_t} + \| | \nabla (h_N - h^\ast)| \|_{L^4} + \| \partial_t (h_N - h^\ast) \|_{L^2} \leq e_N, \\
    & e_N := 8 Q_n \varepsilon_N, \quad \varepsilon_N := \mathrm{A}_\beta N^{- (\beta - 1)}, 
\end{align}
where $ \mathrm{A}_\beta $ is the constant, depending only on $ \beta, \bar{B}, \sigma_\ast $ and, through $ R, $ polylogarithmically on $ n, $ and $ Q_n \leq \sigma^{-c} (\log n)^4 $ is given in \eqref{eq:Q_n_def}. Such an $ h_N $ exists in a class with $ \mathbb{V}_N \lesssim N^d\,\mathrm{polylog}(n), $ and gives $\inf_{h \in \Hcal_N}\mathcal{P}^w(h) \lesssim C d \sigma^{-4}_\ast e^2_N $ (Lemmas \ref{lem:weak_res_approx} and \ref{lem:space_time_approx}). Table \ref{tab:penalties} places this against the alternatives.

\begin{table}[h]
\centering
\small
\caption{Why the residual is measured weakly. The floor is $\inf_{h \in \Hcal_N} \mathcal{P}(h)$ at resolution $N,$ only the last is both gap-free and informative at $\beta = 2$.}
\label{tab:penalties}
\begin{tabular}{lccc}
\hline
Penalty $\mathcal{P}(h)$ & order in $h$ & value at $h^\ast$ & approximation floor \\
\hline
Sobolev, $\int_I \|\nabla h\|^2_{\rho_t}\nu(dt)$ & 1 & $\|\nabla h^\ast\|^2 \neq 0$ & --- (nonzero gap for $ \lambda > 0 $) \\
strong residual, $\|\mathcal{R}[h]\|^2_{L^2(\nu\otimes\rho)}$ & 2 & $0$ & $N^{-2(\beta-2)}$, non-informative at $\beta = 2$ \\
weak residual, $\|\mathcal{R}[h]\|^2_{\mathcal{V}^\ast}$ & 1 & $0$ & $N^{-2(\beta-1)}$ \\
\hline
\end{tabular}
\end{table}

\section{From residual to gradient} \label{sec:coercivity}

Two things remain: that the weak penalty still controls the gradient, and that the part of it an estimator can see suffices. Both rest on an exact identity. Put $\omega(t) := \nu([t,t_1]) = (t_1-t)/\ell$, $A(h,\psi) := \int_I a_t(h,\psi(t,\cdot))\,\nu(dt)$, and $E_\omega := \int_I \omega \| \nabla \varphi \|^2_{\rho_t}\nu(dt)$. It\^o's formula applied to $\varphi^2$ along the reverse process, averaged over a random endpoint $r \sim \nu$, gives (Theorem \ref{thm:weighted_energy_identity})
\begin{equation} \label{eq:wei_main}
    2 \ell E_\omega \;+\; \| \varphi(t_0,\cdot) \|^2_{\rho_{t_0}} \;=\; \mathcal{D}(h) \;-\; 2 \ell \, A(h, \omega \varphi) .
\end{equation}
The endpoint averaging is what produces $\omega$, and is forced: the identity on a fixed $[t_0,r]$ leaves a terminal value at a single instant, which the statistical analysis cannot control. Its cost is that $\omega$ vanishes at $t_1$, so guarantees are stated on
\begin{equation} \label{eq:I_delta}
    I_\delta := [t_0,\, t_1 - \delta \ell], \qquad \delta \in (0,1), \qquad \omega \geq \delta \ \text{ on } I_\delta ,
\end{equation}
covering a $(1-\delta)$ fraction of the window at a cost of $\delta^{-1}$ in the constant.

Bounding the cross term in \eqref{eq:wei_main} by $\|\mathcal{R}[h]\|_{\mathcal{V}^\ast}\|\omega\varphi\|_{\mathcal{V}}$ and absorbing already yields coercivity (Corollary \ref{cor:coercivity}). But no estimator sees $\|\mathcal{R}[h]\|_{\mathcal{V}^\ast}$. What a finite test class $\Psi$ delivers is the projected residual
\begin{equation} \label{eq:proj_main}
    \mathcal{N}_{\Psi,\tau}(h) := \sup_{0 \neq \psi \in \Psi} \frac{A(h,\psi)}{\| \psi \|_{\mathcal{V},\tau}} \;\leq\; \| \mathcal{R}[h] \|_{\mathcal{V}^\ast}, \qquad \| \psi \|^2_{\mathcal{V},\tau} := \| \psi \|^2_{\mathcal{V}} + \tau^2 \| \psi \|_\bullet^2 ,
\end{equation}
with $\|\cdot\|_\bullet$ a positively homogeneous scale on the cone $\Psi$ and $\tau$ a slack. Recovering the left-hand side of \eqref{eq:proj_main} from the right is a discrete inf-sup condition: the requirement that some $\kappa > 0$ satisfy
\begin{equation} \label{eq:infsup}
    \| \mathcal{R}[h] \|_{\mathcal{V}^\ast} \;\leq\; \kappa^{-1}\, \mathcal{N}_{\Psi_N}(h) \qquad \text{for all } h \in \Hcal_N ,
\end{equation}
that is, that the finite test class detect a fixed fraction of any residual the trial class can produce. It is the standard price of a weak formulation, and $\kappa$ would enter the final rate multiplicatively. We do not assume \eqref{eq:infsup}. Instead we impose two structural properties on the classes, neither of which requires linearity or convexity:

\begin{enumerate}[label=(S\arabic*), leftmargin=2.8em, itemsep=1pt, topsep=3pt]
    \item \textbf{Symmetric cone:} $\psi \in \Psi \Rightarrow -\psi \in \Psi$ and $r\psi \in \Psi$ for $r \geq 0$. This makes the inner maximization closed-form up to a scalar problem (Lemma \ref{lem:fenchel_cone}) and costs nothing: leave the output layer unnormalized. \label{property:cone}
    \item \textbf{Absorption:} $\omega \cdot (\Hcal - \Hcal) \subseteq \Psi$. This is what replaces \eqref{eq:infsup}. It costs a test network of twice the width and one more layer than the trial network, since a difference of two networks is a network (Lemma \ref{lem:networks_difference}). \label{property:absorption}
\end{enumerate}

Finally, write $ \Lambda_\ast $ for the uniform $ L^4 $ envelope of the coefficients generated by the integration by parts below,
\begin{equation} \label{eq:lambda_big}
    \Lambda_\ast := \underset{t \in I}{\sup} \left( \| | x + s | \|_{L^4 (\rho_t)} + \| d + \nabla \cdot s \|_{L^4 (\rho_t)} + \| | x + s | | s | \|_{L^4 (\rho_t)} + \| \partial_t \log \rho_t \|_{L^4 (\rho_t)} \right),
\end{equation}
which under Assumptions \ref{ass:bounded}-\ref{ass:compact} satisfies $ \Lambda_\ast \lesssim d^2 \sigma_\ast^{-4}. $

\begin{theorem}[Coercivity without an inf-sup constant] \label{thm:coercivity_main}
    Let Assumptions \ref{ass:bounded}--\ref{ass:compact} hold, $ \mathcal{H}_N, \Psi_N $ satisfy \ref{property:cone}--\ref{property:absorption}, every $ h \in \mathcal{H}_N $ be $ C^{1,2} $ on $ I \times \mathbb{R}^d $ with $ | h | \leq 3 \bar{B} $ and bounded $ \partial_t h, \nabla h, \nabla^2 h, $ and $ \| \omega (h - h') \|_\bullet \leq C_0 $ for all $ h, h' \in \mathcal{H}_N. $ If $ h_N \in \mathcal{H}_N $ satisfies the inequality in \eqref{eq:first_order_approx_floor} with some $ e_N > 0, $ then for every $ h \in \mathcal{H}_N,  \, \tau \geq 0, \, \delta \in (0,1) $ with $ \varphi := h - h^\ast, $
    \begin{equation} \label{eq:coercivity_main}
        \int_{I_\delta}  \| \nabla \varphi \|^2_{\rho_t} \nu (dt) \leq \frac{C}{\delta} \left( \frac{1 + \ell}{\ell} \mathcal{D}(h) + \mathcal{N}_{\Psi, \tau} (h)^2 + C_0^2 \tau^2 + (1 + \Lambda _\ast^2 + \ell^{-2}) e_N^2 \right),
    \end{equation}
    where $ C $ absolute and $ \Lambda_\ast $ as in \eqref{eq:lambda_big}. For the classes of Definition \ref{def:trial_test_classes_main_text} the regularity conditions hold and $ C_0 = 1. $ 
\end{theorem}

\section{Estimator and main result} \label{sec:main}

Let $ \varrho(u) = (u \lor 0)^3 $ be the rectified cubic activation function, which we will use for our proof sketch of main result below.

\begin{definition}(Network class) \label{def:network_class}
    Put $ Q_R := \{ x: | x |_\infty \leq R \} $ and $ Z := I \times Q_R. $ For $ W, L, P_0 \in \mathbb{N} $ and $ \Lambda, \mathcal{A} \geq 1, $ let $ \mathcal{NN}_\varrho (W, L, \Lambda, P_0, \mathcal{A}) $ be the set of $ f_\vartheta (z) = A_L x^{(L)} (z) + b_L $ with $ x^{(0)}(z) := z $ and $ x^{(\ell + 1)}(z) := \varrho (A_\ell x^{(\ell)} (z) + b_\ell), $ subject to: hidden widths at most $ W, $ parameters $ \vartheta = (A_0, b_0, \ldots, A_L, b_L) $ bounded in modulus by $ \Lambda, $ at most $ P_0 $ of them nonzero and $ \sup_{z \in Z} \| x^{(\ell)}(z) \|_\infty \leq \mathcal{A} $ for every hidden layer. We will write $ \mathcal{NN}_{\varrho, \Theta} $ for the same class with $ \Theta $ as the activation of the last hidden layer.
\end{definition}

\begin{definition}(Trial and test classes) \label{def:trial_test_classes_main_text}
    Fix a resolution $ N, $ a time degree $ K, $ a radius $ R $ and budgets $ \Lambda, P_0, \mathcal{A}, $ and set $ W_N \asymp K N^d $ with $ W_N \geq d + 1. $ Let $ \Theta \in C^\infty (\mathbb{R}, [- 2 \bar{B}, 3 \bar{B}]) $ be a fixed 1-Lipschitz retraction with $ \Theta = \mathrm{id} $ on $ [- \bar{B}, 2 \bar{B}] $ and bounded second derivative, and let $ \chi_R $ be a fixed smooth coordinatewise clip, the identity on $ [- (R-1), R-1], $ of modulus at most $ R, $ derivative at most 1 and bounded second derivative. Write $ f^\chi (t, x) := f (t, \chi_R (x)), $ then
    \begin{align} 
        & \mathcal{H}_N := \{ (\Theta \circ g)^\chi : g \in \mathcal{NN}_\varrho (W_N, L, \Lambda, P_0, \mathcal{A}), \, \sup_{I \times \mathbb{R}^d} | \nabla h | \leq G, \, \sup_{I \times \mathbb{R}^d} | \partial_t h | \leq G_t \}, \label{eq:trial_class} \\ 
        & \Psi_N := \{ c \omega \phi^\chi  : c \in \mathbb{R}, \phi \in \mathcal{NN}_{\varrho, \Theta} (2 W_N, L + 1, \Lambda, 2 P_0 + 2, \mathcal{A} \lor 3 \bar{B}), \,  \\
        & \qquad \qquad \qquad \qquad \qquad \qquad \qquad \qquad \qquad \quad \sup_{I \times \mathbb{R}^d} (| \phi^\chi | + | \nabla \phi^\chi | + | \partial_t \phi^\chi |) \leq \Gamma) \}, \label{eq:test_class}
    \end{align}
    the constraint in \eqref{eq:trial_class} being on $ h = (\Theta \circ g)^\chi $ itself, with $ G := 4 \sqrt{d} \bar{B} \sigma^{-2}_\ast, \, G_t := C_t d \bar{B} (R + 1) \sigma^{-4}_\ast, \, \Gamma := 6 \bar{B} + 2 G + 2 G_t + 6 \bar{B} / \ell $ and $ C_t $ absolute.
\end{definition}

With $ \{ (X_0^i, \varepsilon^i) \}_{i=1}^n $ i.i.d., $ X^i_{T-t} := \mu_{T-t} X_0^i + \sigma_{T-t} \varepsilon^i \sim \rho_t, $ and $ \hat{a}_t, \hat{\mathcal{D}}, \| \cdot \|_{\hat{\mathcal{V}}} $ the empirical counterparts,
\begin{equation} \label{eq:h_empirical_main}
    \hat{h} \in \underset{h \in \mathcal{H}_N}{\mathrm{argmin}} \, \hat{\mathcal{D}}(h) + \mu \underset{\psi \in \Psi_N}{\sup} \left[ \int_I \hat{a}_t (h, \psi) \nu (dt) - \frac{\alpha}{2} \| \psi \|^2_{\hat{\mathcal{V}}, \delta_n} \right], ~ \hat{g} := \Pi_{G^\ast} \left( \frac{\nabla \hat{h}}{\hat{h} \lor \mathrm{b}} \right).
\end{equation}

For $ \psi \in \Psi_N $ let $ \| \psi \|_\bullet := \inf \{ |c| : \psi = c \omega \phi^\chi \text{ as in } \eqref{eq:test_class} \} $ be its cone scale, the scale entering \eqref{eq:h_empirical_main}. Writing $ \mathbb{V}_N $ for the metric entropy index, in the supremum norm and normalized by $ \log (e / \varepsilon), $ of the first order objects $ (h, \partial_th, \nabla h), \, h \in \mathcal{H}_N, $ and $ (\psi, \nabla \psi), \, \psi \in \Psi_N, \, \| \psi \|_\bullet \leq 1 $ (see \eqref{eq:eff_dim}), it is admissible that
\begin{equation} \label{eq:mathbb_V}
    \mathbb{V}_N \lesssim N^d \, \mathrm{polylog}(n)
\end{equation}
once $ R, K $ and the budgets are as specified below. The test class is normalized by the cone scale.

The spatial radius, time degree and the approximation constant are
\begin{align} 
    & R := 2 + \sqrt{(8 \beta + 16) \log (2 d n)}, \quad \eta := \min \{ \frac{\ell}{2}, \frac{\pi \sigma_\ast^4}{12 d (5 R + 3)} \}, \quad K := \left\lceil \frac{\ell}{\eta} (\beta + 2) \log n \right\rceil, \label{eq:param_choice1} \\
    & \mathrm{A}_\beta := C_\beta \bar{B} \sigma_\ast^{-2 \beta} (2 R)^{\beta - 1}, \quad C_\beta := \mathfrak{B}_{\beta + 1} \beta! \, (2 \sec (\pi / 6))^{\beta + 1}, \label{eq:param_choice2}
\end{align}
with $ \mathfrak{B}_{\beta + 1} $ the Bell number. Here $ \eta $ is the width of a complex neighborhood of $ I $ on which $ h^* $ is analytic in $ t, \, C_\beta \bar{B} \sigma_\ast^{-2 \beta} $ bounds its $x$-derivatives of order at most $ \beta $ there, and $ (2 R)^{\beta - 1} $ is the cost of rescaling $ Q_R $ to the unit cube. Thus $ R \asymp \sqrt{\log n}, \, K \leq \sigma^{-c} (\log n)^{3/2} $ and $ \mathrm{A}_\beta = \mathrm{polylog}(n) $ at fixed $ \beta. $

\begin{theorem}[Guidance rate] \label{thm:rate_main}
    Grant Assumptions \ref{ass:bounded}--\ref{ass:compact} and fix $ \zeta, \delta \in (0, 1) $ and $ \kappa \geq 1. $ There are constants $ C^\#, c^\# \geq 1 $ and $ L_0 \in \mathbb{N} $ depending only on $ (d, \beta), $ a constant $ C_0 = C_\star \sigma^{-c}_\ast $ and $ n_0 $ not depending on $ n, $ such that the following holds for every $ n \geq n_0. $ Let $ R, K $ and $ \mathrm{A}_\beta $ be given by \eqref{eq:param_choice1}--\eqref{eq:param_choice2}, and choose
    \begin{equation} \label{eq:N_mu_alpha_b_main}
        \begin{aligned}
            & N := \left\lceil (n \mathrm{A}_\beta^2)^{\frac{1}{2 (\beta - 1) + d}} \right\rceil, ~ \varepsilon_n := \mathrm{A}_\beta^{\frac{2d}{2 (\beta - 1) + d}} n^{- \frac{2 (\beta - 1)}{2 (\beta - 1) + d}}, ~ \mu, \alpha \in [\kappa^{-1}, \kappa], \\
            & \mathrm{b} := (\varepsilon_n / \Xi_\mathrm{s})^{\frac{1}{\mathrm{s} + 4}} \land \bar{B}, ~ \delta_n^2 := C_0 \frac{( \mathbb{V}_N + \log (1 / \zeta)) \log^2 (n / \zeta)}{n}.
        \end{aligned}
    \end{equation}
    Let $ \mathcal{H}_N, \Psi_N $ be the classes of Definition \ref{def:trial_test_classes_main_text} with these $ N, R, K $ and the budgets 
    \begin{equation} \label{eq:budgets_main}
        \begin{aligned}
            & W_N := \lceil C^\# K (N + 1)^d \rceil, ~ L := L_0 + \lceil \log_2 (K + 1) \rceil + 1, \\
            & P_0 := \lceil C^\# K ((N + 1)^d + L) \rceil, ~ \Lambda := \mathcal{A} := n^{c^\#},
        \end{aligned}
    \end{equation}
    let $ \mathbb{V}_N $ be given by \eqref{eq:mathbb_V}, and let $ \hat{h}, \hat{g} $ be the estimator \eqref{eq:h_empirical_main}. Then, with probability at least $ 1 - \zeta $
    \begin{equation} \label{eq:g_rate_main}
        \| \hat{g} - g^\ast \|^2_{L^2 (\nu \otimes \rho; I_\delta)} \leq \delta^{-1} \sigma^{-c} \log^8 ( e n / \zeta) \Xi_{\mathrm{s}}^{\frac{4}{\mathrm{s} + 4}} \varepsilon_n^{\frac{\mathrm{s}}{\mathrm{s} + 4}}.
    \end{equation}
    In the boundary case $ \mathrm{s} = \infty, $ where $ w \geq B > 0 $ on $ \mathrm{supp}(p_0), $ the right hand side is $ \delta^{-1} B^{-4} \sigma^{-c} \log^8 (e n / \zeta) \varepsilon_n, $ of order $ n^{-2 / (d + 2)} $ at $ \beta = 2 $ up to logarithmic factors. 
\end{theorem}

\paragraph{Proof sketch.} Write $\Pi_\Psi(h) := \mathcal{N}_{\Psi_N,\delta_n}(h)^2$ and let $h_N \in \Hcal_N$ be the approximant of Section~\ref{sec:weak_residual}, satisfying \eqref{eq:first_order_approx_floor}. Since $Q_n \leq \sigma^{-c} (\log n)^4$ and $\varepsilon_N^2 \leq \varepsilon_n$ for the $N$ of \eqref{eq:N_mu_alpha_b_main}, $e_N^2 \leq \sigma^{-c} (\log n)^8 \varepsilon_n$. The floor gives $\mathcal{D}(h_N) \leq e_N^2$ directly, and \eqref{eq:proj_main} together with the penalty bound of Section~\ref{sec:weak_residual} gives $\Pi_\Psi(h_N) \leq \mathcal{P}^w(h_N) \leq C d \sigma_\ast^{-4} e_N^2$. The argument runs on the event of the deviation bound (Theorem~\ref{thm:dev_bound}), which holds with probability $1-\zeta$ and is proved, not assumed.
 
Step 1 (the inner maximum is the projected residual). Since $\Psi_N$ is a symmetric cone \ref{property:cone}, the supremum in \eqref{eq:h_empirical_main} factors into a radial and a directional part, the radial part being a scalar quadratic (Lemma \ref{lem:fenchel_cone}). The deviation bound transfers between $\hat a$ and $a$ and between the empirical and population $\mathcal{V}$-norms multiplicatively, pinning the inner supremum $\hat F(h)$ between $\Pi_\Psi(h)/(8\alpha) - \delta_n^2/(2\alpha)$ and $ (4\Pi_\Psi(h) + 8\delta_n^2)/\alpha, $ no additive term appears outside the ridge, which is what the ridge secures.
 
Step 2 (basic inequality). $\hat h$ minimizes and $h_N$ is feasible, so $\hat{\mathcal{D}}(\hat h) + \mu\hat F(\hat h) \leq \hat{\mathcal{D}}(h_N) + \mu\hat F(h_N)$. Subtracting $\hat{\mathcal{D}}(h^\ast)$ and applying Step 1 gives
\begin{equation} \label{eq:D_bound_main}
    \mathcal{D}(\hat h) + \Pi_\Psi(\hat h) \;\lesssim\; \mathcal{D}(h_N) + \Pi_\Psi(h_N) + \delta_n^2 \;\lesssim\; d\sigma_\ast^{-4} e_N^2 + \delta_n^2 .
\end{equation}
Calibration enters only here, and only through $\mu/\alpha \in [\kappa^{-2}, \kappa^2].$ No choice of $(\mu,\alpha)$ trades bias against variance, because gap-freeness left no bias to trade.
 
Step 3 (coercivity). Every $h \in \Hcal_N$ is $C^{1,2}$ with bounded derivatives up to second order and $\|\omega(h - h')\|_\bullet \leq 1$ (Proposition \ref{prop:classes_props}), so Theorem \ref{thm:coercivity_main} applies to $\hat h$ at $\tau = \delta_n$ with $C_0 = 1$. With \eqref{eq:D_bound_main} this gives $\int_{I_\delta}\|\nabla\hat\varphi\|^2_{\rho_t}\nu(dt) + \mathcal{D}(\hat h) \lesssim \delta^{-1}\sigma^{-c}(e_N^2 + \delta_n^2)$.
 
Step 4 (balance). With the budgets \eqref{eq:budgets_main}, $\mathbb{V}_N \leq \sigma^{-c} N^d (\log n)^{7/2}$, so $\delta_n^2 \leq \sigma^{-c} (N^d/n) \log^{13/2}(en/\zeta)$. Balancing $\varepsilon_N^2 = \mathrm{A}_\beta^2 N^{-2(\beta-1)}$ against $N^d/n$ gives $N^{2(\beta-1)+d} \asymp n\mathrm{A}_\beta^2$, the choice in \eqref{eq:N_mu_alpha_b_main}, with common value of order $ \varepsilon_n, $ hence $e_N^2 + \delta_n^2 \leq \sigma^{-c} \log^8 (en/\zeta) \, \varepsilon_n$. Here the weak form is decisive: with the strong penalty the floor would be $N^{-2(\beta-2)}$ and the balance empty at $\beta = 2$.
 
Step 5 (clipping). Corollary \ref{cor:guidance_control} converts gradient and value bounds into $\delta^{-1}\sigma^{-c}\big(\log^8(en/\zeta)\,\mathrm{b}^{-4}\varepsilon_n + \Xi_{\mathrm{s}}\mathrm{b}^{\mathrm{s}}\big)$, the two terms being equal to $\Xi_{\mathrm{s}}^{4/(\mathrm{s}+4)}\varepsilon_n^{\mathrm{s}/(\mathrm{s}+4)}$, up to the logarithmic factor, at $\mathrm{b} = (\varepsilon_n / \Xi_{\mathrm{s}})^{1/(\mathrm{s}+4)}$. If this value exceeds $\bar B$, then $\Xi_{\mathrm{s}} \mathrm{b}^{\mathrm{s}} \geq \Xi_{\mathrm{s}} \bar B^{\mathrm{s}} \geq 1$ because $w \leq \bar B$, and the trivial bound $|\hat g - g^\ast| \leq 2 G^\ast$ already gives \eqref{eq:g_rate_main}. At $\mathrm{s} = \infty$, $\mathrm{b} = B \leq h^\ast$ and the second term vanishes. \hfill$\square$

\section{From guidance error to sampling error} \label{sec:sampling}

\begin{corollary}[Local stability of the controlled marginals] \label{cor:sampling_main}
    Let $J := [t_0,\bar t] \subseteq I_\delta$ and, conditionally on the sample, let $ \mathbb{Q}^\ast $ and $ \hat{\mathbb{Q}} $ be the laws on $C(J;\R^d)$ of the solutions of 
    \begin{equation}
        d \overleftarrow{Z}_t = (b + 2 g^\ast) (t, \overleftarrow{Z}_t) dt + \sqrt{2} d B_t, \quad d \overleftarrow{Y}_t = ( b + 2 \hat{g}) (t, \overleftarrow{Y}_t) dt + \sqrt{2} d B_t,
    \end{equation}
    with $ b $ as in \eqref{eq:rev_denoising}, both started at time $ t_0 $ from $ \rho^\ast_{t_0} := h^\ast (t_0, \cdot) \rho_{t_0} / Z, $ and let $ \rho_t^\ast $ and $ \hat{\rho}_t $ be their marginals at time $ t \in J. $ Then the two laws are equivalent, $\mathrm{KL}(\mathbb{Q}^\ast \| \hat{\mathbb{Q}}) = \int_J \|\hat g - g^\ast\|^2_{\rho^\ast_t} dt \leq \ell \Upsilon \|\hat g - g^\ast\|^2_{L^2(\nu\otimes\rho;J)}$ with $\Upsilon := \sup_{t \in J} \|d\rho^\ast_t/d\rho_t\|_\infty \leq \bar{B} \,  \Xi_{\mathrm{s}}^{1/\mathrm{s}}$, and consequently, on the event of Theorem \ref{thm:rate_main} and up to $\sigma^{-c}$ and $\mathrm{polylog}(n)$ factors,
    \begin{equation} \label{eq:tv_main}
        \mathrm{TV}(\rho^\ast_{\bar t}, \hat\rho_{\bar t})^2 \;\lesssim\; \delta^{-1}\sigma^{-c} \log^8(en / \zeta) \ell \bar B \, \Xi_{\mathrm{s}}^{\frac{1}{\mathrm{s}} + \frac{4}{\mathrm{s}+4}} \, \varepsilon_n^{\frac{\mathrm{s}}{\mathrm{s}+4}} ,
    \end{equation}
    which in the boundary case $\mathrm{s} = \infty$ reads $ \delta^{-1}\ell \bar B B^{-5}\sigma^{-c} \log^8 (e n /\zeta) \varepsilon_n. $
\end{corollary}

The projection $\Pi_{G^\ast}$ is what keeps the Novikov constant independent of $n$, and the reweighting bound $\Upsilon \leq \bar B/Z$ is what lets a guarantee stated under the uncontrolled marginals $\rho_t$ --- the ones we can sample --- be read under the controlled ones.

\section{Numerical Experiments}
\label{sec:experiments}

We test the rate of Theorem~\ref{thm:rate_main} where every quantity it bounds is closed-form, and apply the estimator \eqref{eq:h_empirical_main} to pretrained image diffusion models, where none is (details in Appendix~\ref{app:experiments}). The image experiments ask one question: does the estimator the theorem analyses work on a real pretrained model, and at what cost to the base distribution? We measure that cost and compare with the two estimators closest to it in kind.

\textbf{The rate.} $p_0$ is atomic on $512$ points in a ball of radius $0.8$ and $0.5 \leq w \leq 3$, so Assumptions~\ref{ass:bounded}--\ref{ass:compact} hold with $\mathrm{s} = \infty$, and the score is closed-form. In tensor cubic B-splines ($\beta = 4$, $\mathbb{V}_N \asymp N^d$), with $(N, \mu)$ chosen by oracle, the measured exponent over three decades of $n$ is $0.842 \pm 0.026$ (Figure~\ref{fig:main}a), against $0.857$ predicted and $0.667$ for recovering a first derivative from noisy values \citep{stone1982}. At $n = 10^6$, the strong residual, unpenalized least squares and the Sobolev penalty give $1.4$, $5.9$ and $8.5$ times the guidance error of the weak residual. At $\beta = 3$, the strong residual's approximation floor decays as $N^{-1.0}$, against $N^{-1.8}$ for the weak one (Appendix~\ref{app:synthetic}).

\begin{figure}[t]
\centering
\begin{minipage}[t]{0.335\textwidth}\vspace{0pt}\includegraphics[width=\textwidth]{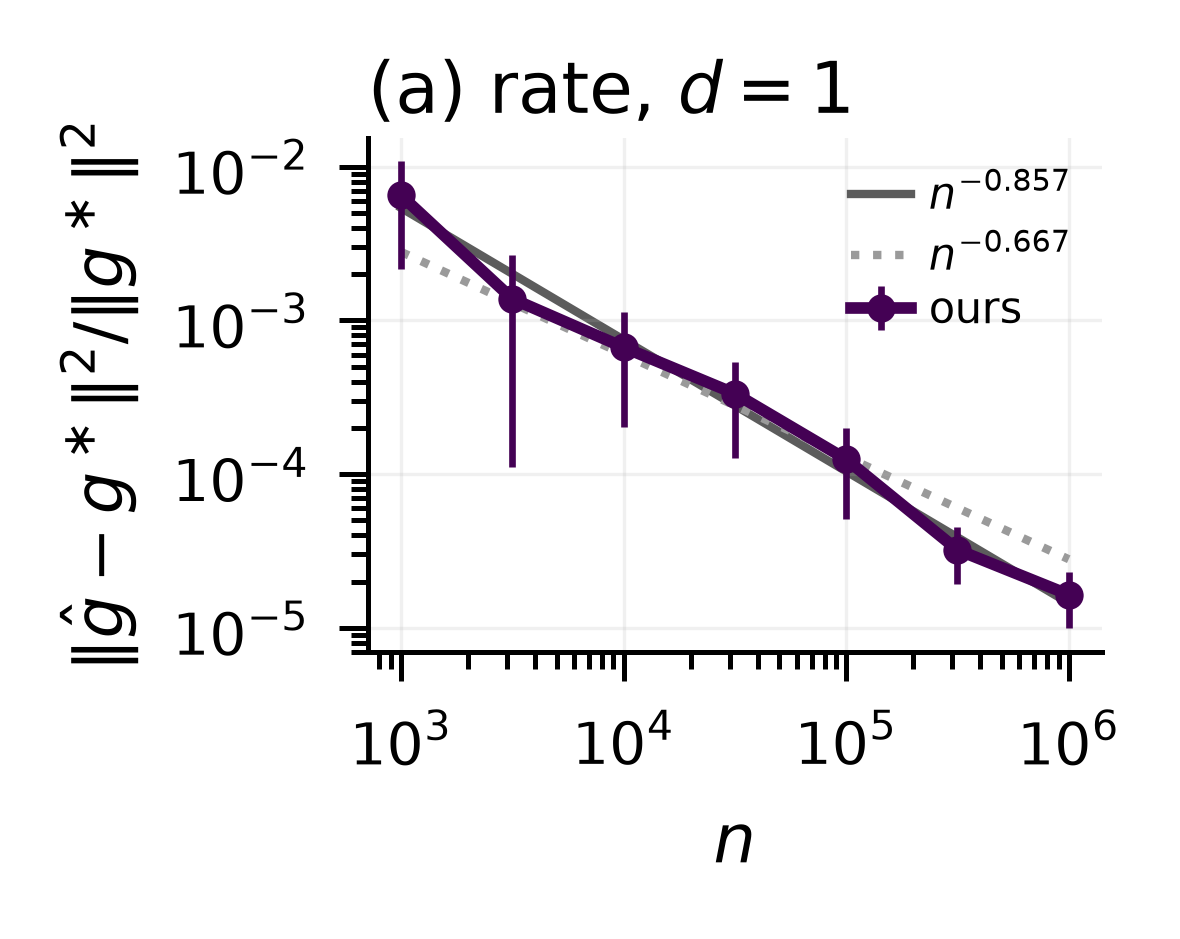}\end{minipage}\hfill
\begin{minipage}[t]{0.335\textwidth}\vspace{0pt}\includegraphics[width=\textwidth]{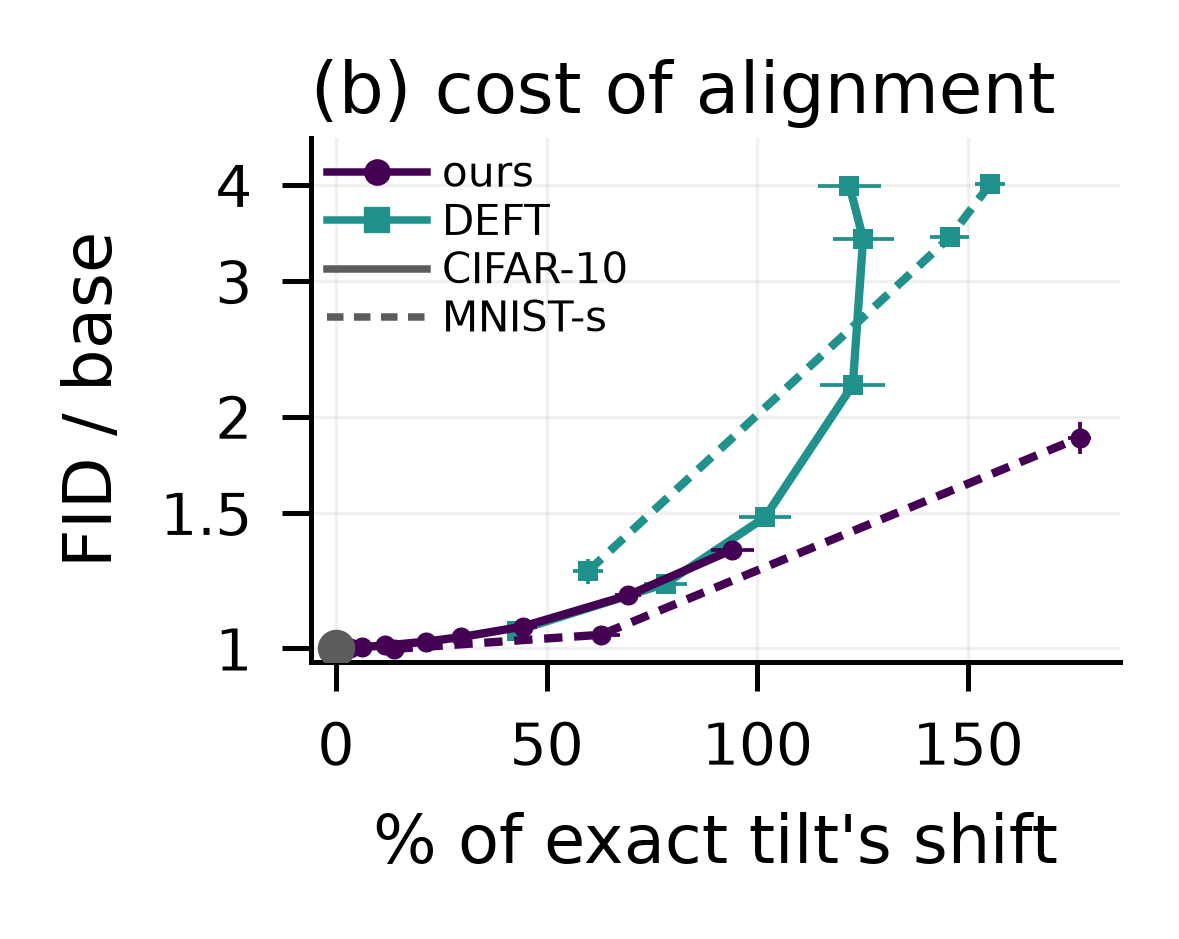}\end{minipage}\hfill
\begin{minipage}[t]{0.29\textwidth}\vspace{0pt}\centering{\small (c)}\\[1pt]
\includegraphics[width=\textwidth]{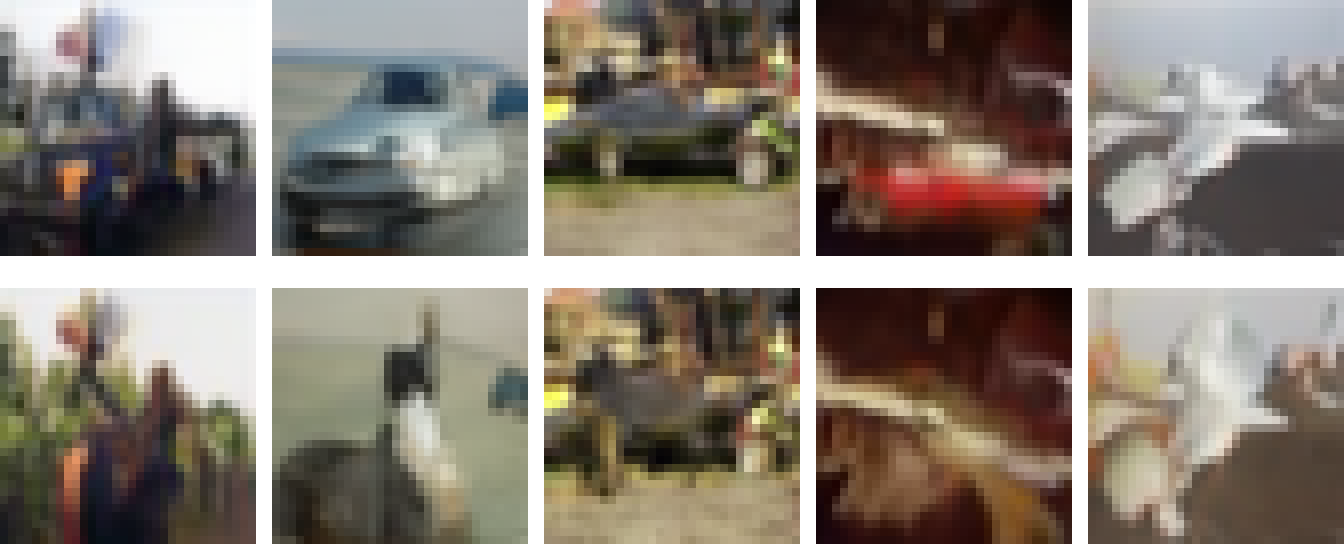}\end{minipage}
\caption{(a) Guidance error at $d=1$, $\pm1$ s.d.\ over $8$ seeds, against slopes $0.857$ and $0.667$. (b) Held-out reward shift, as a share of the exact tilt's, against the FID it costs over that dataset's base ($24.5$, $7.4$), so grey is both base models and lower is cheaper. Colour is the method, line style the dataset, $c \in [0.25, 20]$, bars $\pm1$ s.d.\ over $3$ seeds, DOIT is off the panel at $4.8$ and $21$ times its base, absolute values in Figure~\ref{fig:frontier_all}. (c) The five largest rises in held-out probability among $64$ samples, base above and ours at $c = 20$ below on the same initial noise, all $64$ in Figure~\ref{fig:cifar_samples}.}
\label{fig:main}
\end{figure}

\textbf{Alignment of a pretrained model.} On a model whose score is only a network, we tilt the CIFAR-10 DDPM of \citet{ho2020denoising} (100 DDIM steps, \citealp{song2021ddim}) towards animals, with a bounded, positive weight $w$ built from a classifier (Appendix~\ref{app:real_setup}). The reference sample is drawn from the model, so the score entering the penalty is the pretrained network itself, and the trial and test classes are convolutional instances of Definition~\ref{def:trial_test_classes_main_text}, with absorption holding by construction (Appendix~\ref{app:objective}). We sample along $c\,\nabla \log \hat h$: $c = 1$ is the case covered by Theorem~\ref{thm:rate_main}, while $c > 1$, common in practice \citep{dhariwal2021}, targets a sharper law and is reported separately as a heuristic. Samples are scored by a classifier head that $w$ never saw.

\textbf{Exact guidance ($c = 1$).} The theorem covers a non-degenerate regime: at $c = 1$ the estimated guidance applied only on the estimation window already achieves $11.7\%$ of the exact tilt's reward shift, at an FID to the base distribution of $24.7$ against $24.5.$ This measures practical steering rather than the guarantee itself.

\textbf{Scaled guidance ($c > 1$), a heuristic.} Scaling the same direction traces the frontier of Figure~\ref{fig:main}b: the share of animals rises from $0.653$ to $0.831$ at $c = 20$, against $0.854$ for $32$-fold importance resampling at $32\times$ the cost. Under the same weight, sample, sweep and cap we compare DEFT \citep{denker2024} and DOIT \citep{zhu2026doit}, from their authors' public code (Appendix~\ref{app:cifar}). Of the three only DEFT is given the gradient of the log weight, ours never differentiating the reward model. Read at a fixed departure the frontiers coincide: $62.5\%$ of the exact tilt's shift against $62.7\%$ at FID $28$, $77.0\%$ against $79.2\%$ at FID $30,$ DOIT reaches $70.6\%$ only at FID $116.6$. A sample that changes class keeps its counterpart's layout and palette (Figure~\ref{fig:main}c). The gain transfers to a head that $w$ never saw: at $94\%$ shift the ratio of the gain on the classifier defining $w$ to the gain on the held-out one is $1.19$ for us and $1.20$ for DEFT, while for DOIT it is negative, the two heads moving apart.

\textbf{MNIST.} On a public MNIST DDPM with the CIFAR-10 architecture, tilted towards odd digits (Appendix~\ref{app:mnist}), the ordering is ours: at $c = 20$ the share of odd digits rises from $0.523$ to $0.819$ at FID $13.8$, which is $176\%$ of the exact tilt's shift, where DEFT at the same FID reaches $84\%$ and needs FID $29.5$ for $155\%$ (Figure~\ref{fig:main}b). DOIT reaches $296\%$ at FID $153$, its cap binding on every step (Appendix~\ref{app:cap}).

\section{Discussion and limitations} \label{sec:discussion}

\textbf{Constants, and the choice of $\beta$.} Theorem \ref{thm:rate_main} is a family indexed by $\beta$, and $\beta$ is not free: $\mathrm{A}_\beta$ enters \eqref{eq:N_mu_alpha_b_main} as $\mathrm{A}_\beta^{2d/(2(\beta-1)+d)}$ rather than as a prefactor, and grows factorially, so raising $\beta$ improves the exponent and worsens the constant, and at moderate $n$ and large $d$ the smallest admissible value wins the trade. We therefore read $\beta = 2$, giving $n^{-2/(d+2)}$, as the operative case --- and it is exactly the case in which the strong residual is non-informative (Table \ref{tab:penalties}, see Appendix \ref{app:floors}).

\textbf{The score enters as a known coefficient.} This is what makes the penalty computable from a frozen model and nothing else: the weak form needs exactly the quantity the pretrained sampler already supplies, and no second network is fitted to obtain it. The theory reads that coefficient as exact, and the closed-form instance supplies it exactly. On images it is the base network's own output, and the estimator remains effective there, which is evidence that the construction tolerates the score error a pretrained model carries. Quantifying how the guarantees degrade with that error is a natural next step.

\textbf{Towards intrinsic dimension.} The ambient dimension enters the exponent only through the approximation budget $\mathbb{V}_N \asymp N^d$. Since natural data concentrate near low-dimensional sets \citep{pope2021} and score estimation adapts to the intrinsic dimension \citep{chen2023low,yakovlev2025}, we expect the same theorem with $d$ replaced by an intrinsic $k,$ as the bilinear form is integrated against the $d$-dimensional $\rho_t$, the gain has to come from the approximation side.

\bibliographystyle{abbrvnat}
\bibliography{references}

\appendix


\section{Notation recap} \label{app:notation}

For the reader's convenience we recall the objects fixed in Section \ref{sec:setup}. The forward
process is the variance-preserving Ornstein--Uhlenbeck SDE
\begin{equation} \label{eq:VP_SDE}
    d X_\tau = - X_\tau \, d \tau + \sqrt{2} \, d B_\tau, \quad \tau \in (0, T), \quad X_0 \sim p_0,
\end{equation}
with Gaussian transition kernel
\begin{equation} \label{eq:gauss_likelihood}
    X_\tau \mid X_0 = x_0 \sim \mathcal{N} (\mu_\tau x_0, \sigma_\tau^2 I_d), \quad \mu_\tau = e^{-\tau}, \quad \sigma_\tau^2 = 1 - e^{-2 \tau},
\end{equation}
so that $ \mu_\tau^2 + \sigma_\tau^2 = 1, $ and $ p_\tau := \mathrm{Law}(X_\tau). $ The reverse process
\eqref{eq:rev_denoising} has marginals $ \rho_t = p_{T-t} $ and generator
$ (\mathcal{L}_t f)(x) = b(t,x)\cdot\nabla f(x) + \Delta f(x), $ with $ \mathcal{M} = \partial_t + \mathcal{L}_t. $
The Laplacian, rather than $ \tfrac12 \Delta, $ appears because the diffusion coefficient is $ \sqrt{2}: $
It\^o's rule reads $ df(t, X_t) = (\partial_t f + b\cdot\nabla f + \Delta f)\,dt + \sqrt{2}\,\nabla f \cdot dB_t, $
the second-order term being $ \tfrac{1}{2}\mathrm{tr}(2 I \nabla^2 f) = \Delta f. $

\begin{definition}[Doob's $h$-function] \label{def:h_star}
    For $ t \in [0,T] $ and $ x \in \mathbb{R}^d, $
    \begin{equation}
        h^* (t, x) := \mathbb{E} [w ( \overleftarrow{X}_T ) \mid \overleftarrow{X}_t = x ] = \mathbb{E} [ w(X_0) \mid X_{T-t} = x ],
    \end{equation}
    the two expressions coinciding because $ (\overleftarrow{X}_t, \overleftarrow{X}_T) \overset{d}{=} (X_{T-t}, X_0) $
    under time reversal. The guidance is $ g^* := \nabla \log h^*. $
\end{definition}

By Bayes' rule for \eqref{eq:gauss_likelihood}, the conditional law of $ X_0 $ given $ X_{T-t} = x $ is
\begin{equation} \label{eq:gauss_posterior}
    \pi_{t,x} (dx_0) \propto \phi_{\sigma^2_{T-t}} (x - \mu_{T-t} x_0) \, p_0(dx_0),
\end{equation}
and $ h^*(t, x) = \mathbb{E}_{X_0 \sim \pi_{t,x}} [w (X_0)], $ so $ 0 < w \leq \bar{B} $ implies $ 0 < h^* \leq \bar{B}. $

We write $ \| \varphi \|^2_{L^2 (\nu \otimes \rho)} := \int_I \| \varphi \|^2_{\rho_t} \nu (dt), $ similarly for
$ L^4, $ and $ \| \cdot \|_{\mathcal{C}^\ell} $ for the supremum norm of derivatives up to order $ \ell $ (the
reservation of $ \mathcal{H} $ for hypothesis classes and of $ H^1 $ for Sobolev spaces makes this notation
necessary). Throughout, $ \beta \geq 2 $ is an integer and the approximation analysis of Appendix
\ref{app:approx} measures $ h^* $ in $ W^{\beta, \infty} $ uniformly on the window. This is not an
assumption: by Proposition \ref{prop:derivative_bounds}, $ h^* \in C^{\infty}((0, T) \times \mathbb{R}^d), $ and
Step 1 of Lemma \ref{lem:space_time_approx} produces from Assumptions \ref{ass:bounded}--\ref{ass:compact} alone
the explicit envelope $ \sup | \partial_x^\gamma h^* | \leq C_\beta \bar{B} \sigma_*^{- 2 | \gamma |} $ at every
order $ | \gamma | \leq \beta. $ Thus $ \beta $ is a free parameter of the construction, traded against a
constant growing factorially in $ \beta. $

Finally, we write
\begin{equation} \label{eq:half_window}
    I^\circ := [t_0, \tfrac{t_0 + t_1}{2}]
\end{equation}
for the first half of the window \eqref{eq:window}. It is where the weight $ \omega $ of Theorem
\ref{thm:energy_identity} satisfies $ \omega \geq 1/2, $ which is why the guarantees below are stated
there. Remark \ref{rem:sampling_window} explains how they transfer to a prescribed sampling window.

\section{Regularity of the Doob $h$-function} \label{app:regularity}

\begin{lemma}(Score identity and drift envelope) \label{lem:score_identity}
    Under Assumptions \ref{ass:bounded}--\ref{ass:compact}, for every $ t \in (0, T) $ and $ x \in \mathbb{R}^d, $
    \begin{equation} \label{eq:score_identity_1}
        s(t,x) = - \frac{x - \mu_{T-t} \mathbb{E} [X_0 | X_{T-t} = x]}{\sigma^2_{T-t}}, \quad \text{hence } x + s(t,x) = \left( 1 - \frac{1}{\sigma^2_{T-t}} \right) x + \frac{\mu_{T-t}}{\sigma^2_{T-t}} \mathbb{E}[X_0 | x].
    \end{equation}
    Consequently there is an absolute constant $ C $ such that, for all $ t \in I, $
    \begin{equation}
        \| | x + s (t, \cdot ) | \|_{\psi_2, \rho_t} \leq \frac{C \sqrt{d}}{\sigma^2_{T-t}}, \quad M_4 := \underset{t \in I}{\sup} \| | x + s(t, \cdot) | \|_{L^4 (\rho_t)} \leq \frac{C \sqrt{d}}{\sigma_*^2} < \infty.
    \end{equation}
\end{lemma}
\begin{proof}
    Differentiating $ p_\tau (x) = \int \phi_{\sigma^2_\tau} (x - \mu_\tau x_0) p_0 (dx_0) $ under the integral (legitimate by compact support and Gaussian tails) gives $ \nabla \log p_\tau (x) = \int (- \frac{x - \mu_\tau x_0}{\sigma^2_\tau}) \pi_{\tau, x} (dx_0), $ which is \eqref{eq:score_identity_1} at $ \tau = T - t, $ the second display follows from $ b - s = x + s $ by rearrangement. 

    By Assumption \ref{ass:compact}, $ | X_0 | \leq \sqrt{d} $ almost surely, hence $ | \mathbb{E} [X_0 | x] | \leq \sqrt{d} $ for every $ x. $ Under $ \rho_t $ we may write $ X_{T-t} = \mu_{T-t} X_0 + \sigma_{T-t} \varepsilon $ with $ \varepsilon \sim \mathcal{N}(0, I_d) $ independent of $ X_0, $ so by the triangle inequality for $ \| \cdot \|_{\psi_2} $ and $ \mu_{T-t}, \sigma_{T-t} \leq 1, ~ \| | X_{T-t} | \|_{\psi_2} \leq \mu_{T-t} \sqrt{d} + \sigma_{T-t} \| | \varepsilon | \|_{\psi_2} \leq C \sqrt{d}. $ Applying the triangle inequality to the second identity in \eqref{eq:score_identity_1}
    \begin{equation}
        \| | x + s | \|_{\psi_2, \rho_t} \leq \left| 1 - \frac{1}{\sigma^2_{T-t}} \right| \| | X_{T-t} | \|_{\psi_2} + \frac{\mu_{T-t}}{\sigma^2_{T-t}} \sqrt{d} \leq \frac{C \sqrt{d}}{\sigma^2_{T-t}},
    \end{equation}
    using $ \sigma_{T-t} \leq 1. $ Finally $ \| Z \|_{L^4} \leq C \| Z \|_{\psi_2} $ for any $ Z, $ and $ \sigma_{T-t} \geq \sigma_* $ on $ I. $
\end{proof}

Lemma \ref{lem:score_identity} does more than make $ M_4 $ finite: it supplies a sub-Gaussian envelope for the unbounded coefficient of the weak form, which is what the deviation bound of Section \ref{app:statistical} needs.

The Gaussian likelihood smooths the bounded terminal datum $ w. $ The following makes the smoothness and its $\sigma$-scaling quantitative. The first-derivative identity is Tweedie's formula and it exhibits the guidance as a posterior covariance. 

\begin{proposition} (Smoothness and derivative bounds) \label{prop:derivative_bounds}
    Under Assumptions \ref{ass:bounded}--\ref{ass:compact}, $ h^* \in C^\infty ((0, T) \times \mathbb{R}^d), $ and for all $ t \in (0, T), \, x \in \mathbb{R}^d $ and coordinates $ k, l $
    \begin{equation} \label{eq:tweedie}
        \partial_{x_k} h^* (t, x) = \frac{\mu_{T-t}}{\sigma^2_{T-t}} \Cov (w(X_0), X_{0,k} | X_{T-t} = x),
    \end{equation}
    where the covariance is under the posterior \eqref{eq:gauss_posterior}. Consequently 
    \begin{equation} \label{eq:h_der_bounds} 
        0 < h^* \leq \bar{B}, \quad | \partial_{x_k} h^* | \leq \frac{2 \bar{B}}{\sigma^2_{T-t}}, \quad | \partial^2_{x_k x_l} h^* | \leq \frac{C_2 \bar{B}}{\sigma^4_{T-t}},
    \end{equation}
    for an absolute constant $ C_2. $ Moreover the log-density of the reverse marginal satisfies 
    \begin{equation} \label{eq:hessian}
        - \frac{1}{\sigma^2_{T-t}} I_d \preceq \nabla^2 \log \rho_t (x) = - \frac{1}{\sigma^2_{T-t}} I_d + \frac{\mu^2_{T-t}}{\sigma^4_{T-t}} \Cov (X_0 | x) \preceq \left( \frac{d \mu^2_{T-t}}{\sigma^4_{T-t}} - \frac{1}{\sigma^2_{T-t}} \right) I_d.
    \end{equation}
\end{proposition}
\begin{proof}
    Smoothness: $ h^*(t,x) = (\int \phi_{\sigma^2_{T-t}} (x - \mu_{T-t} x_0) w(x_0) p_0 (x_0) d x_0 ) / ( \int \phi_{\sigma^2_{T-t}} (x - \mu_{T-t} x_0) p_0 (x_0) dx_0 ) $ is a ratio of Gaussian convolutions with a strictly positive denominator on $ (0, T). $ Both are $ C^\infty $ in $ x $ by differentiation under the integral, justified by the compact support and Gaussian tails, and the denominator is positive, so $ h^* \in C^\infty. $ Smoothness in $ t $ is likewise clear since $ \mu_{T-t}, \sigma_{T-t} $ are smooth on $ (0, T). $

    Tweedie identity: From \eqref{eq:gauss_posterior}, $ \partial_{x_k} \log \pi_{t, x} (x_0) = \partial_{x_k} \left[ - \frac{| x - \mu_{T-t} x_0 |^2}{2 \sigma^2_{T-t}} - \log Z_t (x) \right] $ with $ Z_t(x) = \int \phi_{\sigma^2_{T-t}} (x - \mu_{T-t} x_0) p_0(x_0) dx_0. $ The first term differentiates to $ - (x_k - \mu_{T-t} x_{0, k}) / \sigma^2_{T-t}, $ and $ \partial_{x_k} \log Z_t(x) = \int \left( - \frac{x_k - \mu_{T-t} x_{0, k}}{\sigma^2_{T-t}} \right) \pi_{t, x} (x_0) dx_0 = - \frac{x_k}{\sigma^2_{T-t}} + \frac{\mu_{T-t}}{\sigma^2_{T-t}} \mathbb{E} [X_{0,k} | x]. $ Subtracting, $ \partial_{x_k} \log \pi_{t,x} (x_0) = \frac{\mu_{T-t}}{\sigma^2_{T-t}} (x_{0,k} - \mathbb{E} [X_{0,k} | x]). $ Hence, since $ \partial_{x_k} \pi_{t,x} = \pi_{t,x} \partial_{x_k} \log \pi_{t,x}, $
    \begin{multline*}
        \partial_{x_k} h^* = \int w(x_0) \partial_{x_k} \pi_{t,x} (x_0) dx_0
        = \frac{\mu_{T-t}}{\sigma^2_{T-t}} \int w(x_0) (x_{0,k} - \mathbb{E}[X_{0,k} | x]) \pi_{t,x} (x_0) dx_0 \\
        = \frac{\mu_{T-t}}{\sigma^2_{T-t}} \Cov (w(X_0), X_{0,k} | x),
    \end{multline*}
    which is \eqref{eq:tweedie} coordinatewise.

    Bounds: With $ \bar{w} := \mathbb{E} [w | x], $ Cauchy-Schwarz gives $ | \Cov (w, X_{0,k} | x) | \leq \| w - \bar{w} \|_\infty \sqrt{\Var (X_{0,k} | x)} \leq \bar{B}, $ using $ \| w - \bar{w} \|_\infty \leq \bar{B} $ and $ \Var (X_{0, k} | x) \leq 1. $ With $ \mu_{T-t} \leq 1 $ this yields $ | \partial_{x_k} h^* | \leq \bar{B} / \sigma^2_{T-t}, $ which we record in the weaker form \eqref{eq:h_der_bounds} the form in which it is used below. Differentiating  \eqref{eq:tweedie} once more and using that every conditional expectation obeys the same Tweedie rule expresses $ \partial^2_{x_k, x_l} h^* $ as $ \mu^2_{T-t} \sigma^{-4}_{T-t} $ times a fixed combination of posterior central third moments of $ (w(X_0), X_{0,k}, X_{0,l}), $ each bounded by an absolute multiple of $ \bar{B} $ under Assumption \ref{ass:compact}.

    Hessian: By \eqref{eq:score_identity_1}, $ \nabla \log \rho_t (x) = \sigma^{-2}_{T-t} ( \mu_{T-t} \mathbb{E} [X_0 | x] - x ), $ and $ \nabla_x \mathbb{E} [X_0 | x] = \mu_{T-t} \sigma^{-2}_{T-t} \Cov (X_0 | x) $ by the Tweedie rule applied coordinatewise. Combining gives the identity in \eqref{eq:hessian}. The bounds follow from $ 0 \preceq \Cov (X_0 | x) \preceq \mathbb{E} [|X_0|^2 | x] I_d \preceq d I_d. $
\end{proof}

Equation \eqref{eq:tweedie} exhibits the guidance as a rescaled posterior covariance between the weight and the clean signal. The $ \sigma^{-2}, \sigma^{-4} $ scalings are the quantitative form of the endpoint singularity that \eqref{eq:window} excises.

\begin{lemma}(The guidance is bounded, with no lower bound on $ w $) \label{lem:bounded_guidance}
    Under Assumptions \ref{ass:bounded}--\ref{ass:compact}, write $ \pi^w_{t,x}(dx_0) \propto w(x_0) \pi_{t,x} (dx_0) $ for the $w$-tilted posterior. Then for all $ t \in (0, T) $ and $ x \in \mathbb{R}^d $
    \begin{equation} \label{eq:guidance_bound}
        g^* (t,x) = \frac{\mu_{T-t}}{\sigma^2_{T-t}} \left( \mathbb{E}_{\pi^w_{t,x}} [X_0] - \mathbb{E}_{\pi_{t,x}}[X_0] \right), 
        \quad \text{hence } | g^* (t,x) | \leq \frac{2 \sqrt{d}}{\sigma^2_{T-t}},
    \end{equation}
    and in particular $ | g^* (t,x) | \leq G^* := 2 \sqrt{d} / \sigma^2_* $ for $ t \in (0, t_1]. $
\end{lemma}
\begin{proof}
    Since $ w > 0 $ on $ \mathrm{supp}(p_0) $ and $ \pi_{t,x} $ is supported there, $ h^*(t,x) = \mathbb{E}_{\pi_{t,x}}[w] > 0 $ and $ \pi^w_{t,x} $ is a well-defined probability measure. Dividing \eqref{eq:tweedie} by $ h^* $ and expanding the covariance
    \begin{equation}
        \frac{\Cov (w(X_0), X_{0,k} | x)}{\mathbb{E}[w(X_0) | x]} = \frac{\mathbb{E}[w(X_0) X_{0,k} | x]}{\mathbb{E}[w(X_0) | x]} - \mathbb{E}[X_{0,k} | x] = \mathbb{E}_{\pi^w_{t,x}}[X_{0,k}] - \mathbb{E}_{\pi_{t,x}}[X_{0,k}],
    \end{equation}
    which is \eqref{eq:guidance_bound}. Both expectations are averages of $ X_0 $ over probability measures supported in $ \mathrm{supp}(p_0) \subseteq \{ | x_0 |_{\infty} \leq 1 \}, $ hence lie in the convex hull of that cube and differ by at most $ 2 \sqrt{d} $ in norm. Since $ \mu_{T-t} \leq 1 $ this gives the first bound, and $ \sigma_{T-t} \geq \sigma_* $ for $ t \leq t_1 $ gives the second.
\end{proof}

\begin{lemma}(Negative moments transfer to $ h^* $) \label{lem:negative_moments}
    Under Assumptions \ref{ass:bounded}--\ref{ass:compact}, for every $ t \in (0, T) $ and every $ \mathrm{b} > 0, $
    \begin{equation}
        \mathbb{E}_{\rho_t} [ (h^*)^{-\mathrm{s}} ] \leq \Xi_\mathrm{s}, \quad \text{hence } \rho_t (h^* < \mathrm{b}) \leq \Xi_\mathrm{s} \mathrm{b}^\mathrm{s}.
    \end{equation}
\end{lemma}
\begin{proof}
    The map $ u \mapsto u^{-\mathrm{s}} $ is convex on $ (0, \infty), $ so conditional Jensen applied to $ h^* (t,x) = \mathbb{E}[w(X_0) | X_{T-t} = x] $ gives $ h^*(t,x)^{-\mathrm{s}} \leq \mathbb{E}[w(X_0)^{-\mathrm{s}} | X_{T-t} = x] $ pointwise. Taking expectations over $ X_{T-t} \sim \rho_t $ and using the tower property gives $ \mathbb{E}_{\rho_t} [ (h^*)^{-\mathrm{s}} ] \leq \mathbb{E}_{p_0} [ w^{-\mathrm{s}} ] = \Xi_\mathrm{s}, $ uniformly in $ t. $ Markov's inequality applied to $ (h^*)^{-\mathrm{s}} $ at level $ \mathrm{b}^{-\mathrm{s}} $ gives the second claim.
\end{proof}

\section{Harmonicity, the weak form, and gap-freenesss} \label{app:harmonic}

\begin{proposition} (Martingale property) \label{prop:martingale}
    Under Assumption \ref{ass:bounded}, the process $ L_t := h^* (t, \overleftarrow{X}_t), \, t \in [t_0, t_1] $ is a martingale for the law $ \mathbb{P} $ of the reverse SDE \eqref{eq:rev_denoising}, with respect to its natural filtration $ \mathcal{F}_t = \sigma (\overleftarrow{X}_s : 0 \leq s \leq t). $
\end{proposition}
\begin{proof}
    By the Markov property of \eqref{eq:rev_denoising} and the tower property of conditional expectation
    \begin{equation}
        \mathbb{E} [ w(\overleftarrow{X}_T) | \mathcal{F}_t ] = \mathbb{E} [ w (\overleftarrow{X}_T) | \overleftarrow{X}_t ] = h^* (t, \overleftarrow{X}_t) = L_t,
    \end{equation}
    using Definition \ref{def:h_star}. Thus $ (L_t) $ is a Doob martingale -- a conditional expectation of the single terminal random variable $ w(\overleftarrow{X}_T) $ along the filtration. Integrability is following from $ | w | \leq \bar{B}. $
\end{proof}

\begin{theorem} (Harmonicity of $ h^* $) \label{thm:harmonicity}
    Suppose Assumptions \ref{ass:bounded}--\ref{ass:compact} hold. Then $ h^* $ is space-time harmonic for the reverse generator:
    \begin{equation} \label{eq:harmonicity}
        \mathcal{M} h^* = \partial_t h^* + b(t, x) \cdot \nabla h^* + 
        \Delta h^* = 0 \quad \text{on} ~ (0, T) \times \mathbb{R}^d, \quad h^*(T, \cdot) = w.
    \end{equation}
    Moreover $ h^* $ admits the martingale representation $ d h^* (t, \overleftarrow{X}_t) = \sqrt{2} \nabla h^* (t, \overleftarrow{X}_t) \cdot d B_t. $
\end{theorem}
\begin{proof}
    $ h^* \in C^{1,2} (I \times \mathbb{R}^d), $ by Proposition \ref{prop:derivative_bounds}. It\^o's formula for \eqref{eq:rev_denoising} gives
    \begin{equation}
        d L_t = ( \partial_t h^* + \mathcal{L}_t h^* )(t, \overleftarrow{X}_t) dt + \sqrt{2} \nabla h^* (t, \overleftarrow{X}_t) \cdot d B_t = (\mathcal{M} h^*) (t, \overleftarrow{X}_t) dt + \sqrt{2} \nabla h^* (t, \overleftarrow{X}_t) \cdot d B_t.
    \end{equation}
    By Proposition \ref{prop:martingale}, $ (L_t) $ is a martingale, so its drift part must vanish almost surely: $ ( \mathcal{M}h^* ) (x, \overleftarrow{X}_t) = 0 $ for a.e. $ t, $ $ \mathbb{P} $-a.s. The law of $ \overleftarrow{X}_t $ is $ \rho_t = p_{T-t}, $ a Gaussian mixture with full support om $ \mathbb{R}^d. $ Since $ \mathcal{M} h^* $ is continuous and vanishes $ \rho_t $-a.e. for a dense set of $ t, $ it vanishes identically on $ (0, T) \times \mathbb{R}^d. $ The terminal condition is $ h^* (T,x) = \mathbb{E} [w(\overleftarrow{X}_T) | \overleftarrow{X}_T = x] = w(x). $ Substituting $ \mathcal{M}h^* = 0 $ into \eqref{eq:harmonicity} leaves the representation.
\end{proof}

\begin{definition}(Strong residual and penalized risk)
    For $ h \in C^{1,2} $ the residual is $ \mathcal{R}_t [h] := (\mathcal{M}h)(t, \cdot) = \partial_t h + b \cdot \nabla h + \Delta h. $ By Theorem \ref{thm:harmonicity}, $ \mathcal{R}_t [h^*] \equiv 0 $ and, writing $ \varphi := h - h^*, \, \mathcal{M} \varphi = \mathcal{R}_t [h]. $ The population value-fidelity and strong PDE functionals are
    \begin{equation} \label{eq:value_pde_functionals}
        \mathcal{D}(h) := \int_I \| h(t, \cdot) - h^*(t, \cdot) \|^2_{\rho_t} \nu (dt), \quad \mathcal{P}^s (h) := \int_I \| \mathcal{R}_t [h] \|^2_{\rho_t} \nu (dt).
    \end{equation}
\end{definition}

The value functional \eqref{eq:value_pde_functionals} is exactly the population excess least-squares risk: at fixed $ t $ the minimizer over measurable $ h(t, \cdot) $ of $ \mathbb{E}[ ( h(t, X_{T-t}) - w(X_0))^2 ] $ is the conditional mean $ h^* (t, \cdot), $ and the excess risk equals $ \| h(t, \cdot) - h^*(t, \cdot) \|^2_{\rho_t,} $ so $ \mathcal{D} $ is what regression naturally drives down.

\begin{theorem}(Gap-freeness) \label{thm:gap_freeness}
    Under Assumptions \ref{ass:bounded}--\ref{ass:compact}, for every penalty weight $ \mu \geq 0 $ the function $ h^* $ is the unique minimizer, over all measurable $ h $ with $ \varphi \in L^2 (\nu \otimes \rho) $ and $ \mathcal{R}[h] \in L^2 (\nu \otimes \rho), $ of the penalized population risk $ J_\mu := \mathcal{D} + \mu \mathcal{P}^s. $ In particular the penalized target does not move with $ \mu: $ there is no regularization gap. The same holds with $ \mathcal{P}^s $ replaced by any nonnegative functional vanishing at $ h^*. $
\end{theorem}

\begin{proof}
    By Theorem \ref{thm:harmonicity}, $ \mathcal{R}_t [h^*] = 0, $ so $ \mathcal{P}^s (h^*) = 0, $ and $ \mathcal{D}(h^*) = 0 $ by definition. Hence, $ J_\mu (h^*) = 0. $ Since $ \mathcal{D} \geq 0 $ and $ \mathcal{P}^s \geq 0, $ we have $ J_\mu (h) \geq 0 = J_\mu (h^*) $ for all $ h $ and all $ \mu \geq 0, $ so $ h^* $ minimizes. If $ J_\mu (h) = 0 $ then in particular $ \mathcal{D}(h) = 0, $ i.e. $ \| h - h^* \|_{L^2 (\nu \otimes \rho)} = 0, $ whence $ h = h^* $ ($ \nu \otimes \rho $-a.e.). Uniqueness follows.
\end{proof}

Contrast this with a Sobolev regularized objective $ J^\lambda (h) := \mathcal{D}(h) + \lambda \int_I \| \nabla h \|^2_{\rho_t} \nu(dt), $ $ \lambda > 0: $ its penalty does not vanish at $ h^*. $ If $ \int_I \| \nabla h^* \|^2_{\rho_t} \nu (dt) > 0, $ then $ h^* $ does not minimize $ J^\lambda, $ since $ J^\lambda ((1 - \epsilon) h^*) - J^\lambda (h^*) = \epsilon^2 \int_I \| h^* \|^2_{\rho_t} \nu(dt) - \lambda (2 \epsilon - \epsilon^2) \int_I \| \nabla h^* \|^2_{\rho_t} \nu (dt) < 0 $ for all sufficiently small $ \epsilon > 0. $ The resulting shrinkage bias must be traded against variance in $ \lambda. $ Gap-freeness removes that trade-off.

\textbf{The bottleneck.} Gap-freeness alone does not improve the rate. Since $ \mathcal{M} \varphi $ contains $ \Delta \varphi, $ the best second-derivative control a resolution-$N$ approximant of a $ W^{\beta, \infty} $ function supplies is $ \| \nabla^2 (h_N - h^*) \|_{L^2} \lesssim N^{- (\beta - 2)}, $ so the corresponding upper bound is $ \underset{h \in \mathcal{H}}{\inf} \mathcal{P}^s (h) $ is $ N^{-2 (\beta - 2)}, $ which is $ O(1) $ vacuous at a $ \beta = 2. $ Removing this floor is the sole purpose of the weak form.

The idea is to test the residual against a function $ \psi $ and move one derivative off $ \varphi $ onto $ \psi $ by integration by parts. The key algebraic point is that the weight generated by integrating against $ \rho_t $ is exactly the base score, because $ \nabla \log \rho_t = \nabla \log p_{T-t} = s. $

\begin{lemma}(Weak form of residual) \label{lem:weak_form}
    Let $ t \in I, \, h \in C^{1,2} $ and $ \psi \in H^1 (\rho_t), $ with $ h, \nabla h, \psi $ of polynomial growth. Then, with $ \nabla \log \rho_t = s(t, \cdot) $ and $ b - s = x + s, $
    \begin{equation} \label{eq:weak_res}
        \langle \mathcal{R}_t [h], \psi \rangle_{\rho_t} = \underbrace{\int_{\mathbb{R}^d} \left[ \partial_t h \psi \, + \, ((x+s) \cdot \nabla h) \psi \, - \, \nabla h \cdot \nabla \psi \right] \rho_t dx}_{:= a_t (h, \psi)}.
    \end{equation}
    This bilinear form $ a_t $ is first order in both $ h $ and $ \psi $ (no $ \Delta h $ and no $ \Delta \psi $), and is computable from the covariates and the known score $ s. $ Moreover $ a_t (h^*, \psi) = 0 $ for every $ \psi, $ that is, $ h^* $ is the weak solution of \eqref{eq:harmonicity}.
\end{lemma}
\begin{proof}
    Expand $ \langle \mathcal{R}_t [h], \psi \rangle_{\rho_t} = \int ( \partial_t h + b \cdot \nabla h + \Delta h ) \psi \rho_t $ and integrate the Laplacian term by parts. With no boundary contribution at infinity (polynomial growth against the Gaussian-tailed $ \rho_t $)
    \begin{equation}
        \int \Delta h \psi \rho_t = - \int \nabla h \cdot \nabla (\psi \rho_t) = - \int \nabla h \cdot \nabla \psi \rho_t - \int (\nabla h \cdot \nabla \rho_t) \psi.
    \end{equation}
    Since $ \nabla \rho_t = (\nabla \log \rho_t) \rho_t = s \rho_t, $ the last integrand is $ (s \cdot \nabla h) \psi \rho_t. $ Collecting the three terms
    \begin{equation}
        \langle \mathcal{R}_t [h], \psi \rangle_{\rho_t} = \int \partial_t h \psi \rho_t + \int ((b - s) \cdot \nabla h) \psi \rho_t - \int \nabla h \cdot \nabla \psi \rho_t,
    \end{equation}
    and $ b - s = ( x + 2 s) - s = x + s, $ which is \eqref{eq:weak_res}. Finally, taking $ h = h^* $ and reversing the integration by parts $ a_t (h^*, \psi) = \langle \mathcal{M} h^*, \psi  \rangle_{\rho_t} = 0 $ by Theorem \ref{thm:harmonicity}.
\end{proof}

That the integration-by-parts weight is exactly $ s $ is what makes the weak residual free: no second derivatives, and the only coefficient it needs is the output of the pre-trained model.

\begin{lemma}(Continuity) \label{lem:continuity}
    Under the Assumptions \ref{ass:bounded}--\ref{ass:compact}, with $ M_4 \leq C \sqrt{d} \sigma_*^{-2} $ from Lemma \ref{lem:score_identity}, for all $ h \in C^{1,2} $ of polynomial growth, $ \psi \in H^1 (\rho_t) $ and $ \varphi := h - h^* $
    \begin{equation} \label{eq:continuity}
        | a_t (h, \psi) | \leq \left( \| \partial_t \varphi \|_{\rho_t} + \| \nabla \varphi \|_{\rho_t} + M_4 \| | \nabla \varphi |\|_{L^4 (\rho_t)} \right) \| \psi \|_{H^1 (\rho_t)}.
    \end{equation}
    Consequently $ \| \mathcal{R}_t [h] \|_{H^{-1} (\rho_t)} \leq \| \partial_t \varphi \|_{\rho_t} + \| \nabla \varphi \|_{\rho_t} + M_4 \| | \nabla \varphi | \|_{L^4 (\rho_t)}. $
\end{lemma}
\begin{proof}
    Because $ a_t $ is linear in its first argument and $ a_t (h^*, \psi) = 0 $ by Lemma \ref{lem:weak_form}, $ a_t (h, \psi) = a_t (\varphi, \psi). $ Bound the three terms of $ a_t (\varphi, \psi) $ in \eqref{eq:weak_res}:
    \begin{align}
        \left| \int \partial_t \varphi \psi \rho_t \right| & \leq \| \partial_t \varphi \|_{\rho_t} \| \psi \|_{\rho_t}, \\
        \left| \int \nabla \varphi \cdot \nabla \psi \rho_t \right| & \leq \| \nabla \varphi \|_{\rho_t} \| \nabla \psi \|_{\rho_t}, \\
        \left| \int ((x + s) \cdot \nabla \varphi) \psi \rho_t \right| & \leq \| (x + s) \cdot \nabla \varphi \|_{\rho_t} \| \psi \|_{\rho_t} \leq M_4 \| \nabla \varphi \|_{L^4 (\rho_t)} \| \psi \|_{\rho_t},
    \end{align}
    where the last line uses Cauchy-Schwarz twice: first $ \int ((x + s) \cdot \nabla \varphi) \psi \rho_t \leq \| (x + s) \cdot \nabla \varphi \|_{\rho_t} \| \psi \|_{\rho_t}, $ then $ \| (x + s) \cdot \nabla \varphi \|_{\rho_t} \leq \| | x + s | \|_{L^4 (\rho_t)} \| | \nabla \varphi | \|_{L^4 (\rho_t)} \leq M_4 \| \nabla \varphi \|_{L^4 (\rho_t)}. $ Summing and bounding $ \| \psi \|_{\rho_t}, \, \| \nabla \psi \|_{\rho_t} \leq \| \psi \|_{H^1 (\rho_t)} $ gives \eqref{eq:continuity}. The dual-norm consequence is \eqref{eq:h_inv_norm}. 
\end{proof}

\begin{definition}(Weak residual norm and penalty)
    The weak (dual) residual norm at time $ t $ is $ \| \mathcal{R}_t [h] \|_{H^{-1} (\rho_t)} $ as in \eqref{eq:h_inv_norm}, with $ \langle \mathcal{R}_t [h], \psi \rangle_{\rho_t} = a_t (h, \psi) $ by Lemma \ref{lem:weak_form}. The associated population penalty is 
    \begin{equation}
        \mathcal{P}^w (h) := \int_I \| \mathcal{R}_t [h] \|^2_{H^{-1} (\rho_t)} \nu (dt) = \| \mathcal{R}[h] \|^2_{\mathcal{V}^*}.
    \end{equation}
\end{definition}

By Lemma \ref{lem:weak_form}, $ \mathcal{P}^w (h^*) = 0, $ so the gap-freeness argument of Theorem \ref{thm:gap_freeness} applies verbatim with $ \mathcal{P}^w $ in place of $ \mathcal{P}^s. $ $ \mathcal{P}^w $ is measured in the weaker $ H^{-1} $ norm and the two remaining tasks are to shown this weaker penalty still controls the gradient and has better approximation floor. 

\section{The energy identity and coercivity} \label{app:coercivity}

The bridge from residual to gradient is an exact identity obtained by applying It\^o's formula to $ \varphi^2 $ along the reverse process.

\begin{theorem}(Energy identity) \label{thm:energy_identity}
    Let $ \varphi \in C^{1,2} (I \times \mathbb{R}^d) $ be bounded, with $ \nabla \varphi \in L^2 (dt \otimes \rho) $ and $ \mathcal{M} \varphi \in L^1 (dt \otimes \rho). $ Then, for every $ r \in I, $
    \begin{equation} \label{eq:energy_identity}
        2 \int_{t_0}^{r} \| \nabla \varphi (t, \cdot) \|^2_{\rho_t} dt = \| \varphi (r, \cdot) \|_{\rho_{r}}^2 - \| \varphi (t_0, \cdot) \|_{\rho_{t_0}}^2 - 2 \int_{t_0}^{r} \langle \varphi(t, \cdot), \mathcal{M} \varphi (t, \cdot) \rangle_{\rho_t} dt.
    \end{equation}
\end{theorem}
\begin{proof}
    For $ \varphi \in C^{1,2} $ a direct computation gives the pointwise identity
    \begin{equation} \label{eq:m_phi_square}
        \mathcal{M} (\varphi^2) = 2 \varphi \, \mathcal{M} \varphi + 2 | \nabla \varphi |^2 .
    \end{equation}
    Indeed $ \partial_t (\varphi^2) = 2 \varphi \partial_t \varphi, \, b \cdot \nabla (\varphi^2) = 2 \varphi \, b \cdot \nabla \varphi, $ and $ \Delta (\varphi^2) = 2 \varphi \Delta \varphi + 2 | \nabla \varphi |^2. $ Summing and grouping the factor $ 2 \varphi $ against $ \mathcal{M} \varphi $ yields \eqref{eq:m_phi_square}, in which the term $ | \nabla \varphi |^2 $ is the carr\'e du champ $ \Gamma (\varphi, \varphi) $ of $ \mathcal{M} $ \citep[\S1.4.2]{bakry2014}. It\^o's formula for \eqref{eq:rev_denoising} applied to $ t \mapsto \varphi^2 (t, \overleftarrow{X}_t) $ then reads
    \begin{equation} \label{eq:d_phi_square}
        d [ \varphi^2 (t, \overleftarrow{X}_t) ] = ( 2 \varphi \mathcal{M} \varphi + 2 | \nabla \varphi |^2 ) (t, \overleftarrow{X}_t) dt + 2 \sqrt{2} (\varphi \nabla \varphi) (t, \overleftarrow{X}_t) \cdot d B_t.
    \end{equation}

    For $ R > 0 $ let $ \tau_R := \inf \{ t \geq t_0: | \overleftarrow{X}_t | \geq R \} \land r. $ On the stochastic interval $ [t_0, \tau_R] $ the state stays in the ball of radius $ R, $ on which the continuous integrand $ \varphi \nabla \varphi $ is bounded, hence $ \mathbb{E} \int_{t_0}^{\tau_R} | \varphi \nabla \varphi |^2 (t, \overleftarrow{X}_t) \, dt < \infty $ and the stopped stochastic integral $ \int_{t_0}^{\cdot \land \tau_R} \varphi \nabla \varphi \cdot dB $ is a continuous square-integrable martingale started at $ 0, $ so it has mean zero \citep[\S3.2]{karatzas1991}, see also \citep[Ch.~3]{oksendal2003}. Taking expectations in the stopped version of \eqref{eq:d_phi_square}, the martingale term drops:
    \begin{equation}
        \mathbb{E} [ \varphi^2 (\tau_R, \overleftarrow{X}_{\tau_R})] - \mathbb{E} [\varphi^2 (t_0, \overleftarrow{X}_{t_0})] = \mathbb{E} \int_{t_0}^{\tau_R} (2 \varphi \mathcal{M} \varphi + 2 | \nabla \varphi |^2) (t, \overleftarrow{X}_t) dt.
    \end{equation}
    Now let $ R \rightarrow \infty. $ Since $ | \overleftarrow{X}_t | < \infty $ a.s.\ on $ [t_0, r], $ $ \tau_R \uparrow r $ a.s. The left side converges by bounded convergence since $ \varphi $ is bounded. On the right, the $ 2 | \nabla \varphi |^2 \geq 0 $ contribution converges by monotone convergence and the $ 2 \varphi \mathcal{M} \varphi $ contribution by dominated convergence, because $ | \varphi | \leq \| \varphi \|_{\infty} $ and $ \mathcal{M} \varphi \in L^1 (dt \otimes \rho) $ make it absolutely integrable. Using $ \overleftarrow{X}_t \sim \rho_t $ to write $ \mathbb{E}[G(t, \overleftarrow{X}_t)] = \int G(t, x) \rho_t (dx) $ for the surviving deterministic integrands,
    \begin{equation}
        \| \varphi (r) \|^2_{\rho_{r}} - \| \varphi (t_0) \|^2_{\rho_{t_0}} = 2 \int_{t_0}^{r} \left( \langle \varphi, \mathcal{M} \varphi \rangle_{\rho_t} + \| \nabla \varphi \|^2_{\rho_t} \right) dt .
    \end{equation}
    Rearranging gives \eqref{eq:energy_identity}.
\end{proof}

We now show that the weak residual controls the gradient. The mechanism is to recognize the cross term in \eqref{eq:energy_identity} as the bilinear form $ a_t $ evaluated at the test function $ \psi = \varphi $ and to bound it by the $ H^{-1} $ norm of the residual times the $ H^1 $ norm of $ \varphi, $ absorbing the resulting gradient factor.

Throughout this section and the next, $ \nu $ is normalized Lebesgue measure on $ I, ~ \ell := t_1 - t_0, ~ \ell \leq 1 $ and 
\begin{equation}
    \omega(t) := \nu ([t, t_1]) = \frac{t_1 - t}{t_1 - t_0} \in [0,1], \quad \omega(t_0) = 1, ~ \omega(t_1) = 0, ~ \omega \geq \frac{1}{2} \text{ on } I\circ.
\end{equation}
The weight $ \omega $ is not a modeling choice. The energy identity of Theorem \ref{thm:energy_identity} holds on $ [t_0, r] $ for each fixed $ r $ and leaves a terminal term $ \| \varphi (r, \cdot) \|^2_{\rho_r}, $ a value error at a single instant, which the statistical analysis does not control -- it controls only the $\nu$-average $ \mathcal{D}(h). $ Averaging the identity over a random endpoint $ r \sim \nu $ turns that term into exactly $ \mathcal{D}(h), $ and exchanging the order of integration on the left produces $ \int_I f(t) \nu ([t, t_1]) dt. $ Thus $ \omega $ is the survival function of $ \nu, $ and its affine form is simply the uniformity of $ \nu. $ Its three relevant properties each do work below: $ \omega(t_1) = 0 $ removes a boundary term in the integration by parts of Theorem \ref{thm:coercivity}, $ \omega^2 \leq \omega $ allows the in-class term of Theorem \ref{thm:coercivity} to be absorbed into the same weighted energy that appears on the left, and $ \omega \geq \frac{1}{2} $ on $ I^\circ $ converts the weighted bound back into an unweighted one --which is the reason the guarantee is stated on the half-window rather that on $ I. $

For a space-time test function $ \psi $ we write $ A (h, \psi) := \int_I a_t (h, \psi(t, \cdot)) \nu (dt), $ so that $ \| \mathcal{R} [h] \|_{\mathcal{V}^*} = \underset{\psi \neq 0}{\sup} A (h, \psi) / \| \psi \|_{\mathcal{V}}. $

\begin{theorem}(Weighted energy identity) \label{thm:weighted_energy_identity}
    Under Assumptions \ref{ass:bounded}--\ref{ass:compact}, let $ h \in C^{1,2} (I \times \mathbb{R}^d) $ with $ | h | \leq 3 \bar{B} $ and $ \nabla \varphi \in L^2 (dt \otimes \rho) $ and 
    \begin{equation} \label{eq:wei_bound}
        \underset{I \times \mathbb{R}^d}{\sup} ( | \partial_t h | + | \nabla^2 h |) < \infty,
    \end{equation}
    and write $ \varphi := h - h^*. $ Then, with $ E_\omega := \int_I \omega (t) \| \nabla \varphi (t, \cdot) \|^2_{\rho_t} \nu (dt), $
    \begin{equation} \label{eq:weighted_energy_identity}
        2 \ell E_\omega + \| \varphi (t_0, \cdot ) \|^2_{\rho_{t_0}} = \mathcal{D}(h) - 2 \ell A (h, \omega \varphi).
    \end{equation}
\end{theorem}
\begin{proof}
    The hypotheses of Theorem \ref{thm:energy_identity} hold. Boundedness is immediate $ | \varphi | \leq 4 \bar{B}. $ Let's split 
    \begin{equation}
        \mathcal{M} \varphi = (\partial_t h - \partial_t h^*) + b \cdot ( \nabla h - \nabla h^*) + (\Delta h - \Delta h^*)
    \end{equation}
    and bound the six terms. By \eqref{eq:wei_bound}, $ \partial_t h $ and $ \Delta h $ are bounded. By \eqref{eq:h_der_bounds}, $ | \nabla h^* | \leq 2 \sqrt{d} \bar{B} \sigma_*^{-2} $ and $ | \Delta h^* |  \leq d C_2 \bar{B} \sigma_*^{-4}, $ whence by harmonicity
    \begin{equation}
        | \partial_t h^* | = | b \cdot \nabla h^* + \Delta h^* | \leq 2 b \sqrt{d} \bar{B} \sigma_*^{-2} + d C_2 \bar{B} \sigma_*^{-4}.
    \end{equation}
    Finally, $ |b| \leq |x + s| + |s| $ lies in $ L_2 (\rho_t) $ uniformly on $ I $ by Lemma \ref{lem:score_identity} for the first summand and by $ |s| \leq \sigma^{-2}_{T-t} (|x| + \sqrt{d}) $ from \eqref{eq:score_identity_1} for the second, so Cauchy-Schwarz in $ dt \otimes \rho $ gives
    \begin{equation}
        \int_I \int |b| |\nabla h - \nabla h^*| \rho_t dt \leq \sqrt{\ell} \, \underset{t \in I}{\sup} \| | b| \|_{L^2 (\rho_t)} \cdot \| | \nabla \varphi | \|_{L^2 (dt \otimes \rho_t)} < \infty
    \end{equation}
    by the hypothesis $ \nabla \varphi \in L^2 (dt \otimes \rho_t). $ Every remaining term of \eqref{eq:weighted_energy_identity} is bounded by affine function of $ |b|, $ hence integrable against Gaussian-tailed $ \rho_t. $ Thus $ \mathcal{M} \varphi \in L^1(dt \otimes \rho_t). $ Its cross term is $ \langle \varphi, \mathcal{M} \varphi \rangle_{\rho_t} = a_t (h, \varphi) $ by Lemma \ref{lem:weak_form}. Integrating \eqref{eq:energy_identity} over the endpoint $ r $ against $ \nu $ and using $ \int_I ( \int_{t_0}^r f(t) dt ) \nu (dr) = \int_I f(t) \omega(t) dt = \ell \int_I f \omega \nu(dt) $ for $ f \in L^1(I), $ together with $ \int_I \| \varphi (r, \cdot) \|^2_{\rho_r} \nu(dr) = \mathcal{D}(h), $ gives \eqref{eq:weighted_energy_identity}. 
\end{proof}

Identity \eqref{eq:weighted_energy_identity} is exact. Two consequences follow, according to how the cross term $ A(h, \omega \varphi) $ is bounded.

\begin{corollary}(Coercivity, dual-norm form) \label{cor:coercivity}
    Under the hypotheses of Theorem \ref{thm:weighted_energy_identity}, for every $ \theta \in (0, 2 \ell) $
    \begin{equation} \label{eq:coercivity}
        E_\omega \leq \frac{1 + \theta}{2 \ell - \theta} \mathcal{D}(h) + \frac{\ell^2 / \theta}{2 \ell - \theta} \mathcal{P}^w (h),
    \end{equation}
    and in particular, taking $ \theta = \ell $ and using $ \omega \geq \frac{1}{2} $ on $ I^\circ $
    \begin{equation} \label{eq:coercivity_on_half}
        \int_{I^\circ} \| \nabla \varphi \|^2_{\rho_t} \nu (dt) \leq \frac{2 ( 1 + \ell)}{\ell} \mathcal{D}(h) + 2 \mathcal{P}^w (h).
    \end{equation}
\end{corollary}
\begin{proof}
    Abbreviate $ E := E_\omega, ~ D := \mathcal{D}(h) $ and $ P := \mathcal{P}^w (h) = \| \mathcal{R}[h] \|^2_{\mathcal{V}^*}. $

    Step 1: Both terms on the left of \eqref{eq:weighted_energy_identity} are nonnegative, so dropping $ \| \varphi (t_0, \cdot) \|^2_{\rho_{t_0}} $ gives
    \begin{equation} \label{eq:energy_bound}
        2 \ell E \leq D + 2 \ell | A(h, \omega \varphi) |.
    \end{equation}
    Note that $ D $ enters here without a factor $ \ell, $ because it arises from the terminal value $ \| \varphi (r, \cdot) \|^2_{\rho_r} $ of the energy identity, which is averaged over the endpoint $ r $ and not integrated in $ t. $ Every contribution coming from the cross term, by contrast, carries $ \ell. $ 

    Step 2: The function $ \omega \varphi $ belongs to $ \mathcal{V}, $ so by the definition of $ \| \cdot \|_{\mathcal{V}^*} $
    \begin{equation}
        | A (h, \omega \varphi) | \leq \| \mathcal{R}[h] \|_{\mathcal{V}^*} \| \omega \varphi \|_{\mathcal{V}} = \sqrt{P} \| \omega \varphi \|_{\mathcal{V}}.
    \end{equation}

    Step 3: Since $ \omega $ depends only on $ t, $ we have $ \nabla (\omega \varphi) = \omega \nabla \varphi, $ and since $ 0 \leq \omega \leq 1 $ we have both $ \omega^2 \leq 1 $ and $ \omega^2 \leq \omega. $ Using the first on the value part and the second on the gradient part
    \begin{equation}
        \| \omega \varphi \|^2_{\mathcal{V}} = \int_I \omega^2 ( \| \varphi \|^2_{\rho_t} + \| \nabla \varphi \|^2_{\rho_t} ) \nu (dt) \leq \int_I \| \varphi \|^2_{\rho_t} \nu(dt) + \int_I w \| \nabla \varphi \|^2_{\rho_t} \nu (dt) = D + E.
    \end{equation}
    The second inequality is where $ \omega^2 \leq \omega $ is used: it produces the same weighted energy $ E $ that stands on the left of \eqref{eq:energy_bound}, which is what makes absorption possible.

    Step 4: Combining last steps and applying $ 2 a b \leq \theta a^2 + \theta^{-1} b^2 $ with $ a = \sqrt{D + E}, ~ b = \ell \sqrt{P} $ and any $ \theta > 0 $
    \begin{equation}
        2 \ell E \leq D + 2 \ell \sqrt{P} \sqrt{D + E} \leq D + \theta(D + E) + \frac{\ell^2}{\theta} P.
    \end{equation}
    Collecting the $ E $ terms, which requires $ \theta < 2 \ell $
    \begin{equation}
        (2 \ell - \theta) E \leq (1 + \theta) D + \frac{\ell^2}{\theta} P,
    \end{equation}
    which is \eqref{eq:coercivity}. Setting $ \theta = \ell $ gives $ \ell E \leq (1 + \ell) D + \ell P, $ i.e. $ E \leq \frac{1 + \ell}{\ell} D + P. $

    Step 5: For $ \delta \in (0, 1) $ put $ I_\delta := [t_0, t_1 - \delta \ell], $ on which $ \omega \geq \delta, $ so that
    \begin{equation} \label{eq:delta_coercivity}
        \int_{I_\delta} \| \nabla \varphi \|^2_{\rho_t} \nu(dt) \leq \frac{1}{\delta} \int_{I_\delta} \omega \| \nabla \varphi \|^2_{\rho_t} \nu(dt) \leq \frac{1}{\delta} E.
    \end{equation}
    The case $ \delta = 1/2, $ for which $ I_\delta = I^\circ $ and $ \delta^{-1} = 2, $ is \eqref{eq:coercivity_on_half}. Remark \ref{rem:sampling_window} explains how \eqref{eq:delta_coercivity} places the guarantee on a prescribed sampling window.
\end{proof}

The estimator, however, does not control $ \mathcal{P}^w (h) $ but only its projection onto the test class. The second consequence, proved in Section \ref{app:infsup}, is that with the test class chosen appropriately the projected quantity suffices, and no inf-sup constant is incurred. 

\begin{corollary}(Guidance control) \label{cor:guidance_control}
    Under the hypotheses of Theorem \ref{thm:weighted_energy_identity}, let $ \mathrm{b} \in (0, \bar{B}] $ be a clipping level and define the clipped, projected guidance
    \begin{equation} \label{eq:proj_guidance}
        \hat{g} := \Pi_{G^*} \left( \frac{\nabla h}{h \lor \mathrm{b}} \right), \quad \Pi_{G^*} (v) := v \min \{ 1, G^* / | v | \},
    \end{equation}
    with $ G^* = 2 \sqrt{d} \sigma_*^{-2} $ from Lemma \ref{lem:bounded_guidance}. Then
    \begin{equation} \label{eq:guidance_control}
        \| \hat{g} - g^* \|^2_{L^2 (\nu \otimes \rho; I^\circ)} \leq \frac{2}{\mathrm{b}^2} \int_{I^\circ} \| \nabla \varphi \|^2_{\rho_t} \nu (dt) + \frac{8 d \bar{B}^2}{\mathrm{b}^4} \sigma_*^{-4} \mathcal{D}(h) + 4 (G^*)^2 \Xi_\mathrm{s} \mathrm{b}^\mathrm{s}.
    \end{equation}
\end{corollary}
\begin{proof}
    By Lemma \ref{lem:bounded_guidance}, $ | g^* | \leq G^*, $ so $ \Pi_{G^*} (g^*) = g^* $ and $ \Pi_{G^*} $ is the Euclidean projection onto a convex set, hence 1-Lipschitz. Therefore $ | \hat{g} - g^* | \leq | \nabla h / (h \lor \mathrm{b} ) - g^* | $ pointwise, and it suffices to bound the right-hand side. We split according to the size of the true $ h^*. $

    On $ \{ h^* \geq \mathrm{b} \}. $ Put $ u := h \lor \mathrm{b} \geq \mathrm{b} $ and $ v := h^* \geq \mathrm{b}, $ so that
    \begin{equation}
        \frac{\nabla h}{u} - \frac{\nabla h^*}{v} = \frac{\nabla \varphi}{u} + \nabla h^* \frac{v - u}{u v}.
    \end{equation}
    Here $ | v - u | \leq | \varphi |: $ if $ h \geq \mathrm{b} $ then $ u = h $ and $ v - u = - \varphi. $ If $ h < \mathrm{b} $ then $ u = \mathrm{b} $ and $ 0 \leq v - u = h^* - \mathrm{b} \leq h^* - h = - \varphi. $ With $ | \nabla h^* | \leq 2 \sqrt{d} \bar{B} \sigma_*^{-2} $ from \eqref{eq:guidance_bound} and $ u, v \geq \mathrm{b} $
    \begin{equation}
        | \nabla h / (h \lor \mathrm{b}) - g^* | \leq \mathrm{b}^{-1} | \nabla \varphi | + 2 \sqrt{d} \bar{B} \mathrm{b}^{-2} \sigma_*^{-2} | \varphi |,
    \end{equation}
    and squaring with $ (a_1 + a_2)^2 \leq 2 a_1^2 + 2 a_2^2 $ and integrating over $ I^\circ $ gives the first two of \eqref{eq:guidance_control}.

    On $ \{ h^* < \mathrm{b} \}. $ Both $ \hat{g} $ and $ g^* $ lie in the ball of radius $ G^*, $ the former by construction and the latter by Lemma \ref{lem:bounded_guidance}, so $ | \hat{g} - g^* | \leq 2 G^* $ regardless of how small $ h^* $ is. The contribution is at most $ 4 (G^*)^2 \int_{I^\circ} \rho_t (h^* < \mathrm{b}) \nu(dt), $ which Lemma \ref{lem:negative_moments} bounds by $ 4 (G^*)^2 \Xi_s \mathrm{b}^s. $ 
\end{proof}

\section{The hypothesis classes} \label{app:classes}

Neither property requires $ \mathcal{H} $ or $ \Psi $ to be linear or convex. We record the second as a lemma. 

\begin{definition}(Network classes) \label{def:network_classes}
    Fix the evaluation domain and the spatial cube
    \begin{equation} \label{eq:Q_R_def}
        Q_R := \{ x \in \mathbb{R}^d: | x|_\infty \leq R \}, \quad Z := I \times Q_R, \quad R_* := \max (t_1, R).
    \end{equation}
    The sup-norm ball is the right choice, since it is exactly the range of the spatial clip $ \chi_R $ in Definition \ref{def:trial_test_classes}, so a constraint verified on $ Z $ is automatically global for clipped functions. Note that $ |x| \leq \sqrt{d} R $ on $ Q_R, $ and $ \| z \|_\infty \leq R_* $ for $ z = (t,x) \in Z. $ Let $ \varrho $ be an activation function. For $ W, L \in \mathbb{N} $ and $\Lambda, \mathcal{A} \geq 1 $ and $ P_0 \in \mathbb{N}, $ let $ \mathcal{NN}_\varrho (W, L, \Lambda, P_0, \mathcal{A}) $ be the set of functions
    \begin{equation} \label{eq:network_f}
        f_\vartheta (z) = A_L x^{(L)} (z) + b_L, ~ x^{(0)}(z) : = z, ~ x^{(\ell+1)}(z) := \varrho ( A_\ell x^{(\ell)}(z) + b_\ell) \quad (0 \leq \ell \leq L-1),  
    \end{equation}
    with parameter vector $ \vartheta := (A_0, b_0, \ldots, A_L, b_L) \in \mathbb{R}^P, $ subject to:
    \begin{enumerate}[label=(N\arabic*), leftmargin=2.8em]
        \item all hidden widths at most $ W, $ where $ W \geq d + 1, $ so that $ P \leq (L + 1) (W + 1)^2, $ \label{cond:hidden_widths}
        \item all entries of $ \vartheta $ bounded in modulus by $ \Lambda, $ and at most $ P_0 $ of them nonzero, \label{cond:vartheta_entries}
        \item all hidden activations bounded on the evaluation domain: $ \underset{z \in Z}{\sup} \| x^{(\ell)}(z) \|_\infty \leq \mathcal{A} $ for $ 1 \leq \ell \leq L. $ \label{cond:bounded_activations}
    \end{enumerate}
\end{definition}

Condition \ref{cond:bounded_activations} is a constraint on the realized network, not on its parameters. It is checkable and satisfied by construction we invoke and is what makes the entropy bound in Lemma \ref{lem:complexity} polylogarithmic in the parameters. Write $ \mathcal{A}_{\mathrm{max}} := \mathcal{A} \lor R_* $ and 
\begin{equation} \label{eq:mathrm_m_def}
    \mathrm{m} := 3 \Lambda W ( \mathcal{A}_{\mathrm{max}} + 1 ) \geq 3,
\end{equation}
a single quantity in terms of which all the constants below are expressed.

\begin{lemma}(Differences of networks are networks) \label{lem:networks_difference}
    Let $ \Theta : \R \to \R $ be bounded, write $ \norm{\Theta}_\infty := \sup_{u\in\R}|\Theta(u)|, $ and put $ \Theta\circ\mathcal{NN}_\varrho(W,L,\Lambda,P_0,\mathcal{A}) := \{\Theta\circ g : g\in\mathcal{NN}_\varrho(W,L,\Lambda,P_0,\mathcal{A})\}. $ Then, for $ \Lambda\geq1, $
    \begin{multline} \label{eq:difference}
        \Theta\circ\mathcal{NN}_\varrho(W,L,\Lambda,P_0,\mathcal{A}) \;-\; \Theta\circ\mathcal{NN}_\varrho(W,L,\Lambda,P_0,\mathcal{A}) \\
        \;\subseteq\; \mathcal{NN}_{\varrho,\Theta}\bigl( 2W,\, L+1,\, \Lambda,\, 2P_0+2,\, \mathcal{A}\vee\norm{\Theta}_\infty \bigr) ,
    \end{multline}
    where $ \mathcal{NN}_{\varrho,\Theta} $ denotes the same architecture with $ \Theta $ used as the activation of the last hidden layer.
\end{lemma}
\begin{proof}
    Let $ g_1, g_2 \in \mathcal{NN}_\varrho(W,L,\Lambda,P_0,\mathcal{A}) $ have parameters $ (A^{(i)}_\ell, b^{(i)}_\ell)_{0\leq\ell\leq L}. $ Build a new network whose hidden layers are the two originals stacked side by side: duplicate the input by $ \tilde A_0 := (A^{(1)}_0; A^{(2)}_0), $ and for $ 1\leq\ell\leq L-1 $ take block-diagonal weight matrices $ \tilde A_\ell := \operatorname{diag}(A^{(1)}_\ell, A^{(2)}_\ell) $ and stacked biases $ \tilde b_\ell := (b^{(1)}_\ell, b^{(2)}_\ell). $ Because $ \varrho $ acts coordinatewise, no information crosses between the two blocks, so after $ L $ hidden layers the two halves of the activation vector are the $ x^{(L)} $ of $ g_1 $ and of $ g_2. $

    Append one more layer, of width $ 2, $ with weights $ \operatorname{diag}(A^{(1)}_L, A^{(2)}_L) $ and biases $ (b^{(1)}_L, b^{(2)}_L) $ and with $ \Theta $ as its activation: its pre-activations are $ (g_1(z), g_2(z)) $ and its activations are $ (\Theta(g_1(z)), \Theta(g_2(z))). $ Read out with the output weights $ (1,-1) $ and zero bias. The result is $ \Theta(g_1(z)) - \Theta(g_2(z)), $ which is the general element of the left-hand side of \eqref{eq:difference}.

    It remains to check \textup{(N1)--(N3)} for the new network. \textup{(N1)}: the hidden widths of layers $ 1,\dots,L $ are at most $ 2W, $ and the appended layer has width $ 2\leq2W, $ the depth is $ L+1. $ \textup{(N2)}: every parameter used is a parameter of $ g_1 $ or of $ g_2, $ a structural zero, or one of the two output weights $ \pm1, $ so all entries are bounded in modulus by $ \Lambda\vee1 = \Lambda, $ and the number of nonzero entries is at most $ 2P_0 $ (those of $ g_1 $ and $ g_2, $ all of which are reused, the readout weights $ A^{(i)}_L $ and biases $ b^{(i)}_L $ having been moved into the appended layer) plus the two output weights, i.e.\ at most $ 2P_0+2. $ \textup{(N3)}: the activations of layers $ 1,\dots,L $ are those of $ g_1 $ and $ g_2, $ bounded by $ \mathcal{A} $ on $ Z. $ Those of the appended layer lie in the range of $ \Theta, $ hence are bounded by $ \norm{\Theta}_\infty $ everywhere. Since \textup{(N3)} is a single bound over all hidden layers, it holds with budget $ \mathcal{A}\vee\norm{\Theta}_\infty. $
\end{proof}

\begin{definition}(Trial and test classes) \label{def:trial_test_classes}
    Fix a resolution $ N, $ a time degree $ K, $ a truncation radius $ R $ and a weight budget $ \Lambda = \Lambda(N), $ a sparsity budget $ P_0 = P_0(N) $ and an activation budget $ \mathcal{A} = \mathcal{A}(N). $ Let $ \varrho $ be the rectified cubic $ \varrho(u) = (u \lor 0)^3, $ let $ \Theta \in C^\infty (\mathbb{R}, [- 2 \bar{B}, 3 \bar{B}]) $ be a fixed 1-Lipschitz retraction with $ \Theta = \mathrm{id} $ on $ [- \bar{B}, 2 \bar{B}] $ and bounded second derivative, $ \kappa_\Theta := \| \Theta'' \|_\infty < \infty, $ and set $ W_N \asymp K N^d $ with $ W_N \geq d + 1. $ Fix also a spatial clip $ \chi_R : \mathbb{R}^d \rightarrow \mathbb{R}^d, $ acting coordinatewise by a fixed $ C^\infty $ function that is the identity on $ [ - (R - 1), R - 1 ], $ has modulus at most $ R, $ derivative at most 1 in modulus and bounded second derivative. For a function $ f $ on $ I \times \mathbb{R}^d $ write
    \begin{equation}
        f^\chi (t, x) := f( t, \chi_R (x) ).
    \end{equation}
    Like the gate $ \omega, $ the clip is a fixed map with no free parameters. Define
    \begin{align}
        & \mathcal{H}_N := \{ (\Theta \circ g)^\chi : g \in \mathcal{NN}_\varrho (W_N, L, \Lambda, P_0, \mathcal{A}), ~ \underset{I \times \mathbb{R}^d}{\sup} | \nabla h | \leq G, ~ \underset{I \times \mathbb{R}^d}{\sup} | \partial_t h | \leq G_t \}, \label{eq:h_n_class} \\
        & \Psi_N := \{ c \, \omega \, \phi^\chi : c \in \mathbb{R}, ~ \phi \in \mathcal{NN}_{\varrho, \Theta} (2 W_N, L + 1, \Lambda, 2 P_0 + 2, \mathcal{A} \lor 3 \bar{B}), \nonumber \\
        & \qquad\qquad\qquad ~ \underset{I \times \mathbb{R}^d}{\sup} (| \phi^\chi | + | \nabla \phi^\chi | + | \partial_t \phi^\chi |) \leq \Gamma \}, \label{eq:psi_n_class}
    \end{align}
    the constraint in \eqref{eq:h_n_class} being on $ h = (\Theta \circ g)^\chi $ itself. The three constants are
    \begin{equation} \label{eq:three_constants}
        G := \frac{4 \sqrt{d} \bar{B}}{\sigma^2_*}, ~ G_t = \frac{C_t d \bar{B} (R + 1)}{\sigma^4_*}, ~ \Gamma := 6 \bar{B} + 2 G + 2 G_t + \frac{6 \bar{B}}{\ell},
    \end{equation}
    with $ C_t $ absolute.
\end{definition}

\begin{proposition}(The classes have properties \ref{property:cone}--\ref{property:absorption}) \label{prop:classes_props}
    With Definition \ref{def:trial_test_classes}:
    \begin{enumerate}[label=(\roman*), leftmargin=2.2em]
        \item $ \Psi_N $ is a symmetric cone, i.e. \ref{property:cone} holds,
        \item $ \omega \cdot (\mathcal{H}_N - \mathcal{H}_N) \subseteq \Psi_N, $ i.e. \ref{property:absorption} holds,
        \item $ \mathcal{H}_N \cup \Psi_N \subset C^{1,2} (I \times \mathbb{R}^d), $ every element has space-time derivatives up to second order bounded on all $ I \times \mathbb{R}^d, $ and in particular the qualitative condition \eqref{eq:wei_bound} holds at every $ h \in \mathcal{H}_N, $ so the hypotheses of Lemmas \ref{lem:weak_form}--\ref{lem:continuity} and Theorems \ref{thm:energy_identity} and \ref{thm:weighted_energy_identity} are met there,
        \item $ h^* $ satisfies both sup constraints in \eqref{eq:h_n_class} on $ Z, $ and $ \Theta \circ h^* = h^*, $
        \item every $ \varphi_1 \in \mathcal{H}_N - \mathcal{H}_N $ satisfies $ \underset{I \times \mathbb{R}^d}{\sup} ( | \omega \varphi_1 | + | \nabla ( \omega \varphi_1) | + | \partial_t (\omega \varphi_1 ) | ) \leq \Gamma, $ so the inclusion in (ii) lands inside the constrained cone \eqref{eq:psi_n_class}.
    \end{enumerate}
\end{proposition}
\begin{proof}
    (i) The scalar $ c $ in \eqref{eq:psi_n_class} ranges over all of $ \mathbb{R}, $ so $ \Psi_N $ is closed under multiplication by any real number, in particular under negation and under nonnegative scaling.

    (ii) The clip commutes with differences, $ ( \Theta \circ g_1 )^\chi - (\Theta \circ g_2)^\chi = ((\Theta \circ g_1) - (\Theta \circ g_2))^\chi, $ so it suffices to difference the un-clipped networks. By Lemma \ref{lem:networks_difference}, $ (\Theta \circ g_1) - (\Theta \circ g_2) \in \mathcal{NN}_{\varrho, \Theta} ( 2 W_N, L+1, \Lambda, 2 P_0 + 2, \mathcal{A} \lor 3 \bar{B} ) $ -- the constraints in \eqref{eq:h_n_class} only shrink the set being differed, so they do not affect the inclusion. Multiplying by $ \omega $ and taking $ c = 1 $ in \eqref{eq:psi_n_class} gives the claim. Note that $ \omega $ is a fixed function of $ t, $ not a parameter: the test architecture is "network, then multiply the output by $ \omega(t), $" a gate with no free parameters.

    (iii) $ \varrho (u) = (u \lor 0)^3 $ has $ \varrho'(u) = 3 (u \lor 0)^2 $ and $ \varrho''(u) = 6 (u \lor 0), $ both continuous, so $ \varrho \in C^2 $ and any finite composition of affine maps and coordinatewise $ \varrho $ is $ C^2 $ in $ (t,x). $ $ \Theta $ and $ \chi_R $ are $ C^\infty $ and $ \omega $ is affine, so composing, clipping and gating preserves $ C^2. $ In particular the classes are $ C^{1,2}, $ which is what Theorem \ref{thm:energy_identity} requires and what a $ C^1 $ activation such as the rectified quadratic would not give. For the boundedness claim, every $ h = (\Theta \circ g)^\chi $ factors through the clip, so each $ \underset{I \times \mathbb{R}^d}{\sup} | \partial^\alpha h| $ with $ | \alpha | \leq 2 $ is a supremum of derivatives of $ \Theta \circ g $ over the compact set $ Z, $ multiplied by fixed powers of $ \| D \chi_R \|_\infty \leq 1 $ and $ \| D^2 \chi_R \|_\infty < \infty, $ the derivatives of $ \Theta \circ g $ up to second order are continuous, because $ \varrho \in C^2 $ and $ \Theta \in C^\infty, $ hence bounded on the compact set $ Z. $ 

    (iv) All three constraints are verified on $ Z, $ which is what a clipped function needs. By \eqref{eq:h_der_bounds}, $ | \partial_{x_k} h^* | \leq 2 \bar{B} \sigma^{-2}_{T-t}, $ so $ | \nabla h^* | \leq 2 \sqrt{d} \bar{B} \sigma_*^{-2} = G / 2. $ For the time derivative, harmonicity \eqref{eq:harmonicity} gives $ \partial_t h^* = - b \cdot  \nabla h^* - \Delta h^*, $ and on $ Z $ we have $ |x| \leq \sqrt{d} R $ by \eqref{eq:Q_R_def}, so $ |b| = |x + 2s| \leq |x| + 2 \sigma^{-2}_{T-t} ( |x| + \sqrt{d}) \leq 3 \sqrt{d} \sigma_*^{-2} (R+1) $ by \eqref{eq:score_identity_1} and $ \sigma_* \leq 1, $ while $ | \Delta h^* | \leq d C_2 \bar{B} \sigma_*^{-4} $ by \eqref{eq:h_der_bounds}. Hence $ | \partial_t h^* | \leq 6 d \bar{B} \sigma_*^{-4} (R + 1) + d C_2 \bar{B} \sigma_*^{-4} \leq C d \bar{B} \sigma^{-4}_* (R + 1) = G_t / 2 $ for $ C_t $ large enough. Finally, $ 0 < h^* \leq \bar{B} $ lies in the range on which $ \Theta $ is the identity.

    (v) For $ h, h' \in \mathcal{H}_N $ and $ \varphi_1 = h - h', $ we have, on all of $ I \times \mathbb{R}^d $ now that \eqref{eq:h_n_class} is global, $ | \varphi_1 | \leq 6 \bar{B} $ (both lie in $ [- 2\bar{B}, 3 \bar{B}] $), $ | \nabla \varphi_1 | \leq 2 G $ and $ | \partial_t \varphi_1 | \leq 2 G_t. $ Since $ 0 \leq \omega \leq 1 $ and $ | \dot{\omega} | = \ell^{-1}, $ the function $ \omega \varphi_1 $ obeys $ | \omega \varphi_1 | \leq 6 \bar{B}, ~ | \nabla ( \omega \varphi) | = \omega | \nabla \varphi_1 | \leq 2 G $ and $ | \partial_t (\omega \varphi_1) | \leq 2 G_t + 6 \bar{B} / \ell, $ and the sum of the three is at most $ \Gamma $ by \eqref{eq:three_constants}.
\end{proof}

The statistical analysis measures a test function on two scales: its size in $ \mathcal{V}, $ and the size of the scalar that generates it. For $ \psi \in \Psi_N $ put
\begin{equation} \label{eq:bullet_norm}
    \| \psi \|_\bullet := \inf \{ | c | : \psi = c \, \omega \, \phi^\chi \text{ as in } \eqref{eq:psi_n_class} \},
\end{equation}
the cone scale of $ \psi. $ It is symmetric and positively homogeneous -- the only two properties used below -- and it dominates the supremum norm globally, $ \underset{I \times \mathbb{R}^d}{\sup} ( | \psi | + | \nabla \psi | ) \leq \Gamma \| \psi \|_\bullet, $ because every representation $ \psi = c \, \omega \, \phi^\chi $ gives $ | \psi | + | \nabla \psi | = | c | \, \omega \, ( | \phi^\chi | + | \nabla \phi^\chi | ) \leq \Gamma | c | $ by \eqref{eq:psi_n_class} and $ 0 \leq \omega \leq 1, $ so that the quantity
\begin{equation}
    \Gamma_\mathcal{V} : = \sup \{ \| \psi \|_\mathcal{V} : \psi \in \Psi_N, \| \psi \|_\bullet \leq 1 \}
\end{equation}
satisfies $ \Gamma_\mathcal{V} \leq \Gamma, ~ \rho_t $ being a probability measure. By parts (ii) and (v) of Proposition \ref{prop:classes_props}, every displacement obeys
\begin{equation} \label{eq:omega_phi_bullet_bound}
    \| \omega \varphi_1 \|_\bullet \leq 1, \quad \text{for all } \varphi_1 \in \mathcal{H}_N - \mathcal{H}_N,
\end{equation}
since $ c = 1 $ is an admissible representation. The two scales are genuinely different: under the mild conditions of Remark \ref{rem:v_ball} the ratio $ \sup ( | \psi | + | \nabla \psi | ) / \| \psi \|_\mathcal{V}, $ and a fortiori $ \| \psi \|_\bullet / \| \psi \|_\mathcal{V}, $ is unbounded on $ \Psi_N, $ a generator being able to be small in $ L^2 (\nu \otimes \rho) $ without being small in the supremum norm. This is why the deviation bounds of Section \ref{app:statistical} are two-norm bounds, why the estimator carries a ridge at the scale $ \| \cdot \|_\bullet, $ and why the complexity of the test class is measured in \eqref{eq:eff_dim} on the unit ball of $ \| \cdot \|_\bullet $ and not on that of $ \| \cdot \|_\mathcal{V}. $

\begin{remark}(Why the rectified cubic)
    Three requirements pull in different directions and $ \varrho (u) = ( u \lor 0)^3 $ meets all of them. The energy identity needs $ C^{1,2} $ trial functions, which excludes ReLU (not $ C^1 $) and the rectified quadratic (only $ C^1 $). Exact emulation of products is needed for the space-time construction of Lemma \ref{lem:space_time_approx}, and it is available here, since
    \begin{equation} \label{eq:multip_trick}
        (u+1)^3 - (u-1)^3 = 6 u^2 + 2, \quad u^3 = \varrho (u) - \varrho (-u),
    \end{equation}
    a two-node layer computes $ u \mapsto u^2 $ exactly, hence $ uv = \frac{1}{4} [(u+v)^2 - (u-v)^2] $ exactly, hence any polynomial exactly. Smooth activations such as tanh satisfy the first requirement but emulate products only approximately.
\end{remark}

The quantity that the statistical analysis consumes is the metric entropy of first-order jets. For a function $ F $ on $ I \times \mathbb{R}^d $ with values in a Euclidean space write $ \| F \|_\infty := \sup_{I \times \mathbb{R}^d} | F |, $ with $ | \cdot | $ the Euclidean norm, and put
\begin{align}
    & \Psi^\bullet_N := \{ \psi \in \Psi_N : \| \psi \|_\bullet \leq 1 \}, \nonumber \\
    & J \mathcal{H}_N := \{ (h, \partial_t h, \nabla h) : h \in \mathcal{H}_N \}, \qquad J \Psi^\bullet_N := \{ (\psi, \nabla \psi) : \psi \in \Psi^\bullet_N \}, \label{eq:jets}
\end{align}
\begin{equation} \label{eq:eff_dim}
    \mathbb{V}_N := 1 \lor \sup_{0 < \varepsilon \leq 1} \frac{\max \{ \log N (\varepsilon, J \mathcal{H}_N, \| \cdot \|_\infty), \, \log N (\varepsilon, J \Psi^\bullet_N, \| \cdot \|_\infty) \}}{\log (e / \varepsilon)},
\end{equation}
the normalization by $ \log (e / \varepsilon) $ being the one under which a $P$-dimensional parametric class has $ \mathbb{V}_N \asymp P. $ Three features of \eqref{eq:eff_dim} are used below. First, jets are covered jointly, because the summands of Section \ref{app:statistical} depend on $ (\partial_t h, \nabla h) $ and on $ (\psi, \nabla \psi) $ simultaneously, $ \partial_t h $ enters through the bilinear form and $ \nabla \psi $ because the adversary is normed in $ H^1. $ Second, no lower cutoff is imposed on $ \varepsilon, $ so that every scale at which an entropy integral is evaluated in Section \ref{app:statistical} is covered. Third, the free scalar $ c $ of \eqref{eq:psi_n_class} is removed by the cone scale: on $ \Psi^\bullet_N $ one has $ \sup ( | \psi | + | \nabla \psi | ) \leq \Gamma. $ The unit ball of $ \| \cdot \|_\mathcal{V} $ would define the same projected residual, since every ray of the cone meets both balls and $ \psi \mapsto A(h, \psi) / \| \psi \|_{\mathcal{V}, \tau} $ is invariant under positive scaling, but it is not a totally bounded index set: by Remark \ref{rem:v_ball} its gradients are in general unbounded in $ \| \cdot \|_\infty, $ so its covering numbers are infinite at every radius.

\begin{remark}(The unit $\mathcal{V}$-ball of $ \Psi_N $ is not totally bounded) \label{rem:v_ball}
    Suppose that $ \chi_R $ is nondecreasing and that $ R \geq 2, $ $ W_N \geq 3, $ $ 2 P_0 + 2 \geq 40 L + 10 $ and $ \mathcal{A} \geq (R + 1)^3. $ Then
    \begin{equation} \label{eq:v_ball_unbounded}
        \sup \{ \| \nabla \psi \|_\infty : \psi \in \Psi_N, \, \| \psi \|_{\mathcal{V}} \leq 1 \} = \infty.
    \end{equation}
    Consequently $ \{ \nabla \psi : \psi \in \Psi_N, \, \| \psi \|_{\mathcal{V}} \leq 1 \} $ has infinite covering numbers in $ \| \cdot \|_\infty $ at every radius, and, since $ \| \nabla \psi \|_\infty \leq \Gamma \| \psi \|_\bullet, $ the ratio $ \| \psi \|_\bullet / \| \psi \|_{\mathcal{V}} $ is unbounded on $ \Psi_N \setminus \{ 0 \}. $
\end{remark}
\begin{proof}
    Let $ B $ be the centred cardinal cubic B-spline, supported on $ [-2, 2], $ which satisfies $ 0 \leq B \leq 2/3, $ $ | B' | \leq 2/3 $ and $ B'(1) = -1/2, $ and put $ \kappa_B := \| B \|^2_{L^2 (\mathbb{R})} + \| B' \|^2_{L^2 (\mathbb{R})} < \infty. $ Its truncated-power representation gives, for $ \varepsilon > 0 $ and $ y \in \mathbb{R}, $
    \begin{equation*}
        g_\varepsilon (y) := \varepsilon^3 B (y / \varepsilon) = \sum_{k = -2}^{2} a_k \, \varrho (y - k \varepsilon), \qquad (a_{-2}, \ldots, a_2) := \tfrac{1}{6} (1, -4, 6, -4, 1).
    \end{equation*}
    Fix $ 0 < \varepsilon \leq \min \{ 1/4, (3 \bar{B})^{1/3}, \Gamma^{1/2} \}. $ Then $ 0 \leq g_\varepsilon \leq 2 \varepsilon^3 / 3 \leq \min \{ 1, 2 \bar{B} \}, $ $ \mathrm{supp} \, g_\varepsilon \subseteq [-2 \varepsilon, 2 \varepsilon] $ and $ \| g_\varepsilon \|_\infty + \| g'_\varepsilon \|_\infty \leq \tfrac{2}{3} (\varepsilon^3 + \varepsilon^2) \leq \varepsilon^2 \leq \Gamma. $

    Realization. We exhibit $ \phi_\varepsilon \in \mathcal{NN}_{\varrho, \Theta} (2 W_N, L + 1, \Lambda, 2 P_0 + 2, \mathcal{A} \lor 3 \bar{B}) $ with $ \phi_\varepsilon (t, x) = g_\varepsilon (x_1) $ on $ I \times \mathbb{R}^d. $ Hidden layer $ 1 $ consists of the five units $ \varrho (x_1 - k \varepsilon), \, | k | \leq 2, $ so that the affine combination $ \sum_k a_k \varrho (x_1 - k \varepsilon) $ of its outputs is $ u := g_\varepsilon (x_1). $ If $ L \geq 2, $ each of the hidden layers $ 2, \ldots, L $ consists of the six units $ \varrho (u + 1), \varrho (- u - 1), \varrho (u - 1), \varrho (1 - u), \varrho (u), \varrho (-u), $ whose pre-activations are affine combinations of the outputs of the previous layer, and $ u $ is recovered exactly from them by
    \begin{equation*}
        6 u = \varrho (u+1) - \varrho (-u-1) + \varrho (u-1) - \varrho (1-u) - 2 \varrho (u) + 2 \varrho (-u),
    \end{equation*}
    which follows from $ \varrho (v) - \varrho (-v) = v^3 $ and $ (u+1)^3 + (u-1)^3 - 2 u^3 = 6 u. $ Hidden layer $ L + 1 $ is the single unit $ \Theta (u), $ read out with weight $ 1 $ and bias $ 0. $ Since $ u \in [0, 2 \bar{B}] $ and $ \Theta = \mathrm{id} $ there, $ \phi_\varepsilon (t, x) = g_\varepsilon (x_1) $ everywhere. All entries lie in $ \{ 0, \pm 1/6, \pm 1/3, \pm 2/3, \pm 1, \pm \varepsilon, \pm 2 \varepsilon \} $ and are therefore bounded by $ 1 \leq \Lambda. $ Hidden widths are at most $ 6 \leq 2 W_N. $ The number of nonzero entries is $ 15 $ for $ L = 1 $ and $ 9 + 34 + 40 (L - 2) + 6 + 1 $ for $ L \geq 2, $ in both cases at most $ 40 L + 10 \leq 2 P_0 + 2, $ and on $ Z, $ where $ | x_1 | \leq R, $ the units of layer $ 1 $ take values in $ [0, (R + 1)^3], $ those of layers $ 2, \ldots, L $ in $ [0, 8] $ because $ 0 \leq u \leq 1, $ and the last one in $ [0, 2 \bar{B}], $ all within $ \mathcal{A} \lor 3 \bar{B}. $

    Clip and constraint. If $ | y | \leq R - 1 $ then $ \chi_R (y) = y. $ If $ y > R - 1 $ then $ y > 2 \varepsilon $ and, $ \chi_R $ being nondecreasing, $ \chi_R (y) \geq R - 1 > 2 \varepsilon, $ symmetrically for $ y < -(R-1). $ In all cases $ g_\varepsilon (\chi_R (y)) = g_\varepsilon (y), $ so $ \phi_\varepsilon^\chi (t, x) = g_\varepsilon (x_1) $ and $ | \phi^\chi_\varepsilon | + | \nabla \phi^\chi_\varepsilon | + | \partial_t \phi^\chi_\varepsilon | \leq \Gamma. $ Thus $ \psi_\varepsilon (t, x) := \omega(t) \, g_\varepsilon (x_1) $ belongs to $ \Psi_N, $ with $ c = 1. $

    Two norms. The $ x_1 $-marginal of $ \rho_t, $ the law of $ \mu_{T-t} X_0 + \sigma_{T-t} \xi $ with $ \xi \sim \mathcal{N}(0, I_d) $ independent of $ X_0, $ has a positive density bounded by $ (2 \pi)^{-1/2} \sigma_{T-t}^{-1} \leq (2 \pi)^{-1/2} \sigma_*^{-1} $ for $ t \in I. $ Since $ 0 \leq \omega \leq 1 $ and $ \varepsilon \leq 1, $
    \begin{equation*}
        0 < \| \psi_\varepsilon \|^2_{\mathcal{V}} = \int_I \omega^2 \left( \| g_\varepsilon (x_1) \|^2_{\rho_t} + \| g'_\varepsilon (x_1) \|^2_{\rho_t} \right) \nu(dt) \leq \frac{\varepsilon^7 \| B \|^2_{L^2 (\mathbb{R})} + \varepsilon^5 \| B' \|^2_{L^2 (\mathbb{R})}}{\sqrt{2 \pi} \, \sigma_*} \leq \frac{\kappa_B \, \varepsilon^5}{\sqrt{2 \pi} \, \sigma_*},
    \end{equation*}
    with positivity because $ g_\varepsilon \not\equiv 0 $ and $ \omega > 0 $ on $ [t_0, t_1). $ On the other hand $ \| \nabla \psi_\varepsilon \|_\infty \geq \omega (t_0) \, | g'_\varepsilon (\varepsilon) | = \varepsilon^2 | B'(1) | = \varepsilon^2 / 2. $ Hence $ \bar{\psi}_\varepsilon := \psi_\varepsilon / \| \psi_\varepsilon \|_{\mathcal{V}} \in \Psi_N $ has $ \| \bar{\psi}_\varepsilon \|_{\mathcal{V}} = 1 $ and
    \begin{equation*}
        \| \nabla \bar{\psi}_\varepsilon \|_\infty \geq \tfrac{1}{2} (2 \pi)^{1/4} \sigma_*^{1/2} \kappa_B^{-1/2} \, \varepsilon^{-1/2} \longrightarrow \infty \quad (\varepsilon \downarrow 0),
    \end{equation*}
    which is \eqref{eq:v_ball_unbounded}. The first consequence holds because a set covered by finitely many $ \| \cdot \|_\infty $-balls of finite radius is bounded, the second because $ \| \bar{\psi}_\varepsilon \|_\bullet \geq \Gamma^{-1} \| \nabla \bar{\psi}_\varepsilon \|_\infty $ by \eqref{eq:bullet_norm}.
\end{proof}

We first record three elementary facts about covering numbers. Throughout, for a subset $ S $ of a metric space $ (X, \mathrm{d}), $ $ N(\varepsilon, S, \mathrm{d}) $ denotes the minimal cardinality of an $\varepsilon$-net of $ S, $ that is, of a set $ \mathcal{C} \subseteq X $ such that every point of $ S $ lies at distance at most $ \varepsilon $ from $ \mathcal{C}, $ in particular $ N(\varepsilon, S, \mathrm{d}) \leq N(\varepsilon, S', \mathrm{d}) $ whenever $ S \subseteq S'. $

\begin{lemma}(Three covering facts) \label{lem:three_covering}
    \begin{enumerate}[label=(E\arabic*), leftmargin=2.8em]
        \item (Cube) For $ Q := [ - \Lambda, \Lambda]^m $ and $ 0 < \delta \leq \Lambda, $ one has $ N(\delta, Q, \| \cdot \|_\infty) \leq \lceil \Lambda / \delta \rceil^m \leq ( 2 \Lambda / \delta)^m, $ \label{cond:cube}
        \item (Internal nets) If $ S \subseteq X $ and $ N(\delta / 2, S, \mathrm{d}) < \infty, $ then $ S $ contains a $\delta$-net of $ S $ consisting of itself points and of cardinality at most $ N(\delta / 2, S, \mathrm{d}), $ \label{cond:internal_nets}
        \item (Sparse supports) For $ 1 \leq m \leq P, $ the number of subsets of $ \{ 1, \ldots, P \} $ of cardinality at most $ m $ is at most $ (e P / m)^m. $ \label{cond:sparse_supports}
    \end{enumerate}
\end{lemma}
\begin{proof}
    \ref{cond:cube} Partition $ [ - \Lambda, \Lambda ] $ into $ \lceil \Lambda / \delta \rceil $ intervals of length at most $ 2 \delta $ and take their midpoints. The product grid has the stated cardinality and every point of $ Q $ lies within $ \delta $ of a grid point in each coordinate, hence within $ \delta $ in $ \| \cdot \|_\infty. $ The final inequality uses $ \lceil u \rceil \leq 2 u $ for $ u \geq 1. $

    \ref{cond:internal_nets} Call $ \mathcal{S} \subseteq S $ $\delta$-separated if distinct points of $ \mathcal{S} $ are at distance larger than $ \delta. $ Let $ \mathcal{C} $ be a $ (\delta / 2) $-net of $ S $ of minimal size and assign to each point of a $\delta$-separated $ \mathcal{S} $ a point of $ \mathcal{C} $ within $ \delta / 2 $ of it. The assignment is injective, since two points with the same image are within $ \delta $ of each other. Hence every $\delta$-separated subset of $ S $ has at most $ | \mathcal{C} | = N(\delta / 2, S, \mathrm{d}) $ points, and one of maximal cardinality is maximal under inclusion. By maximality no point of $ S $ is at distance larger than $ \delta $ from all of $ \mathcal{S}, $ which is exactly the $\delta$-net property.

    \ref{cond:sparse_supports} $ \sum_{k \leq m} \binom{P}{k} \leq (e P / m)^m. $ Indeed, for $ \lambda := m / P \leq 1, ~ \sum_{k \leq m} \binom{P}{k} \lambda^k \leq (1 + \lambda)^P \leq e^{\lambda P} = e^m, $ while every term on the left with $ k \leq m $ satisfies $ \binom{P}{k} \lambda^k \geq \binom{P}{k} \lambda^m. $ Dividing by $ \lambda^m $ gives the claim.
\end{proof}

Content of the next lemma is that the map from the parameters to functions is Lipschitz with a constant that is exponential in depth but polynomial in everything else. 

\begin{lemma}(Complexity) \label{lem:complexity}
    Let $ \mathcal{F} $ be either $ \mathcal{NN}_\varrho (W, L, \Lambda, P_0, \mathcal{A}) $ with $ \varrho (u) = (u \lor 0)^3, $ or $ \mathcal{NN}_{\varrho, \Theta} (W, L, \Lambda, P_0, \mathcal{A}) $ with $ \Theta $ as in Definition \ref{def:trial_test_classes}. Let $ \mathrm{m} $ be as in \eqref{eq:mathrm_m_def} in the first case and $ \mathrm{m} := 3 \Lambda W (\mathcal{A}_{\mathrm{max}} + 1) \lor \kappa_\Theta $ in the second, put $ \mathrm{L}_* := 2 L \mathrm{m}^{3 L + 1} $ and $ \mathrm{L}'_* := 12 L^2 \mathrm{m}^{9 L + 3}, $ and write $ \nabla_z := (\partial_t, \nabla_x) $ for the space-time gradient and $ | \cdot |_1, | \cdot | $ for the $ \ell^1 $ and the Euclidean norm. Then:
    \begin{enumerate}[label=(\roman*), leftmargin=2.2em]
        \item (Parameter Lipschitz bound) For any two parameter vectors $ \vartheta, \tilde{\vartheta} \in \mathbb{R}^P $ of elements of $ \mathcal{F}, $ realized in the common layout of the proof,
        \begin{equation} \label{eq:lipschitz_bound}
            \sup_{z \in Z} | f_\vartheta (z) - f_{\tilde{\vartheta}} (z) | \leq \mathrm{L}_* \| \vartheta - \tilde{\vartheta} \|_\infty, \qquad \sup_{z \in Z} | \nabla_z f_\vartheta (z) - \nabla_z f_{\tilde{\vartheta}} (z) |_1 \leq \mathrm{L}'_* \| \vartheta - \tilde{\vartheta} \|_\infty.
        \end{equation} \label{cond:lipschitz_bound}

        \item (Entropy of jets) Suppose $ 1 \leq P_0 \leq P. $ For every nonempty $ \mathcal{F}' \subseteq \mathcal{F} $ and $ 0 < \varepsilon \leq 1, $ the jet class $ J \mathcal{F}' := \{ (f, \nabla_z f) : f \in \mathcal{F}' \}, $ normed by $ \| F \|_{\infty, Z} := \sup_{z \in Z} | F(z) |, $ has an $\varepsilon$-net consisting of jets of elements of $ \mathcal{F}' $ and of cardinality at most $ (e P / P_0)^{P_0} (8 \Lambda \mathrm{L}'_* / \varepsilon)^{P_0}, $ in particular
        \begin{equation} \label{eq:entropy}
            \log N( \varepsilon, J \mathcal{F}', \| \cdot \|_{\infty, Z}) \leq P_0 \left[ \log \frac{e P}{P_0} + \log (8 \Lambda \mathrm{L}'_*) + \log \frac{1}{\varepsilon} \right].
        \end{equation} \label{cond:entropy}

        \item (Effective dimension) Let $ \mathcal{H}_N, \Psi_N $ be as in Definition \ref{def:trial_test_classes}, let $ P_{\mathcal{H}} $ be the number of entries of a parameter vector of $ \mathcal{NN}_\varrho (W_N, L, \Lambda, P_0, \mathcal{A}) $ in that layout, and assume $ 1 \leq P_0 \leq P_{\mathcal{H}}, $ a larger sparsity budget being no constraint. Put
        \begin{equation} \label{eq:bar_constants}
            \bar{P}_0 := 2 P_0 + 2, \quad \bar{P} := (L + 2) (2 W_N + 1)^2, \quad \bar{\mathrm{m}} := 6 \Lambda W_N \big( (\mathcal{A} \lor 3 \bar{B} \lor R_*) + 1 \big) \lor \kappa_\Theta.
        \end{equation}
        Then $ \mathbb{V}_N $ of \eqref{eq:eff_dim} satisfies
        \begin{equation} \label{eq:eff_dim_bound}
            \mathbb{V}_N \leq \log (1 + 4 \Gamma) + \bar{P}_0 \left[ \log \frac{e \bar{P}}{\bar{P}_0} + \log \big( 384 \Lambda (L + 1)^2 \big) + (9 L + 12) \log \bar{\mathrm{m}} \right].
        \end{equation}
        In particular, if $ P_0 \lesssim K N^d, \, W_N \lesssim K N^d, \, 2 \leq N \leq n, \, L \lesssim \log N $ and $ \Lambda, \mathcal{A} \leq \mathrm{poly}(n), $ then
        \begin{equation} \label{eq:eff_dim_bound_particular}
            \mathbb{V}_N \lesssim K N^d \log N \log n \lesssim N^d \mathrm{polylog}(n).
        \end{equation} \label{cond:eff_dim}
    \end{enumerate}
\end{lemma}
\begin{proof}
    \ref{cond:lipschitz_bound} Layout. Padding a hidden layer with a unit whose incoming weights, bias and outgoing weights vanish changes neither $ f_\vartheta $ nor \ref{cond:vartheta_entries}--\ref{cond:bounded_activations}, since $ \varrho(0) = \Theta(0) = 0. $ We therefore realize every element of $ \mathcal{F} $ with all hidden widths equal to $ W, $ so that $ A_0 \in \mathbb{R}^{W \times (d+1)}, $ $ A_\ell \in \mathbb{R}^{W \times W} $ for $ 1 \leq \ell \leq L - 1 $ and $ A_L \in \mathbb{R}^{1 \times W}, $ and a parameter vector is an element of $ \mathbb{R}^P $ with
    \begin{equation} \label{eq:layout_count}
        P = W (d + 2) + (L - 1) W (W + 1) + W + 1.
    \end{equation}
    Write $ \sigma_\ell $ for the activation producing $ x^{(\ell + 1)}, $ so that $ \sigma_\ell = \varrho $ except for $ \sigma_{L-1} = \Theta $ in the second case. The only properties of the activations used below are, for $ u, u' \in [ - \mathrm{m}, \mathrm{m} ], $
    \begin{equation} \label{eq:activation_bounds}
        | \sigma_\ell' (u) | \leq 3 \mathrm{m}^2, \qquad | \sigma_\ell' (u) - \sigma_\ell' (u') | \leq 6 \mathrm{m} | u - u' |.
    \end{equation}
    For $ \varrho $ they follow from $ \varrho'(u) = 3 (u \lor 0)^2 $ and $ \varrho''(u) = 6 (u \lor 0), $ and for $ \Theta $ from $ | \Theta' | \leq 1 \leq 3 \mathrm{m}^2 $ and $ | \Theta'' | \leq \kappa_\Theta \leq \mathrm{m}. $

    Write $ x^{(\ell)}, \tilde{x}^{(\ell)} $ for the hidden activations \eqref{eq:network_f} of $ f_\vartheta, f_{\tilde{\vartheta}}, $ set $ v^{(\ell)} := A_\ell x^{(\ell)} + b_\ell $ and $ \tilde{v}^{(\ell)} := \tilde{A}_\ell \tilde{x}^{(\ell)} + \tilde{b}_\ell $ and put
    \begin{equation}
        e_\ell := \underset{z \in Z}{\sup} \| x^{(\ell)} (z) - \tilde{x}^{(\ell)} (z) \|_\infty, \quad \delta := \| \vartheta - \tilde{\vartheta} \|_\infty.
    \end{equation}
    Since $ x^{(0)} = \tilde{x}^{(0)} = z, $ we have $ e_0 = 0. $

    First, the pre-activations are bounded. By \ref{cond:hidden_widths}-\ref{cond:bounded_activations}, each row of $ A_\ell $ has at most $ W $ entries of modulus at most $ \Lambda $ (for $ \ell = 0 $ because $ d + 1 \leq W $), and $ \| x^{(\ell)} \|_\infty \leq \mathcal{A}_{\mathrm{max}} $ for every $ \ell $ (including $ \ell = 0, $ where $ \| z \|_\infty \leq R_* $ ), so
    \begin{equation} \label{eq:v_ell_bound}
        \| v^{(\ell)} \|_\infty \leq \Lambda ( W \mathcal{A}_{\mathrm{max}} + 1) \leq \mathrm{m},
    \end{equation}
    and identically for $ \tilde{v}^{(\ell)}, $ since $ f_{\tilde{\vartheta}} $ also lies in $ \mathcal{F}. $ This is the step that \ref{cond:bounded_activations} exists to supply.

    Second, one layer of the recursion. Splitting the difference into a term from the activations and a term from the parameters
    \begin{equation}
        \| v^{(\ell)} - \tilde{v}^{(\ell)} \|_\infty \leq \| A_\ell ( x^{(\ell)} - \tilde{x}^{(\ell)} ) \|_\infty + \| (A_\ell - \tilde{A}_\ell )\tilde{x}^{(\ell)} \|_\infty + \| b_\ell - \tilde{b}_\ell \|_\infty \leq \Lambda W e_\ell + \delta (W \mathcal{A}_{\mathrm{max}} + 1).
    \end{equation}
    On the interval $ [- \mathrm{m}, \mathrm{m} ] $ containing both $ v^{(\ell)} $ and $ \tilde{v}^{(\ell)} $ by \eqref{eq:v_ell_bound}, the mean value theorem and \eqref{eq:activation_bounds} give $ | \sigma_\ell (u) - \sigma_\ell (u') | \leq 3 \mathrm{m}^2 | u - u'|. $ Hence
    \begin{equation}
        e_{\ell + 1} \leq 3 \mathrm{m}^2 \left[ \Lambda W e_\ell + \delta (W \mathcal{A}_{\mathrm{max}} + 1) \right] = \kappa e_\ell + c_0 \delta, \quad \kappa := 3 \mathrm{m}^2 \Lambda W, \quad c_0 := 3 \mathrm{m}^2 (W \mathcal{A}_{\mathrm{max}} + 1).
    \end{equation}
    Note $ \kappa \leq \mathrm{m}^3 $ and $ c_0 \leq \mathrm{m}^3, $ both by the definition of $ \mathrm{m} $, and $ \kappa \geq 1. $

    Third, unroll. From $ e_0 = 0 $ and $ e_{\ell + 1} \leq \kappa e_\ell + c_0 \delta $ one gets, for $ 1 \leq \ell \leq L, $ $ e_\ell \leq c_0 \delta \sum_{j = 0}^{\ell-1} \kappa^j \leq c_0 \delta \ell \kappa^{\ell-1} \leq \delta L \mathrm{m}^{3 L}. $ Finally the output layer contributes 
    \begin{equation}
        | f_\vartheta (z) - f_{\tilde{\vartheta}}(z) | \leq \Lambda W e_L + \delta (W \mathcal{A}_{\mathrm{max}} + 1) \leq \delta ( \mathrm{m} \cdot L \mathrm{m}^{3 L} + \mathrm{m}) \leq 2 L \mathrm{m}^{3 L + 1} \delta,
    \end{equation}
    which is the first bound in \eqref{eq:lipschitz_bound}.

    By the chain rule applied to \eqref{eq:network_f}, $ \nabla_z f_\vartheta = A_L D_{L-1} A_{L-1} \cdots D_0 A_0 $ with $ D_\ell := \mathrm{diag} (\sigma_\ell' (v^{(\ell)})). $ Each factor is bounded in the $ \ell_\infty \rightarrow \ell_\infty $ operator norm, the maximal absolute row sum: $ \| A_\ell \| \leq \Lambda W \leq \mathrm{m} $ and $ \| D_\ell \| \leq 3 \mathrm{m}^2 \leq \mathrm{m}^3 $ by \eqref{eq:v_ell_bound} and \eqref{eq:activation_bounds}. Each factor is also stable: $ \| A_\ell - \tilde{A}_\ell \| \leq W \delta \leq \mathrm{m} \delta, $ while \eqref{eq:activation_bounds} gives $ \| D_\ell - \tilde{D}_\ell \| \leq 6 \mathrm{m} \| v^{(\ell)} - \tilde{v}^{(\ell)} \|_\infty \leq 6 \mathrm{m} ( \Lambda W e_\ell + \delta \mathrm{m}) \leq 12 L \mathrm{m}^{3 L + 2} \delta \leq 4 L \mathrm{m}^{3 L + 3} \delta, $ the last step by $ \mathrm{m} \geq 3. $ Writing the product difference as a telescopic sum
    \begin{equation}
        \prod_{i=1}^{2L + 1} M_i - \prod_{i=1}^{2 L + 1} \tilde{M}_i = \sum_{i=1}^{2 L + 1} \left( \prod_{j < i} M_j \right) (M_i - \tilde{M}_i) \left( \prod_{j > i} \tilde{M}_j \right),
    \end{equation}
    and bounding every retained factor by $ \mathrm{m}^3 $ and every difference by $ 4 L \mathrm{m}^{3 L + 3} \delta, $ we obtain $ \| \nabla_z f_\vartheta - \nabla_z f_{\tilde{\vartheta}} \|_\infty \leq (2 L + 1) \mathrm{m}^{6 L} \cdot 4 L \mathrm{m}^{3 L + 3} \delta \leq 12 L^2 \mathrm{m}^{9 L + 3} \delta $ using $ 2 L + 1 \leq 3 L, $ at every $ z \in Z. $ The $ \ell_\infty \rightarrow \ell_\infty $ operator norm of a row vector is its $ \ell^1 $ norm, so this is the second bound in \eqref{eq:lipschitz_bound}. 

    \ref{cond:entropy} Let $ \mathcal{P}' $ be the set of parameter vectors, in the layout above, of elements of $ \mathcal{F}'. $ For $ S \subseteq \{ 1, \ldots, P \} $ with $ | S | = P_0 $ let $ \mathcal{P}'_S $ consist of those $ \vartheta \in \mathcal{P}' $ whose entries outside $ S $ vanish. By \ref{cond:vartheta_entries} and $ P_0 \leq P $ every element of $ \mathcal{P}' $ lies in some $ \mathcal{P}'_S, $ and $ \mathcal{P}'_S $ is contained in $ Q_S := \{ \vartheta : \vartheta_i = 0 \text{ for } i \notin S, \, \| \vartheta \|_\infty \leq \Lambda \}, $ which is isometric to $ [ - \Lambda, \Lambda ]^{P_0}. $ Put $ \delta := \varepsilon / (2 \mathrm{L}'_*), $ so that $ \delta / 2 \leq 1 \leq \Lambda. $ By \ref{cond:cube}, $ N(\delta / 2, \mathcal{P}'_S, \| \cdot \|_\infty) \leq N (\delta / 2, Q_S, \| \cdot \|_\infty) \leq (4 \Lambda / \delta)^{P_0}, $ so by \ref{cond:internal_nets} each nonempty $ \mathcal{P}'_S $ contains a $\delta$-net $ \mathcal{S}_S \subseteq \mathcal{P}'_S $ with $ | \mathcal{S}_S | \leq (4 \Lambda / \delta)^{P_0} = (8 \Lambda \mathrm{L}'_* / \varepsilon)^{P_0}. $ For $ \vartheta \in \mathcal{P}'_S $ and $ \tilde{\vartheta} \in \mathcal{S}_S $ with $ \| \vartheta - \tilde{\vartheta} \|_\infty \leq \delta, $ part \ref{cond:lipschitz_bound}, $ | \cdot | \leq | \cdot |_1 $ and $ \mathrm{L}_* \leq \mathrm{L}'_* $ give, for every $ z \in Z, $
    \begin{equation*}
        \big| (f_\vartheta, \nabla_z f_\vartheta)(z) - (f_{\tilde{\vartheta}}, \nabla_z f_{\tilde{\vartheta}})(z) \big| \leq | f_\vartheta (z) - f_{\tilde{\vartheta}} (z) | + | \nabla_z f_\vartheta (z) - \nabla_z f_{\tilde{\vartheta}} (z) |_1 \leq (\mathrm{L}_* + \mathrm{L}'_*) \delta \leq \varepsilon.
    \end{equation*}
    Hence the jets of $ f_{\tilde{\vartheta}}, \, \tilde{\vartheta} \in \bigcup_S \mathcal{S}_S, $ form an $\varepsilon$-net of $ J \mathcal{F}' $ consisting of jets of elements of $ \mathcal{F}'. $ By \ref{cond:sparse_supports} there are at most $ (e P / P_0)^{P_0} $ sets $ S, $ which gives the cardinality bound and, taking logarithms, \eqref{eq:entropy}.

    \ref{cond:eff_dim} Put $ \mathfrak{F} := \mathcal{NN}_{\varrho, \Theta} (2 W_N, L + 1, \Lambda, \bar{P}_0, \mathcal{A} \lor 3 \bar{B}), $ the architecture of $ \Psi_N. $ Its constant $ \mathrm{m} $ in part \ref{cond:lipschitz_bound} is $ \bar{\mathrm{m}} $ of \eqref{eq:bar_constants} and its depth is $ L + 1, $ so its constant $ \mathrm{L}'_* $ is $ \bar{\mathrm{L}}' := 12 (L + 1)^2 \bar{\mathrm{m}}^{9 L + 12}. $ By \eqref{eq:layout_count} and $ d + 1 \leq W_N, $ its number of parameter entries $ P_{\mathfrak{F}} $ satisfies
    \begin{equation*}
        P_{\mathfrak{F}} - (2 P_{\mathcal{H}} + 2) = 2 (L + 1) W_N^2 + 2 W_N - 3 > 0, \qquad P_{\mathfrak{F}} \leq (2 W_N + 1) \big( 2 (L + 1) W_N + 1 \big) \leq \bar{P}.
    \end{equation*}
    Hence $ 1 \leq \bar{P}_0 \leq P_{\mathfrak{F}}, $ part \ref{cond:entropy} applies to every nonempty subset of $ \mathfrak{F}, $ and since $ P \mapsto \log (e P / \bar{P}_0) $ is increasing, $ \bar{P} $ may replace $ P_{\mathfrak{F}} $ in \eqref{eq:entropy}. Fix $ 0 < \varepsilon \leq 1 $ and write $ \mathfrak{e}(\varepsilon) := \bar{P}_0 [ \log (e \bar{P} / \bar{P}_0) + \log (8 \Lambda \bar{\mathrm{L}}') + \log (1 / \varepsilon) ]. $

    Trial class. Every $ h \in \mathcal{H}_N $ equals $ f^\chi $ for some $ f \in \mathfrak{F}. $ Indeed, let $ h = (\Theta \circ g)^\chi $ with $ g = A_L x^{(L)} + b_L \in \mathcal{NN}_\varrho (W_N, L, \Lambda, P_0, \mathcal{A}). $ Turning the readout of $ g $ into a hidden layer of width one with activation $ \Theta, $ and reading this layer out with weight $ 1 $ and bias $ 0, $ gives a network that computes $ \Theta \circ g $ on all of $ I \times \mathbb{R}^d, $ has depth $ L + 1 $ and hidden widths at most $ W_N, $ entries bounded by $ \Lambda \geq 1, $ at most $ P_0 + 1 \leq \bar{P}_0 $ nonzero entries, and hidden activations bounded on $ Z $ by $ \mathcal{A} \lor 3 \bar{B}, $ as $ | \Theta | \leq 3 \bar{B}. $ Let $ \mathcal{F}_{\mathcal{H}} := \{ f \in \mathfrak{F} : f^\chi \in \mathcal{H}_N \}. $ For $ f \in \mathfrak{F} $ and $ (t, x) \in I \times \mathbb{R}^d $ the chain rule gives
    \begin{equation} \label{eq:clip_jet}
        (f^\chi, \partial_t f^\chi, \nabla f^\chi) (t, x) = \big( f, \partial_t f, D \chi_R (x)^\top \nabla_x f \big) (t, \chi_R (x)),
    \end{equation}
    where $ D \chi_R (x) $ is diagonal with entries in $ [-1, 1] $ and $ (t, \chi_R (x)) \in Z. $ Hence, for $ f, \tilde{f} \in \mathfrak{F}, $ $ \| J (f^\chi) - J (\tilde{f}^\chi) \|_\infty \leq \| (f, \nabla_z f) - (\tilde{f}, \nabla_z \tilde{f}) \|_{\infty, Z}, $ and part \ref{cond:entropy} applied to $ \mathcal{F}_{\mathcal{H}} $ yields an $\varepsilon$-net of $ J \mathcal{H}_N $ consisting of jets of elements of $ \mathcal{H}_N, $ so that $ \log N (\varepsilon, J \mathcal{H}_N, \| \cdot \|_\infty) \leq \mathfrak{e}(\varepsilon). $

    Generators. Let $ \Phi_N := \{ \phi^\chi : \phi \in \mathfrak{F}, \, \sup_{I \times \mathbb{R}^d} ( | \phi^\chi | + | \nabla \phi^\chi | + | \partial_t \phi^\chi | ) \leq \Gamma \}, $ so that $ \Psi_N = \{ c \, \omega \, u : c \in \mathbb{R}, \, u \in \Phi_N \} $ by \eqref{eq:psi_n_class}, and $ J \Phi_N := \{ (u, \nabla u) : u \in \Phi_N \}. $ Deleting the time component in \eqref{eq:clip_jet} does not increase distances, so the same argument gives $ \log N (\varepsilon, J \Phi_N, \| \cdot \|_\infty) \leq \mathfrak{e}(\varepsilon). $

    Unit ball of the cone scale. Let $ \psi \in \Psi^\bullet_N. $ By \eqref{eq:bullet_norm} there is a representation $ \psi = c \, \omega \, u $ with $ u \in \Phi_N $ and $ | c | \leq 2, $ and since $ \omega $ depends on $ t $ only, $ (\psi, \nabla \psi) = c \, \omega \, (u, \nabla u). $ Put $ \eta := \varepsilon / (2 \Gamma) $ and let $ \mathcal{C} $ consist of the points $ \max \{ -2, \min \{ 2, -2 + (2k - 1) \eta \} \}, \, 1 \leq k \leq \lceil 2 / \eta \rceil. $ Every point of $ [-2, 2] $ lies within $ \eta $ of $ \mathcal{C}, $ because the intervals $ [ -2 + 2 (k-1) \eta, -2 + 2 k \eta ] $ cover $ [-2, 2] $ and projecting onto $ [-2, 2] $ does not increase distances to its points, and $ | \mathcal{C} | \leq 1 + 2 / \eta = 1 + 4 \Gamma / \varepsilon. $ Let $ \mathcal{U} $ be an $ (\varepsilon / 4) $-net of $ J \Phi_N. $ If $ c' \in \mathcal{C} $ and $ U' \in \mathcal{U} $ satisfy $ | c - c' | \leq \eta $ and $ \| (u, \nabla u) - U' \|_\infty \leq \varepsilon / 4, $ then, using $ 0 \leq \omega \leq 1, $ $ | (u, \nabla u) | \leq | u | + | \nabla u | \leq \Gamma $ and $ | c' | \leq 2, $
    \begin{equation*}
        | c \, \omega \, (u, \nabla u) - c' \omega \, U' | \leq | c - c' | \, | (u, \nabla u) | + | c' | \, | (u, \nabla u) - U' | \leq \eta \Gamma + \varepsilon / 2 = \varepsilon
    \end{equation*}
    everywhere on $ I \times \mathbb{R}^d. $ Hence $ N (\varepsilon, J \Psi^\bullet_N, \| \cdot \|_\infty) \leq (1 + 4 \Gamma / \varepsilon) \, N (\varepsilon / 4, J \Phi_N, \| \cdot \|_\infty), $ and with $ \log (1 + 4 \Gamma / \varepsilon) \leq \log (1 + 4 \Gamma) + \log (1 / \varepsilon) $ for $ \varepsilon \leq 1, $
    \begin{equation*}
        \log N (\varepsilon, J \Psi^\bullet_N, \| \cdot \|_\infty) \leq \log (1 + 4 \Gamma) + \bar{P}_0 \left[ \log \frac{e \bar{P}}{\bar{P}_0} + \log (32 \Lambda \bar{\mathrm{L}}') \right] + (\bar{P}_0 + 1) \log \frac{1}{\varepsilon}.
    \end{equation*}

    Conclusion. With $ u := \log (1 / \varepsilon) \geq 0, $ both logarithms are at most $ a + (\bar{P}_0 + 1) u, $ where $ a := \log (1 + 4 \Gamma) + \bar{P}_0 [ \log (e \bar{P} / \bar{P}_0) + \log (32 \Lambda \bar{\mathrm{L}}') ]. $ The quotient $ (a + (\bar{P}_0 + 1) u) / (1 + u) $ is a convex combination of $ a $ and $ \bar{P}_0 + 1, $ and $ a \geq \bar{P}_0 (1 + \log 32) \geq \bar{P}_0 + 1 \geq 1, $ since $ \bar{P}_0 \leq \bar{P} $ and $ \Lambda, \bar{\mathrm{L}}' \geq 1. $ Therefore $ \mathbb{V}_N \leq a, $ and substituting $ \log \bar{\mathrm{L}}' = \log (12 (L + 1)^2) + (9 L + 12) \log \bar{\mathrm{m}} $ gives \eqref{eq:eff_dim_bound}.

    For \eqref{eq:eff_dim_bound_particular}, insert the stated magnitudes. Then $ \bar{P}_0 \lesssim K N^d, $ $ \log (e \bar{P} / \bar{P}_0) \leq \log (e \bar{P}) \lesssim \log (L K N^d) \lesssim \log n $ and $ \log (384 \Lambda (L + 1)^2) \lesssim \log n. $ Next, $ \log \bar{\mathrm{m}} \lesssim \log n, $ because $ \Lambda, W_N, \mathcal{A} $ are $ \mathrm{poly}(n) $ while $ \bar{B}, \kappa_\Theta $ are fixed and $ R_* = \mathrm{polylog}(n), $ and likewise $ \log (1 + 4 \Gamma) \lesssim \log n, $ since $ \Gamma $ depends on $ n $ only through $ R. $ As $ 9 L + 12 \leq 21 L $ and $ L \lesssim \log N $ with $ N \geq 2, $ the bracket in \eqref{eq:eff_dim_bound} is $ \lesssim \log N \log n, $ and multiplying by $ \bar{P}_0 $ gives $ \mathbb{V}_N \lesssim K N^d \log N \log n. $ With $ K \leq \sigma^{-c} (\log n)^{3/2} $ from \eqref{eq:R_K_choice} this is $ N^d \, \mathrm{polylog}(n). $
\end{proof}

\begin{remark}(The sparsity budget is attained)
    Equation \eqref{eq:eff_dim_bound_particular} rests on the sparsity budget $ P_0 \lesssim K N^d $ of \ref{cond:vartheta_entries}: a dense layer of width $ W_N \asymp K N^d $ carries $ \bar{P} \asymp K^2 N^{2d} $ parameters, and \eqref{eq:eff_dim_bound} would then give $ \mathbb{V}_N \asymp N^{2d} \mathrm{polylog}(n) $ and the exponent $ 2 (\beta - 1) / (2 (\beta - 1) + 2d) $ in Theorem \ref{thm:guidance_rate}. Two facts from \citep{guhring2021} help us. Proposition 4.8 there produces $ O(\varepsilon^{-d/(\beta - k)}) $ nonzero weights, which at the accuracy $ \varepsilon \asymp N^{-(\beta - k)} $ used in Lemma \ref{lem:simult_approx} is exactly $ O(N^d) $ per Chebyshev mode, hence $ O(N^d \mathrm{polylog}(n)) $ in total. Also Corollary 3.8 shows that any architecture with encodable weights achieving accuracy $ \varepsilon $ in $ W^{k, \infty} $ on the unit ball of $ W^{\beta, \infty} $ needs at least $ C \varepsilon^{-d /(\beta - k)} / \log_2 (1 / \varepsilon) $ weights.
\end{remark}

\section{The approximation floor} \label{app:approx}

Every constant in this section and the next depends on $ \beta. $ Write
\begin{equation} \label{eq:A_beta_C_beta_def}
    \mathrm{A}_\beta := C_\beta \bar{B} \sigma_*^{- 2 \beta} (2 R)^{\beta - 1}, \quad C_\beta := \mathfrak{B}_{\beta + 1} \beta! \, (2 \sec (\pi / 6))^{\beta + 1},
\end{equation}
with $ \mathfrak{B}_{\beta + 1} $ the Bell number. The factor $ C_\beta \sigma_*^{-2 \beta} $ is the derivative envelope of Step 1 of Lemma \ref{lem:space_time_approx}. The factor $ (2 R)^{\beta - 1} $ is the price of rescaling $ Q_R $ onto the unit cube in Step 3. Since $ R \asymp \sqrt{\log n} $ by \eqref{eq:R_K_choice}, the whole of $ \mathrm{A}_\beta $ is $ \mathrm{polylog}(n) $ at fixed $ \beta, $ but the bound grows factorially in $ \beta, $ see the remark on constants in Section \ref{sec:discussion}. We write 
\begin{equation} \label{eq:vareps_n_def}
    \varepsilon_N := \mathrm{A}_\beta N^{- (\beta - 1)}
\end{equation}
for the approximation scale, and from here on $ \mathrm{A}_\beta $ replaces the $ \sigma^{-c} $ used to this point wherever a $\beta$-dependent constant occurs. 

\begin{lemma}(Weak residual approximation) \label{lem:weak_res_approx}
    Suppose Assumptions \ref{ass:bounded}--\ref{ass:compact} hold and that $ h_N \in \mathcal{H}_N $ and $ e_N > 0 $ satisfy, with $ \varphi_N := h_N - h^*, $
    \begin{equation} \label{eq:approx_error_eN}
        \| \varphi_N \|_{L^4 (\nu \otimes \rho)} + \sup_{t \in I} \| \varphi_N (t, \cdot) \|_{\rho_t} + \| | \nabla \varphi_N | \|_{L^4 (\nu \otimes \rho)} + \| \partial_t \varphi_N \|_{L^2 (\nu \otimes \rho)} \leq e_N.
    \end{equation}
    Then $ \underset{h \in \mathcal{H}_N}{\inf} \mathcal{P}^w (h) \leq \mathcal{P}^w(h_N) \leq 4 (1 + M_4)^2 e_N^2 \leq C d \sigma_*^{-4} e_N^2. $
\end{lemma}
\begin{proof}
    Lemma \ref{lem:continuity} bounds the dual norm of the residual by first-order quantities only: $ \| \mathcal{R}_t [h_N] \|_{H^{-1} (\rho_t)} \leq \| \partial_t \varphi_N \|_{\rho_t} + \| | \nabla \varphi_N | \|_{\rho_t} + M_4 \| | \nabla \varphi_N | \|_{L^4 (\rho_t)}. $ Since $ \rho_t $ is a probability measure, $ \| \cdot \|_{\rho_t} \leq \| \cdot \|_{L^4 (\rho_t)}, $ so the right-hand side is at most $ \| \partial_t \varphi_N \|_{\rho_t} + (1 + M_4) \| | \nabla \varphi_N | \|_{L^4 (\rho_t)}. $ Squaring, using $ (a + b)^2 \leq 2 a^2 + 2 b^2 $ and integrating against $ \nu, $
    \begin{equation*}
        \mathcal{P}^w (h_N) \leq 2 \| \partial_t \varphi_N \|^2_{L^2 (\nu \otimes \rho)} + 2 (1 + M_4)^2 \int_I \| | \nabla \varphi_N | \|^2_{L^4 (\rho_t)} \nu (dt).
    \end{equation*}
    By Jensen's inequality for the concave map $ \sqrt{\cdot}, $ $ \int_I ( \int | \nabla \varphi_N |^4 \rho_t )^{1/2} \nu (dt) \leq \| | \nabla \varphi_N | \|^2_{L^4 (\nu \otimes \rho)}. $ Hence $ \mathcal{P}^w (h_N) \leq 2 e_N^2 + 2 (1 + M_4)^2 e_N^2 \leq 4 (1 + M_4)^2 e_N^2, $ and $ M_4 \leq C \sqrt{d} \sigma^{-2}_* $ by Lemma \ref{lem:score_identity}.
\end{proof}

It remains to exhibit an $ h_N \in \mathcal{H}_N $ satisfying \eqref{eq:approx_error_eN} with $ e_N $ of order $ \mathrm{A}_\beta N^{-(\beta - 1)} $ up to polylogarithmic factors. The construction separates the two variables -- Chebyshev truncation in $ t $ and spline quasi-interpolation in $ x. $

 We will use fact that a function analytic in a neighbourhood of an interval has geometrically decaying Chebyshev coefficients. For $ \mathrm{r} > 1 $ the Bernstein ellipse $ E_{\mathrm{r}}  \subset \mathbb{C} $ is the image of the circle $ \{ | w | = \mathrm{r} \} $ under the map $ w \mapsto \frac{1}{2} (w + w^{-1}). $ It is the open ellipse with foci $ \pm 1, $ semi-axes
 \begin{equation} \label{eq:chebyshev_coeffs}
     a_{\mathrm{r}} = \frac{1}{2} (\mathrm{r} + \mathrm{r}^{-1}), \quad b_{\mathrm{r}} = \frac{1}{2} (\mathrm{r} - \mathrm{r}^{-1}),
 \end{equation}
 so that $  E_{\mathrm{r}} $ shrinks to $ [-1, 1] $ as $ \mathrm{r} \downarrow 1. $ Write $ T_j $ for the Chebyshev polynomials on $ [-1, 1] $ and, for $ g: [-1, 1] \rightarrow \mathbb{R} $ integrable against the Chebyshev weight, $ a_j (g) := \frac{2}{\pi} \int_{-1}^1 g (\theta) T_j (
 \theta) (1 - \theta^2)^{-1/2} d \theta $ for $ j \geq 1 $ (and half that for $ j = 0 $).

 \begin{lemma}(Chebyshev truncation on a strip) \label{lem:chebyshev}
     Let $ J := [t_0, t_1] $ with $ \ell = t_1 - t_0, $ let $ \iota (\theta) := t_0 + \frac{\ell}{2} (\theta + 1) $ map $ [-1, 1] $ onto $ J, $ and for $ \eta > 0 $ let $ S_\eta := \{ z \in \mathbb{C}: \mathrm{dist}(z, J) < \eta \} $ be the $\eta$-neighbourhood of J. Suppose $ g $ is holomorphic on $ S_\eta $ with $ | g | \leq M $ there. Then:
     \begin{enumerate}[label=(\roman*),leftmargin=2.2em]
         \item $ \iota^{-1} (S_\eta) $ contains the Bernstein ellipse $ E_{\mathrm{r}} $ for every $ \mathrm{r} \leq 1 + 2 \eta / \ell, $
         \item the Chebyshev coefficients of $ g \circ \iota $ obey $ | a_j (g \circ \iota) | \leq 2 M \mathrm{r}^{-j}, $
         \item for every $ K \geq 1, $ writing $ \Pi_K g := \sum_{j \leq K} a_j (g \circ \iota) T_j \circ \iota^{-1}, $
         \begin{equation} \label{eq:chebyshev_bounds}
             \underset{t \in J}{\sup} \left| g(t) - \Pi_K g(t) \right| \leq \frac{2 M \mathrm{r}^{-K}}{\mathrm{r} - 1}, \quad \underset{t \in J}{\sup} \left| g'(t) - ( \Pi_K g)' (t) \right| \leq \frac{2}{\ell} \cdot \frac{10 M (K+1)^2 \, \mathrm{r}^{2-K}}{(\mathrm{r} - 1)^3}.
         \end{equation}
     \end{enumerate}
 \end{lemma}
 \begin{proof}
     (i) A point of $ E_{\mathrm{r}} $ is $ a_{\mathrm{r}} \cos \vartheta + i b_{\mathrm{r}} \sin \vartheta. $ Its distance to $ [-1, 1] $ is $ b_{\mathrm{r}} | \sin \vartheta | \leq b_{\mathrm{r}} $ where $ | a_{\mathrm{r}} \cos \vartheta | \leq 1. $ Elsewhere, with $ c := | \cos \vartheta | > 1 / a_{\mathrm{r}}, $ the squared distance is $ F(c) = ( a_{\mathrm{r}} c - 1 )^2 + b_{\mathrm{r}}^2 (1 - c^2) $ and $ F'(c) = 2 c (a^2_{\mathrm{r}} - b^2_{\mathrm{r}}) - 2 a_{\mathrm{r}} = 2 (c - a_{\mathrm{r}}) < 0 $ by \eqref{eq:chebyshev_coeffs}, so $ F(c) \leq F(1 / a_{\mathrm{r}}) = b^4_{\mathrm{r}} / a^2_{\mathrm{r}} \leq b^2_{\mathrm{r}}. $ Hence $ \mathrm{dist} (E_\mathrm{r}, [-1, 1]) = b_{\mathrm{r}} $ exactly. Under $ \iota $ lengths scale by $ \ell / 2, $ and writing $ \mathrm{r} = 1 + \varepsilon $ with $ \varepsilon \leq 1 $ one has $ b_{\mathrm{r}} = \frac{\varepsilon (2 + \varepsilon)}{2 (1 + \varepsilon)} \leq \varepsilon, $ so the distance is at most $ \ell \varepsilon / 2 \leq \eta $ when $ \varepsilon \leq 2 \eta / \ell. $

     (ii) and the first half of (iii) are the Bernstein estimates: for $ g \circ \iota $ holomorphic and bounded by $ M $ on $ E_{\mathrm{r}}, \, |a_j| \leq 2 M \mathrm{r}^{-j} $ and the truncation error is at most $ 2 M \mathrm{r}^{-K} / (\mathrm{r} - 1) $ \citep{trefethen2013} (Theorems 8.1--8.2). The second follows from the first by summing the geometric tail, $ \sum_{j > K} | a_j | \| T_j \|_\infty \leq 2 M \sum_{j > K} \mathrm{r}^{-j}. $

     For the derivative bound, Markov's inequality gives $ \| T_j' \|_{\infty, [-1,1]} = j^2, $ so $ \sum_{j > K} | a_j | \| T_j' \|_\infty \leq 2 M \sum_{j > K} j^2 \mathrm{r}^{-j}. $ For $ x \in (0, 1) $ and $ m \geq 1 $
     \begin{equation}
         \sum_{j \geq m} j^2 x^j = x^m \sum_{i \geq 0} (i + m)^2 x^i \leq x^m \left[ \frac{m^2}{1 - x} + \frac{2 m x}{(1 - x)^2} + \frac{x ( 1 + x)}{(1-x)^3} \right] \leq \frac{5 m^2 x^m}{(1-x)^3},
     \end{equation}
     using $ m^2 + 2m + 2 \leq 5 m^2, $ with $ x = \mathrm{r}^{-1}, $ so $ 1 - x = (\mathrm{r} - 1) / \mathrm{r}, $ and $ m = K + 1 $ this is $ 5 (K + 1)^2 \mathrm{r}^2 \mathrm{r}^{-K} (\mathrm{r} - 1)^{-3}. $ The chain rule contributes $ | (\iota^{-1})'| = 2 / \ell. $ 
 \end{proof}

 \begin{lemma}(Exact arithmetic with the rectified cubic) \label{lem:exact_arithmetic}
     Let $ \varrho (u) = ( u \lor 0)^3 $ and let $ M \geq 1. $ Then as identities of functions:
     \begin{enumerate}[label=(\roman*),leftmargin=2.2em]
     \item (Cube) $ u^3 = \varrho(u) - \varrho(-u) $ for all $ u \in \mathbb{R}, $
     \item (Square) $ u^2 = \frac{1}{6M} \left[ \varrho(u+M) + \varrho(M-u) - 2 M^3\right] $ for all $ |u| \leq M, $
     \item (Identity) $ u = \frac{1}{6 M^2} \left[ \varrho (u + M) - \varrho (M - u) - 2 \varrho(u) + 2 \varrho(-u)  \right] $ for all $ | u | \leq M, $
     \item (Product) $ uv = \frac{1}{4} \left[ (u+v)^2 - (u-v)^2 \right] $ for all $ |u|, |v| \leq M / 2, $ hence $ (u,v) \mapsto uv $ is realized exactly by one hidden layer of four $\varrho$-units,
     \item (Cubic B-splines) For knots $ \tau_0 < \ldots < \tau_4 $ the cardinal cubic $B$-spline $ B(u) := (\tau_4 - \tau_0) \sum_{i=0}^4 ( \prod_{j \neq i} (\tau_i - \tau_j))^{-1} \varrho (\tau_i - u) $ is realized exactly by one hidden layer of five $\varrho$-units, with inner weights $ \pm 1 $ and outer weights of modulus $ O (\mathfrak{h}^{-3}) $ for a uniform mesh of width $ \mathfrak{h}. $
     \end{enumerate}
 \end{lemma}
 \begin{proof}
     (i) For $ u \geq 0 $ the right side is $ u^3 - 0, $ for $ u < 0 $ its is $ 0 - (- u)^3 = u^3. $

     (ii) For $ |u| \leq M $ both $ u + M $ and $ M - u $ are nonnegative, so $ \varrho(u + M) + \varrho(M - u) = (u+M)^3 + (M-u)^3 = (u+M)^3 - (u - M)^3 = 6 u^2 M + 2 M^3, $ using $ (a+b)^3 - (a-b)^3 = 6 a^2 b + 2 b^3, $ with $ a = u, b = M. $

     (iii) Similarly $ \varrho (u + M) - \varrho(M - u) = (u + M)^3 + (u - M)^3 = 2 u^3 + 6 u M^2, $ using $ (a + b)^3 + (a - b)^3 = 2 a^3 + 6 ab^2. $ Subtracting $ 2 u^3 $ in the form (i) leaves $ 6 u M^2. $

     (iv) The polarization identity, with (ii) applied to $ u \pm v, $ both of modulus at most $ M. $

     (v) The divided-difference representation of the $B$-spline of degree 3 in the knots $ \tau_0, \ldots, \tau_4 $ is exactly the stated combination of truncated cubics $ (\tau_i - u)^3_+ = \varrho (\tau_i - u) $ \citep{deboor1978} (Chapter IX). On a uniform mesh of width $ \mathfrak{h}, \, \prod_{j \neq i} | \tau_i - \tau_j | \asymp \mathfrak{h}^4 $ and $ \tau_4 - \tau_0 = 4 \mathfrak{h}, $ giving outer weights of modulus $ O (\mathfrak{h}^{-3}). $
 \end{proof}

 \begin{definition}(Exact partition of unity of radius $ \mathrm{a}_* $) \label{def:exact_partition}
     Let $ j \in \mathbb{N}_0 $ and $ \mathrm{a}_* \in \mathbb{N}. $ For $ N \in \mathbb{N} $ put $ M_N := \{ - \mathrm{a}_*, \ldots, N + \mathrm{a}_* \}^d $ A family $ \Psi^{(j, N)} := \{ \varphi_m : m \in M_N \} $ of functions $ \varphi_m: \mathbb{R}^d \rightarrow \mathbb{R} $ is an exact $j$-PU of radius $ \mathrm{a}_* $ for the activation $ \varrho $ if there is $ C = C(d, j) $ such that for every $ N $ and every $ k \in \{0, \ldots, j \}: $
     \begin{enumerate}[label=(\roman*), leftmargin=2.2em]
         \item $ \| \varphi_m \|_{W^{k, \infty}(\mathbb{R}^d)} \leq C N^k $ for every $ m \in M_N, $
         \item $ \varphi_m \equiv 0 $ on $ \{ x: \| x - m / N \|_\infty \geq \mathrm{a}_* / N \}, $
         \item $ \sum_{m \in M_N} \varphi_m = 1 $ on $ (0, 1)^d, $
         \item for each $ m $ there is a network $ \Phi_m $ with $d$-dimensional input and output, two layers and at most $ C $ nonzero weights, such that $ \prod_{l=1}^d [ R_\varrho (\Phi_m) ]_l = \varphi_m, \, \| R_\varrho (\Phi_m) \|_{W^{k, \infty}((0,1)^d)} \leq C N^k $ and $ \| \Phi_m \|_{\max} \leq CN. $
     \end{enumerate}
 \end{definition}

 \begin{lemma}(The cubic $B$-spline family is an exact 3-PU of radius 2 for $ \varrho_3 $) \label{lem:3_pu_from_cubic}
     Let $ B $ be the cardinal cubic $B$-spline of Lemma \ref{lem:exact_arithmetic} with unit knot spacing, i.e. $ \tau_i = i - 2 $ for $ i = 0, \ldots, 4, $ and set
     \begin{equation} \label{eq:varphi_for_pu}
         \varphi_m (x) := \prod_{l=1}^d B (N x_l - m_l), \quad m \in M_N = \{ -2, \ldots, N+2 \}^d.
     \end{equation}
     Then $ \Psi^{(3, N)} := \{ \varphi_m : m \in M_N \} $ is an exact 3-PU of radius $ \mathrm{a}_* = 2 $ for $ \varrho_3, $ with $ C = 3^d \lor 15 d. $
 \end{lemma}
 \begin{proof}
     Lemma \ref{lem:exact_arithmetic} writes $ B(u) = 4 \sum_{i=0}^4 ( \prod_{j \neq i} (\tau_i - \tau_j))^{-1} \varrho_3 (\tau_i - u), $ whose outer coefficients are $ (\frac{1}{6}, - \frac{2}{3}, 1 - \frac{2}{3}, \frac{1}{6}). $ The standard properties of the cardinal cubic B-spline \citep[Ch.~IX]{deboor1978} are $ B \geq 0, \, \mathrm{supp} B = [-2, 2], \, \sum_{m \in \mathbb{Z}} B(\cdot - m) \equiv 1, \, B \in C^2 (\mathbb{R}) \cap W^{3, \infty} (\mathbb{R}), $ with
     \begin{equation} \label{eq:b_spline_properties}
         \| B \|_\infty = \| B' \|_\infty = \frac{2}{3}, \quad \| B'' \|_\infty = 2, \quad \| B''' \|_\infty = 3.
     \end{equation}

     \begin{enumerate}[label=(\roman*), leftmargin=2.2em]
         \item For $ | \alpha | = k \leq 3, \, \partial^\alpha \varphi_m (x) = N^k \prod_l B^{(\alpha_l)} (N x_l - m_l), $ so $ \| \varphi_m \|_{W^{k, \infty}} \leq 3^d N^k $ by \eqref{eq:b_spline_properties}.
         \item $ B (N x_l - m_l) = 0 $ once $ | x_l - m_l / N | \geq 2 / N, $ so $ \varphi_m $ vanishes off $ \{ \| x - m / N \|_\infty < 2 / N \}. $ This is where the radius $ \mathrm{a}_* = 2 $ enters. The radius is a convenience: $ b(u) := B(3u - 1) + B(3u) + B(3u + 1) $ vanishes for $ |u| \geq 1, $ is realized by fifteen $ \varrho_3 $-units, and satisfies $ \sum_{m \in \mathbb{Z}} b(u - m) = \sum_{k \in \mathbb{Z}} B(3u - k) = 1, $ so radius $ 1 $ is also attainable, radius $ 2 $ is used because it keeps the construction and its constants simplest.
         \item Fix $ x \in (0, 1)^d $ and $ l. $ Then $ B ( N x_l - m_l) \neq 0 $ forces $ m_l \in (N x_l - 2, N x_l + 2) \subset (-2, N + 2), $ so every index contributing at $ x $ already lies on $ M_N $ and 
         \begin{equation}
             \sum_{m \in M_N} \varphi_m (x) = \prod_{l=1}^d \sum_{m_l \in \mathbb{Z}} B (N x_l - m_l) = 1.
         \end{equation}
         The overhang is needed, since over $ \{0, \ldots, N \}^d $ the sum equals $ (5/6)^d $ at the corners of the cube and differs from 1 on a boundary layer of width $ 2 / N. $
         \item $ u \mapsto B (N u - m_l) $ is one hidden layer of five $ \varrho_3 $ units with inner weight $ -N, $ biases $ \tau_i + m_l $ of modulus at most $ N + 4, $ and outer weights of modulus at most 1. Parallelizing the $ d $ coordinates (Lemma C.2 of \citep{guhring2021}) gives $ \Phi_m $ with two layers, at most $ 15d $ nonzero weights, $ \prod_l [R_{\varrho_3} (\Phi_m) ]_l = \varphi_m, \, \| R_{\varrho_3} (\Phi_m) \|_{W^{k, \infty}} \leq 3 N^k $ and $ \| \Phi_m \|_{\max} \leq N + 4 \leq 3 N $ for $ N \geq 2. $
     \end{enumerate}
 \end{proof}

\begin{lemma}(Localization under a widened partition of unity) \label{lem:localization_widened_partition}
    Let $ \Psi^{(j,N)} $ be an exact $j$-PU of radius $ \mathrm{a}_*, \, k \in \{ 0, \ldots, j \}, \, n \in \mathbb{N}_{\geq k+1} $ and $ 1 \leq p \leq \infty. $ Then the conclusion of Lemma D.1 of \citep{guhring2021} holds: there are polynomials $ p_{f, m} $ of degree $ \leq n - 1, $ the averaged Taylor polynomials of $ f $ on the balls $ \hat{\Omega}_{m,N} := B_{\mathrm{a}_* / N, \| \cdot \|_\infty} (m/N), $ independent of $ k, $ such that $ f_N := \sum_{m \in M_N} \varphi_m p_{f,m} $ satisfies 
    \begin{equation}
        \| f - f_N \|_{W^{k, p}((0,1)^d)} \leq C \| f \|_{W^{n,p}((0,1)^d)} N^{-(n-k)}, \quad C = C (d, n, p, k, \mathrm{a}_*).
    \end{equation}
\end{lemma}
\begin{proof}
    First, in Step 1 the Bramble-Hilbert polynomial is taken from their Lemma B.4 on $ \hat{\Omega}_{m, N} $ rather than on $ \Omega_{m, N} = B_{1 / N, \| \cdot \|_\infty} (m / N). $ That lemma holds on a ball of any radius $ c / N $ with a constant $ C(n, d, c), $ and returns $ ( \mathrm{a}_* / N)^{n - k}, $ i.e. the same rate with the constant multiplied by $ \mathrm{a}_*^{n - k}. $ The polynomial is still the averaged Taylor polynomial, hence still independent of $ k, $ which is what makes the simultaneity in Lemma \ref{lem:simult_approx} possible. 

    Second, their Steps 2 and 3 split the index set at $ \| \tilde{m} - m \|_\infty > \mathrm{a}_*  + 1 $ instead of $ > 1. $ By (ii) of Definition \ref{def:exact_partition}, $ \mathrm{supp} \, \varphi_m \, \cap \, \Omega_{\tilde{m}, N} = \emptyset $ as soon as $ \| \tilde{m} - m \|_\infty > \mathrm{a}_* + 1, $ so the far contribution vanishes identically, as it does for an exact PU in their Step 2. The near contribution now runs over at most $ (2 \mathrm{a}_* + 3)^d $ indices instead of $ 3^d $ and is estimated as in their Step 1, on $ \hat{\Omega}_{m, N}. $ For $ \mathrm{a}_* = 2 $ the count is $ 7^d. $

    Third, in their Step 4a the term $ \| \tilde{f}( \mathbf{1}_{(0,1)^d} - \sum_m \varphi_m ) $ vanishes identically by (iii), and in their Step 4b H\"{o}lder inequality is applied with $ (2 \mathrm{a}_* + 3)^d $ in place of $ 3^d $ while the overlap decomposition preceding their Eq. (D.9) uses $ (2 \mathrm{a}_*)^d $ disjoint classes in place of $ 2^d. $ Both are $d$-dependent constants.
\end{proof}

 \begin{lemma}(Simultaneous approximation by $ \varrho_3 $-networks) \label{lem:simult_approx}
     Let $ \beta \geq 2 $ be an integer and let $ f $ lie in the unit ball of $ W^{\beta, \infty} ((0, 1)^d). $ There are constants $ C, \tilde{N} $ and a depth $ L_0 $ depending only on $ (d, \beta) $ such that for every $ N \geq \tilde{N} $ there is a $ \varrho_3 $-network $ \Phi_N $ with at most $ L_0 $ layers and at most $ C (N + 1)^d $ nonzero weights, all of modulus at most $ C N^C $ and encodable with $ \lceil C \beta \log_2 N \rceil $ bits, whose realization satisfies simultaneously
     \begin{equation} \label{eq:simult_approx}
         \| f - R (\Phi_N) \|_{W^{0, \infty}} \leq C N^{-\beta}, \quad \| f - R (\Phi_N) \|_{W^{1, \infty}} \leq C N^{-(\beta - 1)}.
     \end{equation}
 \end{lemma}
 \begin{proof}
     We assemble this from \citep{guhring2021} with the partition of unity supplied by Lemma \ref{lem:3_pu_from_cubic} rather than by Remark 4.6 of that paper.

     Localization. By Lemma \ref{lem:3_pu_from_cubic} the family \eqref{eq:varphi_for_pu} is an exact 3-PU of radius 2 for $ \varrho_3, $ so Lemma \ref{lem:localization_widened_partition} applies with $ j = 3, \, \mathrm{a}_* = 2, \, n = \beta, \, p = \infty. $ It produces one localized approximant $ f_N = \sum_{m \in M_N} \varphi_m p_{f, m}, $ whose polynomials are averaged Taylor polynomials and therefore do not depend on $ k, $ with
     \begin{equation} \label{eq:simult_approx_proof}
         \| f - f_N \|_{W^{k, \infty}((0,1)^d)} \leq C N^{- (\beta - k)} \quad \text{simultaneously for } k \in \{ 0, 1 \}. 
     \end{equation}

     Lemma \ref{lem:localization_widened_partition} requires $ k \leq j = 3 $ and $ n = \beta \geq k + 1, $ so \eqref{eq:simult_approx_proof} holds for $ k \in \{ 0, \ldots, \min (3, \beta - 1) \}, $ which contains $ \{ 0, 1 \} $ because $ \beta \geq 2. $ Only $ k \in \{ 0, 1 \} $ is used, so neither this restriction nor the fact that $ \varrho_3 \notin C^3 $ makes the chain rule of Corollary B.6 of \citep{guhring2021} unavailable at order three ever binds. 

     Realization. Lemma D.5 of \citep{guhring2021} is now applied unchanged. It is insensitive to the widening, since its proof invokes only conditions (i) and (iv) of the partition of unity -- the bump network of Definition \ref{def:exact_partition} and the derivative bound $ \| R_\varrho ( \Phi_m) \|_{W^{k, \infty}} \leq C N^k $ of (i) -- and never the support condition (ii) or the summation condition (iii), which are the two that Lemma \ref{lem:localization_widened_partition} had to absorb. The only trace of the enlarged index set is that its count $ C (N + 1)^d $ becomes $ C | M_N | = C (N+5)^d \leq C' (N + 1)^d. $ Its side hypothesis that $ \varrho $ be three times continuously differentiable near some $ x_0 $ with $ \varrho''(x_0) \neq 0 $ -- holds at any $ x_0 > 0, $ where $ \varrho_3 $ agrees with $ u \mapsto u^3 $ and $ \varrho''_3 (x_0) = 6x_0 \neq 0. $ The lemma realizes $ f_N $ to accuracy $ \varepsilon $ in $ W^{k, \infty} $ with at most $ C $ layers and $ C(N+1)^d $ nonzero weights independently of $ \varepsilon, $ only the weight magnitude growing, as $ \varepsilon^{-2} $ times a power of $ N. $

     Simultaneity. Apply it once, at the same $ N, $ with $ k = 1 $ and $ \varepsilon := N^{-\beta}. $ Since $ \| \cdot \|_{W^{0, \infty}} \leq \| \cdot \|_{W^{1, \infty}}, $ the resulting network approximates $ f_N $ to $ N^{-\beta} $ in both norms at once, and the triangle inequality against \eqref{eq:simult_approx_proof} gives \eqref{eq:simult_approx}. It is the $\varepsilon$-independence of the weight count that makes this free. One may spend accuracy $ N^{- \beta} $ on the realization while the localization error is only $ N^{- (\beta - 1)}, $ and pay for it in weight magnitude rather than in weight count.
     
     Encodability. The weights are $ \mathrm{poly}(N). $ Round each of them to the grid of spacing $ N^{-C \beta}. $ The resulting perturbation is controlled by the parameter-Lipschitz bound of Lemma \ref{lem:complexity}: $ \| R (\Phi_\vartheta) - R (\Phi_{\tilde{\vartheta}}) \|_{C^1} \leq \mathrm{L}'_* \| \vartheta - \tilde{\vartheta} \|_\infty $ with $ \mathrm{L}'_* = 12 L^2 \mathrm{m}^{9L + 3} = \mathrm{poly}(N) $ at the constant depth $ L_0. $ Taking $ C $ large enough makes the perturbation at most $ N^{- \beta}, $ absorbed into \eqref{eq:simult_approx}, and each surviving weight is encoded by $ \lceil C \beta \log_2 N \rceil $ bits.
 \end{proof}

\begin{lemma}(Space-time approximation) \label{lem:space_time_approx}
    Grant Assumptions \ref{ass:bounded}--\ref{ass:compact}, let $ \beta \geq 2 $ be an integer and $ n \geq 2, $ and put
    \begin{equation} \label{eq:R_K_choice}
        R := 2 + \sqrt{(8 \beta + 16) \log (2 d n)}, \quad \eta := \min \Big\{ \frac{\ell}{2}, \, \frac{\pi \sigma_*^4}{12 d (5 R + 3)} \Big\}, \quad K := \Big\lceil \frac{\ell}{\eta} (\beta + 2) \log n \Big\rceil,
    \end{equation}
    so that $ R \asymp \sqrt{\log n} $ and $ K \lesssim (\log n)^{3/2}, $ and
    \begin{equation} \label{eq:Q_n_def}
        Q_n := 16 C_{\mathrm{F}} \sqrt{d} \, R K^2 (1 + \eta^{-1}) \leq \sigma^{-c} (\log n)^4,
    \end{equation}
    where $ C_{\mathrm{F}} \geq 1, \tilde{N} $ and $ L_0 $ are the constants of Lemma \ref{lem:simult_approx}. There are constants $ C^\sharp, c^\sharp \geq 1 $ depending only on $ (d, \beta), $ and $ n_0 $ depending only on $ (d, \beta, \bar{B}, \ell, \sigma_*, C_t, \kappa_\Theta), $ with the following property. Suppose $ n \geq n_0, $
    \begin{equation} \label{eq:approx_regime}
        \tilde{N} \leq N \leq n, \qquad Q_n \mathrm{A}_\beta N^{-(\beta - 1)} \leq \min \{ \bar{B}, \, G/2, \, G_t/2 \},
    \end{equation}
    and that the budgets of Definition \ref{def:trial_test_classes} satisfy
    \begin{equation} \label{eq:budgets}
        \begin{aligned}
            & W_N \geq C^\sharp K (N+1)^d, \quad L = L_0 + \lceil \log_2 (K + 1) \rceil + 1, \quad P_0 \geq C^\sharp K \big( (N+1)^d + L \big), \\
            & \Lambda \geq n^{c^\sharp}, \quad \mathcal{A} \geq n^{c^\sharp}.
        \end{aligned}
    \end{equation}
    Then $ \mathcal{H}_N $ contains an $ h_N $ such that $ \varphi_N := h_N - h^* $ satisfies
    \begin{equation} \label{eq:space_time_bounds}
        \begin{aligned}
            \| \varphi_N \|_{L^4 (\nu \otimes \rho)} + \sup_{t \in I} \| \varphi_N (t, \cdot) \|_{\rho_t} &\leq 4 Q_n \mathrm{A}_\beta N^{-\beta}, \\
            \| | \nabla \varphi_N | \|_{L^4 (\nu \otimes \rho)} + \| \partial_t \varphi_N \|_{L^2 (\nu \otimes \rho)} &\leq 4 Q_n \mathrm{A}_\beta N^{-(\beta - 1)},
        \end{aligned}
    \end{equation}
    so that \eqref{eq:approx_error_eN} holds with $ e_N := 8 Q_n \mathrm{A}_\beta N^{-(\beta - 1)}. $
\end{lemma}
\begin{proof}
    The residual is a space-time object, so $ h $ must depend on $ t $ as well as $ x. $ A class of resolution $ N $ in $ d + 1 $ variables would have complexity $ N^{d+1}. $ The proof shows that the time direction costs only a polylogarithmic factor, because $ h^* $ is analytic in $ t $ on the window while being merely controlled in $ W^{\beta, \infty} $ in $ x. $ Norms of vectors are Euclidean, and for a function $ f $ on $ Q_R $ we write $ \| f \|_{W^{\beta, \infty}(Q_R)} := \max_{|\gamma| \leq \beta} \sup_{Q_R} | \partial^\gamma f |. $ The bound in \eqref{eq:Q_n_def} follows from $ R \leq C_\star \sqrt{\log n}, $ $ \eta^{-1} \leq \sigma^{-c} R $ and $ K \leq \sigma^{-c} R \log n. $

    Step 1: holomorphic extension in $ t. $ Let $ S_\eta := \{ z \in \mathbb{C} : \mathrm{dist}(z, I) < \eta \} $ and write $ z = s + i \tau. $ Since $ \sigma^2_* = 1 - e^{-2 (T - t_1)} \leq 2 (T - t_1) $ and $ \eta \leq \sigma_*^2 / 8, $ every $ z \in S_\eta $ has $ s < t_1 + \eta < T. $ The functions $ \mu_{T-z} = e^{-(T-z)} $ and $ \sigma^2_{T-z} = 1 - e^{-2(T-z)} $ are entire, with $ | \mu_{T-z} | = e^{-(T-s)} \leq 1 $ and $ \partial_z \mu_{T-z} = \mu_{T-z}, \, \partial_z \sigma^2_{T-z} = - 2 \mu_{T-z}^2. $ Moreover $ s \mapsto \sigma^2_{T-s} $ is nonincreasing with derivative in $ [-2, 0], $ so $ \sigma^2_{T-s} \geq \sigma_*^2 - 2 \eta $ for $ s < t_1 + \eta, $ and $ | \sigma^2_{T-z} - \sigma^2_{T-s} | = e^{-2(T-s)} | e^{2 i \tau} - 1 | \leq 2 \eta, $ hence $ | \sigma^2_{T-z} | \geq \sigma_*^2 - 4 \eta \geq \sigma_*^2 / 2 $ on $ S_\eta. $ By Bayes' rule, cancelling the factor $ e^{- |x|^2 / 2 \sigma^2_{T-t}} $ from numerator and denominator,
    \begin{equation} \label{eq:h_u_frac_form}
        h^* (t,x) = \frac{\int w (x_0) e^{u_t (x, x_0)} p_0 (dx_0)}{\int e^{u_t (x, x_0)} p_0 (dx_0)}, \quad u_z (x, x_0) := \frac{\mu_{T-z} \, x \cdot x_0}{\sigma^2_{T-z}} - \frac{\mu^2_{T-z} |x_0|^2}{2 \sigma^2_{T-z}},
    \end{equation}
    and for $ x \in Q_R $ both integrals extend holomorphically to $ z \in S_\eta, $ their integrands being holomorphic in $ z $ and locally uniformly bounded because $ p_0 $ is supported in the unit cube. For $ z \in S_\eta, $ $ | \partial_z (\mu / \sigma^2) | = | \mu / \sigma^2 + 2 \mu^3 / \sigma^4 | \leq 2 \sigma_*^{-2} + 8 \sigma_*^{-4} \leq 10 \sigma_*^{-4} $ and $ | \partial_z (\mu^2 / 2 \sigma^2) | = | \mu^2/\sigma^2 + \mu^4 / \sigma^4 | \leq 6 \sigma_*^{-4}, $ and $ | x \cdot x_0 | \leq |x|_1 |x_0|_\infty \leq d R, $ $ | x_0 |^2 \leq d. $ Since $ u_s $ is real for real $ s $ and the vertical segment from $ s $ to $ z $ stays in $ S_\eta, $
    \begin{equation*}
        | \mathrm{Im} \, u_z (x, x_0) | \leq | \tau | \sup_{S_\eta} | \partial_z u | \leq \eta \, \sigma_*^{-4} d (10 R + 6), \quad \text{hence} \quad \mathrm{osc}_{x_0} \mathrm{Im} \, u_z \leq \frac{\pi}{3}
    \end{equation*}
    by the choice of $ \eta. $ Then all numbers $ e^{i \, \mathrm{Im} \, u_z (x, x_0)} $ lie in a sector of half-angle $ \pi / 6, $ so $ | \int f e^{u_z} p_0 (dx_0) | \leq \sec(\pi/6) \sup_{\mathrm{supp} \, p_0} |f| \, | \int e^{u_z} p_0 (dx_0) | $ for bounded measurable $ f, $ and in particular the denominator does not vanish. Writing $ \langle f \rangle := \int f e^{u_z} p_0 / \int e^{u_z} p_0, $ we get $ | \langle f \rangle | \leq \sec (\pi / 6) \sup |f| $ and $ | h^* | \leq \sec(\pi/6) \bar{B} $ on $ S_\eta \times Q_R. $ For the $ x $-derivatives, $ u_z $ is affine in $ x $ with $ \partial_{x_k} u_z = \lambda_k (x_0) := \mu_{T-z} x_{0,k} / \sigma^2_{T-z}, $ $ | \lambda_k | \leq 2 \sigma_*^{-2}, $ so $ \partial_{x_k} \langle f \rangle = \langle f \lambda_k \rangle - \langle f \rangle \langle \lambda_k \rangle $ for $ f $ independent of $ x. $ Iterating, $ \partial_x^\gamma h^* $ with $ |\gamma| = m $ is the joint cumulant of $ (w, \lambda_{k_1}, \ldots, \lambda_{k_m}) $ under $ \langle \cdot \rangle, $ that is $ \sum_\pi (-1)^{|\pi| - 1} (|\pi| - 1)! \prod_{B \in \pi} \langle \prod_{i \in B} \cdot \rangle $ over the $ \mathfrak{B}_{m+1} $ partitions $ \pi $ of the $ m + 1 $ symbols. Each product is bounded by $ \sec(\pi/6)^{m+1} \bar{B} (2 \sigma_*^{-2})^m, $ so, with $ C_\beta $ as in \eqref{eq:A_beta_C_beta_def},
    \begin{equation} \label{eq:eta_asymp}
        \sup_{S_\eta \times Q_R} | \partial_x^\gamma h^* | \leq \mathfrak{B}_{m + 1} m! \, (2 \sec (\pi/6))^{m+1} \bar{B} \sigma_*^{-2m} \leq M_\beta := C_\beta \bar{B} \sigma_*^{-2 \beta}, \quad |\gamma| = m \leq \beta,
    \end{equation}
    and $ \partial_x^\gamma h^* (\cdot, x) $ is holomorphic on $ S_\eta $ for every $ x \in Q_R. $ Note $ M_\beta (2R)^{\beta - 1} = \mathrm{A}_\beta. $

    Step 2: truncation in $ t. $ Put $ \mathrm{r} := 1 + 2 \eta / \ell \in (1, 2] $ and let $ \Pi_K $ be the Chebyshev truncation of Lemma \ref{lem:chebyshev}, acting in $ t $ only. Since $ \log (1 + y) \geq y/2 $ on $ [0, 1], $ $ \mathrm{r}^{-K} \leq e^{- K \eta / \ell} \leq n^{-(\beta + 2)}, $ and $ \mathrm{r} / (\mathrm{r} - 1) \leq \ell / \eta. $ Lemma \ref{lem:chebyshev} applied to $ \partial_x^\gamma h^* (\cdot, x), $ $ | \gamma | \leq 1, $ with $ M = M_\beta, $ and the fact that $ \Pi_K $ commutes with $ \partial_x^\gamma, $ give uniformly in $ (t, x) \in I \times Q_R $
    \begin{gather*}
        | h^* - \Pi_K h^* | \leq M_\beta \ell \eta^{-1} n^{-(\beta + 2)}, \qquad | \nabla (h^* - \Pi_K h^*) | \leq \sqrt{d} M_\beta \ell \eta^{-1} n^{-(\beta+2)}, \\
        | \partial_t (h^* - \Pi_K h^*) | \leq 10 M_\beta (K + 1)^2 \ell^2 \eta^{-3} n^{-(\beta + 2)}.
    \end{gather*}
    Let $ n_0 $ be so large that $ \sqrt{d} \eta^{-1} \leq n^2 $ and $ 10 (K+1)^2 \eta^{-3} \leq n^3 $ for $ n \geq n_0, $ this is possible since the left-hand sides are $ \sigma^{-c} \mathrm{polylog}(n). $ Using $ M_\beta \leq \mathrm{A}_\beta, $ $ \ell \leq 1 $ and $ N \leq n, $ the three errors are then at most $ \mathrm{A}_\beta N^{-\beta}, $ $ \mathrm{A}_\beta N^{-\beta} $ and $ \mathrm{A}_\beta N^{-(\beta - 1)} $ respectively.

    Step 3: approximation of the coefficients. Write $ \Pi_K h^* (t, x) = \sum_{j = 0}^K c_j^* (x) T_j (\tilde{t}), $ $ \tilde{t} := 2 (t - t_0) / \ell - 1. $ Since $ \partial_x^\gamma c_j^* $ is the $ j $-th Chebyshev coefficient of $ \partial_x^\gamma h^* (\cdot, x), $ Lemma \ref{lem:chebyshev}(ii) and \eqref{eq:eta_asymp} give $ \| c_j^* \|_{W^{\beta, \infty} (Q_R)} \leq 2 M_\beta \mathrm{r}^{-j}. $ Put $ S_j := 2 M_\beta (2R)^\beta \mathrm{r}^{-j} $ and $ f_j (y) := c_j^* (2 R y - R \mathbf{1}) / S_j, $ $ y \in [0,1]^d, $ since $ 2R \geq 1, $ $ f_j $ lies in the unit ball of $ W^{\beta, \infty} ((0,1)^d). $ Let $ \Phi_j $ be the network of Lemma \ref{lem:simult_approx} for $ f_j $ and put $ \hat{c}_j (x) := S_j R(\Phi_j) ((x + R \mathbf{1}) / 2R). $ By \eqref{eq:simult_approx}, with $ | \nabla | \leq \sqrt{d} \max_k | \partial_k |, $ on $ Q_R $
    \begin{gather*}
        | \hat{c}_j - c_j^* | \leq C_{\mathrm{F}} S_j N^{-\beta} = 4 C_{\mathrm{F}} R \mathrm{A}_\beta \mathrm{r}^{-j} N^{-\beta}, \\
        | \nabla (\hat{c}_j - c_j^*) | \leq \frac{\sqrt{d} C_{\mathrm{F}} S_j}{2R} N^{-(\beta - 1)} = 2 \sqrt{d} C_{\mathrm{F}} \mathrm{A}_\beta \mathrm{r}^{-j} N^{-(\beta - 1)}.
    \end{gather*}
    Set $ H_N (t, x) := \sum_{j=0}^K \hat{c}_j (x) T_j (\tilde{t}). $ Using $ | T_j | \leq 1, $ $ | T_j' | \leq j^2 \leq K^2 $ on $ [-1, 1], $ $ \partial_t T_j (\tilde{t}) = \frac{2}{\ell} T_j' (\tilde{t}) $ and $ \sum_{j \leq K} \mathrm{r}^{-j} \leq \ell / \eta, $ on $ I \times Q_R $
    \begin{gather*}
        | H_N - \Pi_K h^* | \leq 4 C_{\mathrm{F}} R \ell \eta^{-1} \mathrm{A}_\beta N^{-\beta}, \qquad | \nabla (H_N - \Pi_K h^*) | \leq 2 \sqrt{d} C_{\mathrm{F}} \ell \eta^{-1} \mathrm{A}_\beta N^{-(\beta - 1)}, \\
        | \partial_t (H_N - \Pi_K h^*) | \leq 8 C_{\mathrm{F}} R K^2 \eta^{-1} \mathrm{A}_\beta N^{-\beta}.
    \end{gather*}

    Step 4: bounds on $ I \times Q_R, $ constraints, and tails. By Steps 2--3 and the definition of $ Q_n, $ on $ I \times Q_R $
    \begin{equation} \label{eq:inner_errors}
        | H_N - h^* | \leq Q_n \mathrm{A}_\beta N^{-\beta}, \qquad | \nabla (H_N - h^*) | + | \partial_t (H_N - h^*) | \leq Q_n \mathrm{A}_\beta N^{-(\beta - 1)},
    \end{equation}
    and by \eqref{eq:approx_regime} each left-hand side is at most $ \bar{B}, $ resp.\ $ \min \{ G/2, G_t/2 \}. $ Let $ h_N := (\Theta \circ H_N)^\chi. $ Step 5 shows that $ H_N $ is a network of the required architecture. Constraints: for $ (t, x) \in I \times \mathbb{R}^d $ put $ y := \chi_R (x) \in Q_R. $ By the chain rule, $ | \Theta' | \leq 1 $ and $ | D \chi_R | \leq 1, $ $ | \nabla h_N (t,x) | \leq | \nabla H_N (t, y) | \leq | \nabla h^* (t, y) | + G/2 \leq G $ and $ | \partial_t h_N (t,x) | \leq | \partial_t H_N (t, y) | \leq | \partial_t h^* (t, y) | + G_t / 2 \leq G_t, $ using $ | \nabla h^* | \leq G/2 $ and $ | \partial_t h^* | \leq G_t/2 $ on $ Z $ (Proposition \ref{prop:classes_props}(iv)). Hence $ h_N \in \mathcal{H}_N. $ Inside: on $ I \times Q_{R-1} $ the clip is the identity with $ D \chi_R = I, $ and $ H_N $ takes values in $ (-\bar{B}, 2 \bar{B}] $ since $ 0 < h^* \leq \bar{B}, $ as $ \Theta = \mathrm{id} $ on $ [-\bar{B}, 2 \bar{B}], $ also $ \Theta' = 1 $ there, so $ \varphi_N = H_N - h^* $ together with its first derivatives on $ I \times Q_{R-1}, $ and \eqref{eq:inner_errors} applies. Tails: on $ \{ |x|_\infty > R - 1 \} $ we have $ | \varphi_N | \leq 3 \bar{B} $ because $ h_N \in [-2 \bar{B}, 3 \bar{B}] $ and $ h^* \in (0, \bar{B}], $ $ | \nabla \varphi_N | \leq 3 G / 2, $ and $ | \partial_t \varphi_N | \leq G_t + | \partial_t h^* |. $ Under $ \rho_t, $ $ x = \mu_{T-t} x_0 + \sigma_{T-t} \xi $ with $ | x_0 |_\infty \leq 1 $ and $ \xi \sim \mathcal{N}(0, I_d), $ so a union bound over coordinates and the choice of $ R $ give $ p_R := \rho_t (|x|_\infty > R - 1) \leq 2 d e^{-(R-2)^2/2} \leq n^{-(4 \beta + 8)}. $ By harmonicity, $ | \partial_t h^* | \leq | x + 2 s | \, | \nabla h^* | + | \Delta h^* | \leq 3 \sigma_*^{-2} (|x| + \sqrt{d}) \cdot 2 \sqrt{d} \bar{B} \sigma_*^{-2} + d C_2 \bar{B} \sigma_*^{-4} $ on $ I $ by Lemma \ref{lem:bounded_guidance} and \eqref{eq:h_der_bounds}, and $ \| |x| \|_{L^4 (\rho_t)} \leq \sqrt{d} + (d^2 + 2d)^{1/4} \leq 3 \sqrt{d}, $ so $ C_h := \sup_{t \in I} \| \partial_t h^* \|_{L^4 (\rho_t)} \leq (24 + C_2) d \bar{B} \sigma_*^{-4}. $ Enlarge $ n_0 $ so that $ 3 G + 2 G_t + 2 C_h \leq 2 \mathrm{A}_\beta n^3 $ for $ n \geq n_0 $ (the left-hand side is $ \sigma^{-c} \sqrt{\log n} $). Splitting each norm into $ I \times Q_{R-1} $ and its complement, and using Hölder on the complement,
    \begin{align*}
        \| \varphi_N (t, \cdot) \|_{L^4 (\rho_t)} + \| \varphi_N (t, \cdot) \|_{\rho_t} &\leq 2 Q_n \mathrm{A}_\beta N^{-\beta} + 6 \bar{B} p_R^{1/4} \leq 4 Q_n \mathrm{A}_\beta N^{-\beta}, \\
        \| | \nabla \varphi_N (t, \cdot) | \|_{L^4 (\rho_t)} + \| \partial_t \varphi_N (t, \cdot) \|_{\rho_t} &\leq 2 Q_n \mathrm{A}_\beta N^{-(\beta - 1)} + \big( \tfrac{3}{2} G + G_t + C_h \big) p_R^{1/4} \\
        &\leq 4 Q_n \mathrm{A}_\beta N^{-(\beta - 1)},
    \end{align*}
    because $ p_R^{1/4} \leq n^{-(\beta + 2)}, $ $ N \leq n $ and $ n \geq 2 $ give $ 6 \bar{B} p_R^{1/4} \leq 6 \bar{B} n^{-2} N^{-\beta} \leq 2 \mathrm{A}_\beta N^{-\beta} \leq 2 Q_n \mathrm{A}_\beta N^{-\beta} $ (as $ \bar{B} \leq \mathrm{A}_\beta $ and $ Q_n \geq 1 $), and $ ( \frac{3}{2} G + G_t + C_h ) p_R^{1/4} \leq ( \frac{3}{2} G + G_t + C_h ) n^{-3} N^{-(\beta - 1)} \leq \mathrm{A}_\beta N^{-(\beta - 1)}. $ For the $ \partial_t $ term on the complement we used $ \| \partial_t \varphi_N \mathbf{1} \|_{\rho_t} \leq \| G_t + | \partial_t h^* | \|_{L^4(\rho_t)} p_R^{1/4}. $ Taking the supremum over $ t \in I, $ and noting that $ L^4(\nu \otimes \rho) $ and $ L^2 (\nu \otimes \rho) $ norms are bounded by the corresponding suprema over $ t, $ gives \eqref{eq:space_time_bounds}.

    Step 5: the network. All constants $ C^\sharp, c^\sharp $ below depend only on $ (d, \beta) $ and may be enlarged from line to line, $ n_0 $ is enlarged finitely many times. By Lemma \ref{lem:simult_approx}, each $ \Phi_j $ has at most $ L_0 $ hidden layers, at most $ C_{\mathrm{F}} (N+1)^d $ nonzero weights (hence width at most $ C_{\mathrm{F}} (N+1)^d $) and weights of modulus at most $ C_{\mathrm{F}} N^{C_{\mathrm{F}}}. $ For a network with at most $ L_0 $ hidden layers, widths at most $ W' \geq d + 1, $ weights of modulus at most $ \Lambda' \geq 1 $ and inputs in $ [-1, 1]^{d+1}, $ induction over the layers gives hidden activations bounded by $ (2 W' \Lambda')^{3^{L_0 + 1}}, $ since each layer maps $ a := \max(1, \| x^{(\ell)} \|_\infty) $ to at most $ (2 W' \Lambda' a)^3. $ On $ Z $ the inputs $ y = (x + R \mathbf{1})/2R \in [0,1]^d $ and $ \tilde{t} \in [-1,1] $ are affine in $ z, $ so they are absorbed in the first-layer parameters.

    (a) Coefficient blocks. Realize $ \hat{c}_j $ as $ \Phi_j $ composed with the input map, with output weights multiplied by $ S_j. $ Since $ S_j \leq 4 M_\beta (2 R)^\beta = \sigma^{-c} \mathrm{polylog}(n) $ and $ \sup_{Q_R} | \hat{c}_j | \leq 2 M_\beta + C_{\mathrm{F}} S_j, $ for $ n \geq n_0 $ we have $ S_j \leq n $ and $ M_c := 1 + \max_j \sup_{Q_R} | \hat{c}_j | \leq n. $

    (b) Chebyshev block. With $ T_0 = 1, T_1 = \tilde{t}, T_{2m} = 2 T_m T_m - 1, T_{2m+1} = 2 T_m T_{m+1} - T_1, $ hidden layer $ k \leq k_* := \lceil \log_2 (K+1) \rceil $ of this block outputs, as affine functions of its units, all $ T_j (\tilde{t}) $ with $ j \leq \min \{ 2^k, K \}: $ each new $ T_j $ is a product of two values available at layer $ k - 1, $ realized by four units via Lemma \ref{lem:exact_arithmetic}(iv) with $ M = 2, $ and each value already available, including $ T_1, $ is carried by four units via Lemma \ref{lem:exact_arithmetic}(iii) with $ M = 2. $ This is exact because $ | T_j | \leq 1 $ on $ [-1, 1]. $ Each layer has at most $ 8 (K + 1) $ units, each with at most $ 17 $ nonzero incoming parameters of modulus at most $ 4, $ and activations at most $ 64 $ on $ Z. $

    (c) Alignment and products. Let $ L_1 := L - 1 = L_0 + \lceil \log_2 (K+1) \rceil, $ which exceeds the depth of every block. Pad every block to $ L_1 $ hidden layers by carrying its outputs through Lemma \ref{lem:exact_arithmetic}(iii) with $ M = M_c $ (resp.\ $ M = 2 $). Hidden layer $ L $ computes the $ K + 1 $ products $ \hat{c}_j T_j $ by Lemma \ref{lem:exact_arithmetic}(iv) with $ M := 2 M_c + 2, $ and the readout sums them, which gives $ H_N $ exactly. The resulting network has depth $ L, $ width at most $ (K + 1) (C_{\mathrm{F}} (N+1)^d + 20) \leq C^\sharp K (N+1)^d, $ at most $ 6 (K + 1) C_{\mathrm{F}} (N+1)^d + 320 (K+1) L \leq C^\sharp K ((N+1)^d + L) $ nonzero parameters (biases created by the input map, and the first carrying layer after each coefficient block, whose units read the whole last layer of $ \Phi_j, $ account for the factor $ 6 $), parameters of modulus at most $ n^{c^\sharp}, $ and hidden activations on $ Z $ at most $ n^{c^\sharp} $ by the induction above, for $ n \geq n_0 $ (using $ N \leq n, $ $ M_c \leq n, $ $ 2/\ell + 2 t_1 / \ell + 2 \leq n $). Under \eqref{eq:budgets} and after padding widths with inactive units, $ H_N \in \mathcal{NN}_\varrho (W_N, L, \Lambda, P_0, \mathcal{A}). $
\end{proof}

\section{Removing the inf-sup assumption} \label{app:infsup}

\begin{definition}(Projected residual) \label{def:proj_res}
    Let $ \Psi \subset \mathcal{V} $ be a test class carrying a symmetric, positively homogeneous scale $ \| \cdot \|_\bullet $ -- for $ \Psi_N $ that of \eqref{eq:bullet_norm} -- and let $ \tau \geq 0. $ Set 
    \begin{equation} \label{eq:proj_res}
        \| \psi \|^2_{\mathcal{V}, \tau} := \| \psi \|^2_{\mathcal{V}} + \tau^2 \| \psi \|^2_\bullet, \quad 
        \mathcal{N}_{\Psi, \tau} (h) := \underset{0 \neq \psi \in \Psi}{\sup} \frac{A (h, \psi)}{\| \psi \|_{\mathcal{V}, \tau}} \leq \mathcal{N}_\Psi (h)\leq \| \mathcal{R}[h] \|_{\mathcal{V}^*},
    \end{equation}
    where $ \mathcal{N}_\Psi := \mathcal{N}_{\Psi, 0}, $
    the largest amount of the residual that the test class can detect at slack $ \tau. $ The second inequality is immediate from $ \Psi \subseteq \mathcal{V}, $ and the first from $ \| \cdot \|_{\mathcal{V}, \tau} \geq \| \cdot \|_\mathcal{V}. $ Together they make the approximation input of Lemma \ref{lem:weak_res_approx} available at every $ \tau $ without any further hypothesis. The slack is $ \tau = 0 $ at the population level and $ \tau = \delta_n $ in the estimator \eqref{eq:h_empirical}, where it is what makes the empirical and population adversary norms comparable multiplicatively.
\end{definition}

\begin{lemma}(Fenchel identity for a cone) \label{lem:fenchel_cone}
    Let $ \Psi $ satisfy \ref{property:cone} and let $ \alpha > 0. $ Then
    \begin{equation}
        \underset{\psi \in \Psi}{\sup} \left[ A(h, \psi) - \frac{\alpha}{2} \| \psi \|^2_{\mathcal{V}} \right] = \frac{\mathcal{N}_\Psi (h)^2}{2 \alpha}.
    \end{equation}
    The same identity holds with $ \| \cdot \|_{\mathcal{V}} $ replaced by any other symmetric positively homogeneous functional -- the proof uses nothing else -- in particular by $ \| \cdot \|_{\mathcal{V}, \tau}, $ by the empirical norm $ \| \cdot \|_{\hat{\mathcal{V}}} $ and by the ridged empirical norm $ \| \cdot \|_{\hat{\mathcal{V}}, \tau} $ of \eqref{eq:h_empirical}.
\end{lemma}
\begin{proof}
    Every nonzero $ \psi \in \Psi $ factors as $ \psi = r \hat{\psi} $ with $ r := \| \psi \|_{\mathcal{V}} > 0 $ and $ \| \hat{\psi} \|_{\mathcal{V}} = 1. $ By \ref{property:cone} both $ \hat{\psi} $ and $ r' \hat{\psi} $ lie in $ \Psi $ for every $ r' \geq 0. $ Hence the supremum may be computed in two stages, first over the radius and then over the direction:
    \begin{equation}
        \underset{\psi \in \Psi}{\sup} \left[ A(h, \psi) - \frac{\alpha}{2} \| \psi \|^2  \right] = \underset{\hat{\psi}}{\sup} \, \underset{r \geq 0}{\sup} \left[ r A(h, \hat{\psi}) - \frac{\alpha}{2} r^2 \right] = \underset{\hat{\psi}}{\sup} \frac{(A (h, \hat{\psi})^2_+)}{2 \alpha},
    \end{equation}
    the inner supremum being attained at $ r = (A (h, \hat{\psi}))_+ / \alpha, $ which is admissible precisely because $ \Psi $ is a cone. Since $ \Psi $ is symmetric, $ \underset{\hat{\psi}}{\sup} (A(h, \hat{\psi}))_+ = \underset{\hat{\psi}}{\sup} | A(h, \hat{\psi}) | = \mathcal{N}_\Psi (h). $
\end{proof}

\begin{lemma}(Coefficient envelopes) \label{lem:coeffs_envelopes}
Under Assumptions \ref{ass:bounded}--\ref{ass:compact},
\begin{equation}
    \Lambda_* := \underset{t \in I}{\sup} \left( \| | x + s | \|_{L^4 (\rho_t)} + \| d + \nabla \cdot s \|_{L^4 (\rho_t)} + \| | x + s | | s | \|_{L^4 (\rho_t)} + \| \partial_t \log \rho_t \|_{L^4 (\rho_t)} \right) \lesssim d^2 \sigma^{-4}_*.
\end{equation}
\end{lemma}
\begin{proof}
    The first term is Lemma \ref{lem:score_identity}, together with $ \| Z \|_{L^4} \leq C \| Z \|_{\psi_2}. $ For the second, \eqref{eq:hessian} gives the deterministic bound $ | \nabla \cdot s | = | \mathrm{tr} \nabla^2 \log \rho_t | \leq d \sigma^{-2}_{T-t} (1 + d \mu^2_{T-t} \sigma^{-2}_{T-t}). $ For the third \eqref{eq:score_identity_1} gives $ | s | \leq \sigma^{-2}_{T-t} ( |x | + \sqrt{d}), $ whence $ \| | s | \|_{\psi_2, \rho_t} \lesssim \sqrt{d} \sigma^{-2}_{T-t} $ exactly as in Lemma \ref{lem:score_identity}, and Cauchy-Schwarz in $ L^8 $ bounds the product. For the fourth, Fokker-Planck equation for $ d X_\tau = - X_\tau d \tau + \sqrt{2} d B_\tau $ reads $ \partial_\tau p_\tau = \nabla \cdot (x p_\tau) + \Delta p_\tau $ and dividing by $ p_\tau $ and using $ \rho_t = p_{T-t} $ (so that $ \partial_t = - \partial_{\tau}) $ gives $ \partial_t \log \rho_t = - ( d + x \cdot s + \nabla \cdot s + | s |^2), $ each summand of which was just bounded in $ L^4 (\rho_t). $ All bounds are largest at $ t = t_1, $ where $ \sigma_{T-t} = \sigma_*. $
\end{proof}

\begin{theorem}(Coercivity without an inf-sup constant) \label{thm:coercivity}
    Let Assumptions \ref{ass:bounded}--\ref{ass:compact} hold and let $ \mathcal{H}_N, \Psi_N $ satisfy \ref{property:cone}--\ref{property:absorption} and the following two conditions: every $ h \in \mathcal{H}_N $ is $ C^{1,2} $ on $ I \times \mathbb{R}^d $ with $ | h | \leq 3 \bar{B} $ and $ \sup_{I \times \mathbb{R}^d} ( | \partial_t h | + | \nabla h | + | \nabla^2 h | ) < \infty, $ and, for some $ C_0 \geq 0 $ and the scale $ \| \cdot \|_\bullet $ of Definition \ref{def:proj_res},
    \begin{equation} \label{eq:scale_condition}
        \| \omega (h - h') \|_\bullet \leq C_0 \quad \text{for all } h, h' \in \mathcal{H}_N.
    \end{equation}
    Let $ h_N \in \mathcal{H}_N $ satisfy \eqref{eq:approx_error_eN} with some $ e_N > 0, $ and let $ \tau \geq 0 $ and $ \delta \in (0,1). $ Then for every $ h \in \mathcal{H}_N, $ writing $ \varphi := h - h^* $ and $ I_\delta := [t_0, t_1 - \delta \ell], $
    \begin{equation} \label{eq:coercivity_thm}
        \int_{I_\delta} \| \nabla \varphi \|^2_{\rho_t} \nu (dt) \leq \frac{C}{\delta} \left(  \frac{1 + \ell}{\ell} \mathcal{D}(h) + \mathcal{N}_{\Psi_N, \tau} (h)^2 + C_0^2 \tau^2 + (1 + \Lambda_*^2 + \ell^{-2}) e^2_N \right),
    \end{equation}
    with $ C $ absolute and $ \Lambda_* \lesssim d^2 \sigma_*^{-4} $ the coefficient envelope of Lemma \ref{lem:coeffs_envelopes}. For the classes of Definition \ref{def:trial_test_classes} all hypotheses other than \eqref{eq:approx_error_eN} hold, with $ C_0 = 1, $ by Proposition \ref{prop:classes_props}(iii) and \eqref{eq:omega_phi_bullet_bound}, for $ \tau = 0 $ the scale condition is not used.
\end{theorem}
\begin{proof}
    By the first hypothesis and Lemma \ref{lem:bounded_guidance}, $ h $ satisfies the hypotheses of Theorem \ref{thm:weighted_energy_identity}, whose exact identity reads
    \begin{equation} \label{eq:weighted_energy_for_proof}
        2 \ell E_\omega + \| \varphi (t_0, \cdot) \|^2_{\rho_{t_0}} = \mathcal{D}(h) - 2 \ell A(h, \omega \varphi), \quad E_\omega = \int_I \omega \| \nabla \varphi \|^2_{\rho_t} \nu(dt).
    \end{equation}
    Both left-hand terms are favourable, so the whole problem is to bound the cross term $ A(h, \omega \varphi). $ The naive bound $ | A(h, \omega \varphi)| \leq \| \mathcal{R} [h] \|_{\mathcal{V}^*} \| \omega \varphi \|_{\mathcal{V}} $ is what produces the inf-sup constant, since the estimator controls $ \mathcal{N}_{\Psi_N} (h) $ and not $ \| \mathcal{R}[h] \|_{\mathcal{V}^*}. $ We split the test function into a part that lies in $ \Psi_N $ and a part that is small.

    Set $ \varphi_1 := h - h_N $ and $ \varphi_2 := h_N - h^*, $ so that $ \varphi = \varphi_1 + \varphi_2, $ and abbreviate $ E := E_\omega, \, D := \mathcal{D}(h), \, \mathcal{N} := \mathcal{N}_{\Psi_N, \tau}(h). $ Since $ A(h^*, \cdot) = 0 $ and $ A $ is linear in its first argument, $ A(h, \psi) = A(\varphi, \psi) $ for every admissible $ \psi. $ Throughout we use $ 0 \leq \omega \leq 1, $ hence $ \omega^2 \leq \omega, $ and the consequences $ \int_I \| \varphi_2 \|^2_{\rho_t} \nu \leq e_N^2, $ $ \int_I \| \nabla \varphi_2 \|^2_{\rho_t} \nu \leq \| |\nabla \varphi_2| \|^2_{L^4 (\nu \otimes \rho)} \leq e_N^2 $ of \eqref{eq:approx_error_eN}.

    (a) The in-class part. Both $ h $ and $ h_N $ lie in $ \mathcal{H}_N, $ so $ \omega \varphi_1 \in \Psi_N $ by \ref{property:absorption} and $ \| \omega \varphi_1 \|_\bullet \leq C_0 $ by \eqref{eq:scale_condition}. By the definition \eqref{eq:proj_res} of $ \mathcal{N}, $ $ | A (h, \omega \varphi_1) | \leq \mathcal{N} \| \omega \varphi_1 \|_{\mathcal{V}, \tau}. $ Since $ \omega $ depends only on $ t, $ $ \nabla (\omega \varphi_1) = \omega \nabla \varphi_1, $ and using $ \omega^2 \leq \omega $ and $ (a + b)^2 \leq 2 a^2 + 2 b^2 $ inside each integral,
    \begin{equation*}
        \| \omega \varphi_1 \|^2_{\mathcal{V}, \tau} \leq \int_I \omega ( \| \varphi_1 \|^2_{\rho_t} + \| \nabla \varphi_1 \|^2_{\rho_t} ) \nu (dt) + C_0^2 \tau^2 \leq 2 D + 2 E + 4 e^2_N + C_0^2 \tau^2.
    \end{equation*}
    The $ E $ produced here is the weighted energy on the left of \eqref{eq:weighted_energy_for_proof}. Young's inequality $ 2 \ell \mathcal{N} \kappa \leq \frac{\ell}{4} \kappa^2 + 4 \ell \mathcal{N}^2 $ with $ \kappa := \| \omega \varphi_1 \|_{\mathcal{V}, \tau} $ gives
    \begin{equation*}
        2 \ell | A (h, \omega \varphi_1) | \leq \frac{\ell}{2} D + \frac{\ell}{2} E + \ell e^2_N + \frac{\ell}{4} C_0^2 \tau^2 + 4 \ell \mathcal{N}^2.
    \end{equation*}

    (b) The out-of-class part, gradient term. By \eqref{eq:weak_res} the cross term $ A(\varphi, \omega \varphi_2) $ has three pieces. The last is $ - \int_I \omega \int \nabla \varphi \cdot \nabla \varphi_2 \rho_t \, dx \, \nu (dt), $ which by Cauchy--Schwarz against the weight $ \omega $ and $ ab \leq \frac{1}{8} a^2 + 2 b^2 $ is at most $ \frac{1}{8} E + 2 \int_I \omega \| \nabla \varphi_2 \|^2_{\rho_t} \nu \leq \frac{1}{8} E + 2 e^2_N. $

    (c) The out-of-class part, drift term. The piece $ \int_I \omega \int ((x + s) \cdot \nabla \varphi) \varphi_2 \rho_t \, dx \, \nu (dt) $ cannot be treated by Cauchy--Schwarz directly: that would produce $ \| | \nabla \varphi | \|_{L^4}, $ which the estimator does not control. For $ V := x + s $ and $ g $ bounded, $ C^1 $ and with bounded gradient,
    \begin{equation*}
        \int ( V \cdot \nabla \varphi) g \rho_t = - \int \varphi \, \nabla \cdot (g V \rho_t) = - \int \varphi [ (\nabla \cdot V) g + V \cdot \nabla g + (V \cdot s) g ] \rho_t,
    \end{equation*}
    where the boundary term at infinity vanishes because $ \varphi $ and $ g $ are bounded, $ |V| $ has at most linear growth and $ \rho_t $ has Gaussian tails, and $ \nabla \rho_t = s \rho_t. $ Take $ g := \omega \varphi_2 $ and note $ \nabla \cdot V = d + \nabla \cdot s. $ Cauchy--Schwarz in $ x, $ H\"older with the envelopes of Lemma \ref{lem:coeffs_envelopes}, and Cauchy--Schwarz in $ t $ give
    \begin{equation*}
        \begin{aligned}
        \left| \int_I \omega \int (V \cdot \nabla \varphi) \varphi_2 \rho_t \, dx \, \nu (dt) \right| &\leq \Lambda_* \sqrt{D} \left( 2 \| \varphi_2 \|_{L^4 (\nu \otimes \rho)} + \| | \nabla \varphi_2 | \|_{L^4 (\nu \otimes \rho)} \right) \\
        &\leq 3 \Lambda_* \sqrt{D} e_N \leq \frac{1}{8} D + 18 \Lambda_*^2 e^2_N.
        \end{aligned}
    \end{equation*}

    (d) The out-of-class part, time term. With $ F(t) := \int \varphi \, \omega \varphi_2 \, \rho_t \, dx, $ the function $ F $ is absolutely continuous on $ I $ with $ F' = \int \partial_t (\varphi \omega \varphi_2 \rho_t) \, dx, $ by dominated convergence: $ \varphi, \varphi_2 $ are bounded, $ | \partial_t \varphi |, | \partial_t \varphi_2 | $ grow at most linearly in $ x $ by the first hypothesis and harmonicity, and $ \partial_t \rho_t = \rho_t \, \partial_t \log \rho_t $ with $ \partial_t \log \rho_t $ of polynomial growth. Hence, with $ \nu(dt) = dt / \ell, $
    \begin{equation*}
        \int_I \omega \int \partial_t \varphi \, \varphi_2 \rho_t \, dx \, \nu(dt) = \frac{F(t_1) - F(t_0)}{\ell} - \int_I \int \varphi \, \partial_t (\omega \varphi_2 \rho_t ) \, dx \, \nu(dt).
    \end{equation*}
    The boundary term at $ t_1 $ vanishes because $ \omega (t_1) = 0. $ At $ t_0, $ $ \ell^{-1} | F (t_0) | \leq \frac{1}{4 \ell} \| \varphi (t_0) \|^2_{\rho_{t_0}} + \frac{1}{\ell} \| \varphi_2 (t_0) \|^2_{\rho_{t_0}} \leq \frac{1}{4 \ell} \| \varphi (t_0) \|^2_{\rho_{t_0}} + \frac{1}{\ell} e_N^2. $ For the bulk term, $ \partial_t (\omega \varphi_2 \rho_t) = \rho_t [ \dot{\omega} \varphi_2 + \omega \partial_t \varphi_2 + \omega \varphi_2 \partial_t \log \rho_t ] $ with $ | \dot{\omega} | = \ell^{-1}, $ so by Cauchy--Schwarz, H\"older and Lemma \ref{lem:coeffs_envelopes} it is at most $ \sqrt{D} ( \ell^{-1} \| \varphi_2 \|_{L^2(\nu \otimes \rho)} + \| \partial_t \varphi_2 \|_{L^2 (\nu \otimes \rho)} + \Lambda_* \| \varphi_2 \|_{L^4 (\nu \otimes \rho)} ) \leq (1 + \ell^{-1} + \Lambda_*) \sqrt{D} e_N \leq \frac{1}{8} D + 2 (1 + \ell^{-1} + \Lambda_*)^2 e^2_N. $

    (e) Assembly. Multiplying the bounds of (b), (c), (d) by $ 2 \ell $ and inserting them with (a) into \eqref{eq:weighted_energy_for_proof}: the terms in $ E $ total $ \frac{\ell}{2} E + \frac{\ell}{4} E = \frac{3 \ell}{4} E, $ absorbed by $ 2 \ell E, $ the boundary term of (d) contributes $ \frac{1}{2} \| \varphi (t_0) \|^2_{\rho_{t_0}}, $ absorbed by the favourable term, the terms in $ D $ total $ (1 + \ell) D, $ the projected residual contributes $ 4 \ell \mathcal{N}^2 $ and the slack $ \frac{\ell}{4} C_0^2 \tau^2. $ The terms in $ e_N^2 $ are $ \ell e_N^2 + 4 \ell e_N^2 + 36 \ell \Lambda_*^2 e_N^2 + 2 e_N^2 + 4 \ell (1 + \ell^{-1} + \Lambda_*)^2 e_N^2 \leq C \ell (1 + \Lambda_*^2 + \ell^{-2}) e_N^2, $ using $ 2 \leq \ell (1 + \ell^{-2}) $ and $ (1 + \ell^{-1} + \Lambda_*)^2 \leq 3 (1 + \ell^{-2} + \Lambda_*^2). $ Hence
    \begin{equation*}
        \frac{5 \ell}{4} E \leq (1 + \ell) D + 4 \ell \mathcal{N}^2 + \frac{\ell}{4} C_0^2 \tau^2 + C \ell (1 + \Lambda^2_* + \ell^{-2}) e^2_N,
    \end{equation*}
    and dividing by $ \ell $ gives $ E \leq C ( \frac{1 + \ell}{\ell} D + \mathcal{N}^2 + C_0^2 \tau^2 + (1 + \Lambda_*^2 + \ell^{-2}) e_N^2 ). $ Finally $ \omega (t) = (t_1 - t) / \ell \geq \delta $ on $ I_\delta, $ so $ \int_{I_\delta} \| \nabla \varphi \|^2_{\rho_t} \nu(dt) \leq \delta^{-1} \int_{I_\delta} \omega \| \nabla \varphi \|^2_{\rho_t} \nu (dt) \leq \delta^{-1} E, $ which is \eqref{eq:coercivity_thm}. The bound $ \Lambda_* \lesssim d^2 \sigma^{-4}_* $ is Lemma \ref{lem:coeffs_envelopes}.
\end{proof}

\section{Statistical guarantee} \label{app:statistical}

Let $ \{ (X_0^i, \varepsilon^i) \}_{i=1}^n $ be i.i.d. with $ X_0^i \sim p_0, \, \varepsilon^i \sim \mathcal{N}(0, I_d) $ independent, and $ X^i_{T-t} := \mu_{T-t} X_0^i + \sigma_{T-t} \varepsilon^i \sim \rho_t. $ With $ \hat{a}_t, \hat{\mathcal{D}} $ the empirical versions of $ a_t, \, \mathcal{D} $ and $ \| \cdot \|_{\hat{\mathcal{V}}} $ the empirical $\mathcal{V}$-norm, the estimator is
\begin{equation} \label{eq:h_empirical}
    \hat{h} \in \underset{h \in \mathcal{H}_N}{\mathrm{argmin}} \, \hat{\mathcal{D}}(h) + \mu \underset{\psi \in \Psi_N}{\sup} \left[ \int_I \hat{a}_t (h, \psi) \nu (dt) - \frac{\alpha}{2} \| \psi \|^2_{\hat{\mathcal{V}}, \delta_n} \right], \quad \| \psi \|^2_{\hat{\mathcal{V}}, \tau} := \| \psi \|^2_{\hat{\mathcal{V}}} + \tau^2 \| \psi \|^2_\bullet,
\end{equation}
with $ \hat{g} = \Pi_{G^*} (\nabla \hat{h} / (\hat{h} \lor \mathrm{b})) $ as in \eqref{eq:proj_guidance} and $ \delta_n $ as in \eqref{eq:delta_n}.

The second term in the adversary norm is a ridge at the statistical scale. $ \| \psi \|_\bullet / \| \psi \|_{\mathcal{V}} $ is unbounded on the cone (Remark \ref{rem:v_ball}) and no bound of the form $ \delta_n \| \psi \|_{\mathcal{V}} + \delta^2_n $ holds uniformly on $ \Psi_N. $ The naive estimator -- ridged in $ \| \cdot \|_{\hat{\mathcal{V}}} $ alone -- has an inner supremum that the deviation bounds do not control in directions where the empirical $\mathcal{V}$-norm nearly degenerates. As we will see at Theorem \ref{thm:dev_bound}, adding $ \delta_n^2 \| \cdot \|^2_\bullet $ removes the difficulty. The cost is one extra term of order $ \delta_n^2 $ in Theorem \ref{thm:coercivity}, taken there at $ \tau = \delta_n. $ 

\begin{align}
    G_{h, \psi} (z) := & \int_I \left[  \partial_t h \, \psi + ((x + s) \cdot \nabla h) \psi - \nabla h \cdot \nabla \psi \right] (t, X_{T-t}(z)) \nu(dt), \label{eq:G_h_psi_def} \\
    D_h(z) := & \int_I \left[ (h - w(x_0))^2 - (h^* - w(x_0))^2 \right] (t, X_{T-t}(z)) \nu (dt), \label{eq:D_h_def} \\
    Q_\psi (z) := & \int_I \left[ \psi^2 + | \nabla \psi |^2 \right] (t, X_{T-t} (z)) \nu (dt), \label{eq:Q_psi_def}
\end{align}
for $ h \in \mathcal{H}_N $ and $ \psi \in \Psi_N. $

All three quantities are ordinary empirical processes over a single i.i.d. sample and the integration in $ t $ having been absorbed into the summand. Indeed, put $ Z_i := (X_0^i, \varepsilon^i) \in \mathcal{Z} := \mathrm{supp}(p_0) \times \mathbb{R}^d, $ which are i.i.d., and write $ X_{T-t}(z) := \mu_{T-t} x_0 + \sigma_{T-t} e $ for $ z = (z_0, e). $

Writing $ P_n $ for the empirical and $ P $ for population measure of $ Z $
\begin{equation} \label{eq:three_empirical_things}
    \int_I \hat{a}_t (h, \psi) \nu(dt) = P_n G_{h, \psi}, ~ \hat{\mathcal{D}}(h) - \hat{\mathcal{D}}(h^*) = P_n D_h, ~ \| \psi \|^2_{\hat{\mathcal{V}}} = P_n Q_\psi,
\end{equation}
with population counterparts $ P G_{h, \psi} = A(h, \psi), ~ P D_h = \mathcal{D}(h) $ and $ P Q_\psi = \| \psi \|^2_{\mathcal{V}}. $ Only $ \mathcal{G} := \{ G_{h, \psi} \} $ has an unbounded envelope, and only through the factor $ x + s. $

We will also use that $ \mathcal{G} $ is linear in $ \psi $ and $ \mathcal{Q} $ quadratic, while $ \Psi_N $ is a cone whose scale is measured by $ \| \cdot \|_\bullet. $ Exact homogeneity therefore removes the radial variable from the empirical process problem: it suffices to work at generator scale $ \| \psi \|_\bullet \leq 1 $ and multiply back. No restriction of the adversary to a norm ball is imposed: with the ridge of \eqref{eq:h_empirical} the bounds below hold on all of $ \Psi_N. $

\begin{lemma}(Envelopes and variances) \label{lem:envelopes_and_variances}
    Let $ \mathcal{T} > 0 $ and, for $ z = (x_0, e), $ write $ \bar{V}(z) := \sigma_*^{-2} ( 2 \sqrt{d} + | e |). $ Under Assumptions \ref{ass:bounded}--\ref{ass:compact} and Definition \ref{def:trial_test_classes}:
    \begin{enumerate}[label=(\roman*), leftmargin=2.2em]
        \item (Uniform drift envelope) $ \underset{t \in I}{\sup} | (x + s) (t, X_{T-t} (z)) | \leq \bar{V}(z) $ for every $ z, $ and $ \| \bar{V} \|_{\psi_2} \leq C \sqrt{d} \sigma_*^{-2}. $

        \item (Envelope) $ | G_{h, \psi} | \leq \Gamma \| \psi \|_\bullet ( \max (G_t, G) + G \bar{V}) $ with $ \| \psi \|_\bullet $ the cone scale \eqref{eq:bullet_norm}, and for the envelope $ F := \Gamma (\max (G_t, G) + G \bar{V}) $ of $ \{ G_{h, \psi}: \| \psi \|_\bullet \leq 1 \}, $
        \begin{equation} \label{eq:envelope}
            \| F \|_{\psi_2} \leq C \sigma^{-c} \log n, \quad \text{hence } \| \underset{i \leq n}{\max} F (Z_i) \|_{\psi_1} \leq C \log n \| F \|_{\psi_2} \leq C \sigma^{-c} \log^2 n.
        \end{equation}

        \item (Bounded surrogates) Let $ G_{h, \psi}^{\mathcal{T}} $ be defined as in \eqref{eq:G_h_psi_def} with $ | x + s | $ replaced by $ (x + s) \mathbf{1} \{ | x + s| \leq \mathcal{T} \}. $ Then $ \| G_{h, \psi}^{\mathcal{T}} \|_\infty \leq \Gamma (G_t + G + G \mathcal{T}) \| \psi \|_\bullet $ and 
        \begin{equation} \label{eq:P_G_T_2}
            P (G_{h, \psi}^{\mathcal{T}})^2 \leq 3 (G_t^2 + G^2 \mathcal{T}^2 + G^2) \| \psi \|^2_{\mathcal{V}},
        \end{equation}
        moreover $ G_{h, \psi} = G^{\mathcal{T}}_{h, \psi} $ for every $ (h, \psi) $ simultaneously on the event $ \{ \underset{i \leq n}{\max} \bar{V}(Z_i) \leq \mathcal{T} \}, $ whose complement has probability at most $ 2 n e^{- c \mathcal{T}^2 \sigma^4_* / d}, $ and, if $ \mathcal{T} \geq C \sqrt{d} \sigma_*^{-2}, $ $ P | G_{h, \psi} - G^{\mathcal{T}}_{h, \psi} | \leq \Gamma G \| \psi \|_\bullet P [ \bar{V} \mathbf{1} \{ \bar{V} > \mathcal{T} \} ] \leq C \mathcal{T} \, \Gamma G \| \psi \|_\bullet e^{-c \mathcal{T}^2 \sigma^4_* / d}. $

        \item (Bernstein condition for $ \mathcal{D} $) $ \| D_h \|_\infty \leq 12 \bar{B}^2 $ and $ P D_h^2 \leq 16 \bar{B}^2 P D_h, $

        \item (Bernstein condition for $ \mathcal{Q} $) $ 0 \leq Q_\psi \leq \Gamma^2 \| \psi \|^2_\bullet $ and $ P Q_\psi^2 \leq \Gamma^2 \| \psi \|^2_\bullet P Q_\psi. $
    \end{enumerate}
\end{lemma}
\begin{proof}
    (i) By Tweedie's formula, $ x + s(t,x) = (1 - \sigma_{T-t}^{-2}) x + \mu_{T-t} \sigma^{-2}_{T-t} \mathbb{E} [ X_0 | x], $ so $ | x + s| \leq \sigma^{-2}_{T-t}  ( | x| + \sqrt{d} )$ using $ | 1 - \sigma^{-2} | \leq \sigma^{-2} $ for $ \sigma \leq 1, \mu \leq 1 $ and $ | X_0 | \leq \sqrt{d}. $ At $ x = X_{T-t} (z) = \mu_{T-t} x_0 + \sigma_{T-t} e $ we have $ |x| \leq \sqrt{d} + |e|, $ so the supremum over $ t \in I $ is at most $ \bar{V} (z). $ Since $  \| | e | \|_{\psi_2} \leq C \sqrt{d}, $ also $ \| \bar{V} \|_{\psi_2} \leq C \sqrt{d} \sigma_*^{-2}. $ 
    
    (ii) The three terms of \eqref{eq:G_h_psi_def} are bounded pointwise by $ G_t | \psi |, \, | x + s| G |\psi| $ and $ G | \nabla \psi |, $ using \eqref{eq:h_n_class}, whose suprema are over $ I \times \mathbb{R}^d $ and so are valid at the sample points $ (t, X_{T-t}(z)) $ -- this is what the clip of Definition \ref{def:trial_test_classes} is for. Grouping the first and third through $ G_t | \psi | + G | \nabla \psi | \leq \max (G_t, G) ( |\psi | + | \nabla \psi |) \leq \max (G_t, G) \Gamma \| \psi \|_\bullet, $ bounding the middle by $ \bar{V}(z) G \Gamma \| \psi \|_\bullet $ through (i), and integrating the probability measure $ \nu $ gives the envelope. For \eqref{eq:envelope}, $ F $ is affine in $ \bar{V}, $ so $ \| F \|_{\psi_2} \leq C \Gamma (\max (G_t, G) + G \sqrt{d} \sigma_*^{-2}). $ Since $ \Gamma \asymp 2 G_t $ and $ G_t \asymp \sigma^{-c} \sqrt{\log n} $ through $ R \asymp \sqrt{\log n}, $ this is $ C \sigma^{-c} \log n, $ and the standard $ \| \underset{i \leq n}{\max} \xi_i \|_{\psi_1} \leq C \log n \max_i \| \xi_i \|_{\psi_1} $ with $ \| \cdot \|_{\psi_1} \leq C \| \cdot \|_{\psi_2} $ gives $ C \sigma^{-c} \log^2 n. $ This term is tight against $ \delta_n^2 $ rather than of lower order.
    
    (iii) The sup bound is immediate. For \eqref{eq:P_G_T_2}, Cauchy-Schwarz in $ \nu $ gives $ ( G_{h, \psi}^{\mathcal{T}} )^2 \leq \int_I ( \ldots)^2 \nu (dt), $ and $ (\partial_t h \psi + (\cdot) \psi - \nabla h \cdot \nabla \psi)^2 \leq 3 (G_t^2 + G^2 \mathcal{T}^2) \psi^2 + 3 G^2 | \nabla \psi |^2 $ pointwise. Taking $ P $ and using $ P \int_I \psi^2 \nu = \| \psi \|^2_{L^2 (\nu \otimes \rho)} $ and likewise for $ \nabla \psi $ gives \eqref{eq:P_G_T_2}, since both are at most $ \| \psi \|^2_{\mathcal{V}}. $ On the event $ \{ \max_i \bar{V} (Z_i) \leq \mathcal{T} \} $ part (i) gives $ | x + s | (t, X_{T-t} (Z_i)) \leq \mathcal{T} $ for every $ t $ and every $ i, $ so the truncation is inactive on the sample. Its complement has probability at most $ 2n e^{-c \mathcal{T}^2 \sigma_*^4 / d} $ by a union bound over $ i \leq n $ alone. Using the averaged $ \int_I | x + s | \nu $ here would not help, since that quantity being below $ \mathcal{T} $ does not force $ |x + s| \leq \mathcal{T} $ at every $ t, $ and the event would involve an uncountable index set. The population remainder uses $ P [ \bar{V} \mathbf{1} \{ \bar{V} > \mathcal{T} \} ] \leq C \mathcal{T} \, \| \bar{V} \|_{\psi_2} e^{-c \mathcal{T}^2 / \| \bar{V} \|^2_{\psi_2} } $ for $ \mathcal{T} \geq 2 \| \bar{V} \|_{\psi_2}. $
    
    (iv) Pointwise in $ t, $ $ (h - w)^2 - (h^* - w)^2 = \varphi (h + h^* - 2w) $ with $ \varphi := h - h^*. $ Since $ h \in [- 2 \bar{B}, 3 \bar{B} ], \, h^* \in (0, \bar{B}] $ and $ w \in (0, \bar{B}], $ the extreme values are $ | \varphi | \leq 3 \bar{B} $ -- the constraint $ h^* > 0 $ cutting one side -- and $ | h + h^* - 2 w | \leq 4 \bar{B}, $ so $ \| D_h \|_\infty \leq 12 \bar{B}^2. $ By Cauchy-Schwarz in $ \nu, \, D_h^2 \leq \int_I \varphi^2 \nu \cdot \int_I (h + h^* - 2w)^2 \nu \leq 16 \bar{B}^2 \int_I \varphi^2 \nu, $ and taking $ P $ gives $ P D_h^2 \leq 16 \bar{B}^2 \mathcal{D}(h) = 16 \bar{B}^2 P D_h. $ 
    
    (v) $ Q_\psi \geq 0, $ and $ \psi^2 + | \nabla \psi |^2 \leq ( |\psi| + |\nabla \psi|)^2 \leq \Gamma^2 \| \psi \|^2_\bullet $ by \eqref{eq:psi_n_class}, with no factor 2. For a nonnegative bounded function, $ P Q^2 \leq \| Q \|_\infty P Q. $ 
\end{proof}

\begin{lemma}(Entropy of the index classes) \label{lem:entropy_index_classes}
	Let $ \mathcal{T} > 0 $ and let $ \mathcal{G}_1^\mathcal{T} := \{ G^\mathcal{T}_{h, \psi} : h \in \mathcal{H}_N, \, \psi \in \Psi^\bullet_N \}, \, \mathcal{D}_\bullet := \{ D_h : h \in \mathcal{H}_N \} $ and $ \mathcal{Q}_1 := \{ Q_\psi : \psi \in \Psi^\bullet_N \}, $ where $ \Psi^\bullet_N = \{ \psi \in \Psi_N : \| \psi \|_\bullet \leq 1 \} $ as in \eqref{eq:jets}. Put
	\begin{equation} \label{eq:frak_b_l}
		\mathfrak{b} := \Gamma (\Gamma + G \mathcal{T}) + 12 \bar{B}^2, \qquad \mathfrak{l} := (1 + \Gamma)^2 (1 + \mathcal{T}) (1 + G_t + G) + 12 \bar{B}^2.
	\end{equation}
	Then every member of $ \mathcal{G}_1^{\mathcal{T}} \cup \mathcal{D}_\bullet \cup \mathcal{Q}_1 $ is a Borel function on $ \mathcal{Z} $ bounded in modulus by $ \mathfrak{b}, $ one has $ 1 \lor \mathfrak{b} \leq \mathfrak{l}, $ and for every probability measure $ Q $ on $ \mathcal{Z} $ and every $ 0 < \varepsilon \leq \mathfrak{l} $
	\begin{equation} \label{eq:entropy_index_classes}
	\log N ( \varepsilon, \mathcal{G}_1^{\mathcal{T}} \cup \mathcal{D}_\bullet \cup \mathcal{Q}_1, \, L^2 (Q) ) \leq 6 \, \mathbb{V}_N \log \frac{e \mathfrak{l}}{\varepsilon},
	\end{equation}
	with $ \mathbb{V}_N $ as in \eqref{eq:eff_dim}, bounded in Lemma \ref{lem:complexity}\ref{cond:eff_dim}. For $ \varepsilon \geq \mathfrak{b} $ the covering number equals $ 1. $ Each of the three classes, and each of their subclasses, is separable in the supremum norm on $ \mathcal{Z}. $
\end{lemma}
\begin{proof}
	If $ \mathbb{V}_N = \infty $ there is nothing to prove, so let $ \mathbb{V}_N < \infty. $

	Envelopes. For $ h \in \mathcal{H}_N $ and $ \psi \in \Psi^\bullet_N $ we have, globally on $ I \times \mathbb{R}^d, $ $ h \in [-2 \bar{B}, 3 \bar{B}], \, | \partial_t h | \leq G_t $ and $ | \nabla h | \leq G $ by \eqref{eq:h_n_class}, and $ | \psi | + | \nabla \psi | \leq \Gamma \| \psi \|_\bullet \leq \Gamma. $ Write $ v^{\mathcal{T}} := (x + s) \mathbf{1} \{ | x + s | \leq \mathcal{T} \}, $ so that $ | v^{\mathcal{T}} | \leq \mathcal{T}. $ The integrands of $ G^{\mathcal{T}}_{h, \psi}, D_h, Q_\psi $ are bounded Borel functions of $ (t, z), $ so by Fubini the three functions are bounded Borel functions on $ \mathcal{Z}, $ and their $ L^2(Q) $-distances are at most their supremum distances. Pointwise, $ | \partial_t h \, \psi + (v^{\mathcal{T}} \cdot \nabla h) \psi - \nabla h \cdot \nabla \psi | \leq (\max (G_t, G) + G \mathcal{T}) ( | \psi | + | \nabla \psi | ) \leq \Gamma (\Gamma + G \mathcal{T}), $ using $ \Gamma \geq 2 G + 2 G_t, $ while $ | D_h | \leq 12 \bar{B}^2 $ and $ 0 \leq Q_\psi \leq \Gamma^2 $ by Lemma \ref{lem:envelopes_and_variances}. Each bound is at most $ \mathfrak{b}. $ Moreover $ (1 + \Gamma)^2 (1 + \mathcal{T}) (1 + G_t + G) \geq (1 + \Gamma)^2 (1 + G \mathcal{T}) \geq \Gamma^2 + \Gamma G \mathcal{T}, $ so $ \mathfrak{b} \leq \mathfrak{l}, $ and clearly $ \mathfrak{l} \geq 1. $ The single function $ 0 $ covers the union at every radius $ \varepsilon \geq \mathfrak{b}. $

	Lipschitz bounds in the jets. Let also $ h' \in \mathcal{H}_N $ and $ \psi' \in \Psi^\bullet_N, $ and put, pointwise on $ I \times \mathbb{R}^d, $ $ \Delta_h := | (h, \partial_t h, \nabla h) - (h', \partial_t h', \nabla h') | $ and $ \Delta_\psi := | (\psi, \nabla \psi) - (\psi', \nabla \psi') |. $ Adding and subtracting, and using the bounds above for $ h $ and for $ \psi', $
	\begin{align*}
		| \partial_t h \, \psi - \partial_t h' \, \psi' | &\leq | \psi' | \, | \partial_t h - \partial_t h' | + | \partial_t h | \, | \psi - \psi' | \leq \Gamma \Delta_h + G_t \Delta_\psi, \\
		| (v^{\mathcal{T}} \cdot \nabla h) \psi - (v^{\mathcal{T}} \cdot \nabla h') \psi' | &\leq \mathcal{T} | \psi' | \, | \nabla h - \nabla h' | + \mathcal{T} | \nabla h | \, | \psi - \psi' | \leq \Gamma \mathcal{T} \Delta_h + G \mathcal{T} \Delta_\psi, \\
		| \nabla h \cdot \nabla \psi - \nabla h' \cdot \nabla \psi' | &\leq | \nabla \psi' | \, | \nabla h - \nabla h' | + | \nabla h | \, | \nabla \psi - \nabla \psi' | \leq \Gamma \Delta_h + G \Delta_\psi.
	\end{align*}
	Also $ | (h - w)^2 - (h' - w)^2 | = | h - h' | \, | (h - w) + (h' - w) | \leq 6 \bar{B} \Delta_h, $ since $ w \in (0, \bar{B}] $ gives $ | h - w |, | h' - w | \leq 3 \bar{B}, $ and $ | \psi^2 + | \nabla \psi |^2 - \psi'^2 - | \nabla \psi' |^2 | \leq 2 \Gamma ( | \psi - \psi' | + | \nabla \psi - \nabla \psi' | ) \leq 2 \sqrt{2} \, \Gamma \Delta_\psi. $ Integrating against the probability measure $ \nu $ at $ x = X_{T-t}(z) $ and taking the supremum over $ z \in \mathcal{Z}, $
	\begin{align} 
		& \| G^{\mathcal{T}}_{h, \psi} - G^{\mathcal{T}}_{h', \psi'} \|_\infty \leq \mathrm{L}_{\mathcal{G}} \max (\| \Delta_h \|_\infty, \| \Delta_\psi \|_\infty), \nonumber \\
		& \| D_h - D_{h'} \|_\infty \leq 6 \bar{B} \| \Delta_h \|_\infty, \qquad \| Q_\psi - Q_{\psi'} \|_\infty \leq 3 \Gamma \| \Delta_\psi \|_\infty, \label{eq:index_lipschitz}
	\end{align}
	with $ \mathrm{L}_{\mathcal{G}} := \Gamma (\mathcal{T} + 2) + G_t + G (\mathcal{T} + 1). $ All three constants are at most $ \mathfrak{l}: $ indeed $ 6 \bar{B} \leq \Gamma \leq 3 \Gamma \leq (1 + \Gamma)^2, $ and $ \mathrm{L}_{\mathcal{G}} \leq (1 + \mathcal{T}) (2 \Gamma + G_t + G) \leq (1 + \mathcal{T}) (1 + 2 \Gamma) (1 + G_t + G) \leq (1 + \mathcal{T}) (1 + \Gamma)^2 (1 + G_t + G). $

	Separability. By \eqref{eq:clip_jet} and Lemma \ref{lem:complexity}\ref{cond:lipschitz_bound}, the map from admissible parameters to $ J(f^\chi) $ is continuous into the supremum norm, for $ \mathcal{H}_N $ and for the internal nets of $ \Psi_N, $ hence $ J \mathcal{H}_N $ is separable, and $ J \Psi^\bullet_N, $ contained in the continuous image of $ [-2, 2] \times \{ \text{parameters} \} $ under $ (c, \vartheta) \mapsto c \, \omega J(\phi_\vartheta^\chi), $ is separable. By \eqref{eq:index_lipschitz} the three classes are Lipschitz images of $ J \mathcal{H}_N \times J \Psi^\bullet_N, $ $ J \mathcal{H}_N $ and $ J \Psi^\bullet_N, $ hence separable, and subsets of separable metric spaces are separable.

	Nets. Fix $ 0 < \varepsilon \leq \mathfrak{l} $ and put $ \eta := \varepsilon / (2 \mathfrak{l}) \leq 1/2. $ Since $ \eta / 2 \leq 1, $ \eqref{eq:eff_dim} gives $ \log N (\eta / 2, J \mathcal{H}_N, \| \cdot \|_\infty) \leq \mathbb{V}_N \log (4 e \mathfrak{l} / \varepsilon), $ and the same for $ J \Psi^\bullet_N. $ By \ref{cond:internal_nets} there are $ \mathcal{S}_{\mathcal{H}} \subseteq \mathcal{H}_N $ and $ \mathcal{S}_\Psi \subseteq \Psi^\bullet_N, $ each of cardinality at most $ \exp (\mathbb{V}_N \log (4 e \mathfrak{l} / \varepsilon)), $ whose jets form $\eta$-nets of $ J \mathcal{H}_N $ and $ J \Psi^\bullet_N $ respectively. Given $ h \in \mathcal{H}_N $ and $ \psi \in \Psi^\bullet_N, $ choose $ h' \in \mathcal{S}_{\mathcal{H}} $ and $ \psi' \in \mathcal{S}_\Psi $ with $ \| \Delta_h \|_\infty, \| \Delta_\psi \|_\infty \leq \eta. $ By \eqref{eq:index_lipschitz}, $ G^{\mathcal{T}}_{h', \psi'}, D_{h'}, Q_{\psi'} $ lie within $ \mathfrak{l} \eta = \varepsilon / 2 $ of $ G^{\mathcal{T}}_{h, \psi}, D_h, Q_\psi $ in the supremum norm on $ \mathcal{Z}, $ hence in $ L^2(Q). $ The union class therefore has an $\varepsilon$-net with at most $ | \mathcal{S}_{\mathcal{H}} | \, | \mathcal{S}_\Psi | + | \mathcal{S}_{\mathcal{H}} | + | \mathcal{S}_\Psi | \leq 3 | \mathcal{S}_{\mathcal{H}} | \, | \mathcal{S}_\Psi | $ elements, and
	\begin{equation*}
		\log N ( \varepsilon, \mathcal{G}_1^{\mathcal{T}} \cup \mathcal{D}_\bullet \cup \mathcal{Q}_1, \, L^2 (Q) ) \leq \log 3 + 2 \, \mathbb{V}_N \log \frac{4 e \mathfrak{l}}{\varepsilon}.
	\end{equation*}
	Finally, $ \log (e \mathfrak{l} / \varepsilon) \geq 1 $ and $ \mathbb{V}_N \geq 1 $ give $ \log 3 \leq (\log 3) \, \mathbb{V}_N \log (e \mathfrak{l} / \varepsilon) $ and $ \log (4 e \mathfrak{l} / \varepsilon) \leq (1 + \log 4) \log (e \mathfrak{l} / \varepsilon), $ and $ \log 3 + 2 (1 + \log 4) < 6 $ yields \eqref{eq:entropy_index_classes}.
\end{proof}

We restate the external results in the form used below. Throughout, $ \varpi_1, \ldots, \varpi_n $ are i.i.d.\ Rademacher signs independent of the sample, and for a class $ \mathcal{G} $ of functions on $ \mathcal{Z} $ we write $ \mathfrak{R}_n (\mathcal{G}) := \sup_{f \in \mathcal{G}} | n^{-1} \sum_{i=1}^n \varpi_i f(Z_i) |. $ A class is called separable if it is a separable subset of the bounded functions on $ \mathcal{Z} $ with the supremum norm. Suprema of empirical and Rademacher processes over a separable class coincide with suprema over a countable dense subclass, so they are measurable and the results below, stated for countable classes, apply.

\begin{theorem}(Localization, Theorem 3.3 of \citep{Bartlett_2005} with $ K = 2 $) \label{thm:localization_bbm}
	Let $ \mathcal{F} $ be a separable class of functions on $ \mathcal{Z}, $ star-shaped around $ 0, $ with $ \| f \|_\infty \leq b $ and $ P f^2 \leq \beta_0 P f $ for all $ f \in \mathcal{F}. $ Let $ \mathfrak{s} $ be a sub-root function with fixed point $ r^* $ such that $ \mathfrak{s}(r) \geq \beta_0 \mathbb{E} \mathfrak{R}_n \{ f \in \mathcal{F}: P f^2 \leq r \} $ for all $ r \geq r^*. $ Then for every $ u > 0, $ with probability at least $ 1 - 2 e^{-u}, $ simultaneously for all $ f \in \mathcal{F} $
\begin{equation} \label{eq:localization_bbm}
	\frac{1}{2} Pf - c_1 \left( \frac{r^*}{\beta_0} + \frac{(b + \beta_0) u}{n} \right) \leq P_n f \leq 2 P f + c_1 \left( \frac{r^*}{\beta_0} + \frac{( b + \beta_0) u}{n} \right),
\end{equation}	 
with $ c_1 $ absolute. (The Rademacher average of \citep{Bartlett_2005} carries no absolute value and is dominated by $ \mathfrak{R}_n. $)
\end{theorem}

\begin{theorem}(Unbounded suprema, Theorem 4 of \citep{adamczak2008tailinequalitysupremaunbounded} with $ \eta = \delta = \alpha = 1 $) \label{thm:localization_adamczak}
	Let $ \mathcal{F} $ be a separable class of functions on $ \mathcal{Z} $ with a measurable envelope $ F \geq \sup_{f \in \mathcal{F}} |f| $ such that $ \| \max_{i \leq n} F(Z_i) \|_{\psi_1} < \infty, $ and let $ \varsigma^2_{\mathcal{F}} := \sup_{f \in \mathcal{F}} \Var (f). $ Then for every $ u \geq 1, $ with probability at least $ 1 - e^{-u} $
	\begin{equation} \label{eq:localization_adamczak}
	\underset{f \in \mathcal{F}}{\sup} | P_n f - P f| \leq 2 \mathbb{E} \, \underset{f \in \mathcal{F}}{\sup} | P_n f - Pf | + c_2 \varsigma_{\mathcal{F}} \sqrt{\frac{u}{n}} + c_2 \frac{u \, \| \max_{i \leq n} F (Z_i) \|_{\psi_1}}{n},
	\end{equation}
	with $ c_2 $ absolute.
\end{theorem}
\begin{proof}
    Theorem 4 of \citep{adamczak2008tailinequalitysupremaunbounded} applies to the centred class $ \{ f - Pf \}, $ whose summed variance is $ n \varsigma_{\mathcal{F}}^2 $ and whose envelope satisfies $ \| \max_i \sup_f | f(Z_i) - Pf | \|_{\psi_1} \leq (1 + 1/\log 2) \| \max_i F(Z_i) \|_{\psi_1} =: \mathfrak{M}, $ since $ \sup_f |Pf| \leq PF \leq \| F(Z_1) \|_{\psi_1} \leq \| \max_i F(Z_i) \|_{\psi_1} $ and a constant $ a \geq 0 $ has $ \psi_1 $-norm $ a / \log 2. $ With $ Y := n \sup_f | P_n f - P f | $ it gives $ \mathbb{P} ( Y \geq 2 \mathbb{E} Y + t ) \leq \exp ( - t^2 / (4 n \varsigma^2_{\mathcal{F}}) ) + 3 \exp ( - t / (C \mathfrak{M}) ) $ with $ C $ absolute. Choosing $ t := 2 \sqrt{n \varsigma^2_{\mathcal{F}} u'} + C \mathfrak{M} u' $ with $ u' := u + \log 4 \leq (1 + \log 4) u $ makes the right-hand side at most $ 4 e^{-u'} = e^{-u}, $ dividing by $ n $ gives the claim.
\end{proof}

\begin{lemma}(Closure) \label{lem:entropy_closure}
    Let $ \mathcal{F} $ be a separable class with $ \sup_{f \in \mathcal{F}} \| f \|_\infty \leq b, $ and suppose that for some $ V \geq 1 $ and $ \mathfrak{l} \geq b $
    \begin{equation} \label{eq:uniform_entropy_generic}
        \log N (\varepsilon, \mathcal{F}, L^2 (Q)) \leq V \log (e \mathfrak{l} / \varepsilon) \quad \text{for every probability measure } Q \text{ on } \mathcal{Z} \text{ and every } 0 < \varepsilon \leq \mathfrak{l}.
    \end{equation}
    Then every subclass of $ \mathcal{F} $ is separable and satisfies \eqref{eq:uniform_entropy_generic}, and the star hull $ \mathrm{star} \, \mathcal{F} := \{ \lambda f : \lambda \in [0,1], f \in \mathcal{F} \} $ is separable, star-shaped around $ 0, $ bounded by $ b, $ and satisfies \eqref{eq:uniform_entropy_generic} with $ 4 V $ in place of $ V. $
\end{lemma}
\begin{proof}
    Subsets of separable metric spaces are separable, and covering numbers, with centres anywhere, do not increase under inclusion. The star hull is the image of $ [0, 1] \times \mathcal{F} $ under the continuous map $ (\lambda, f) \mapsto \lambda f, $ hence separable. Fix $ Q $ and $ 0 < \varepsilon \leq \mathfrak{l}, $ an $ (\varepsilon/2) $-net $ \{ g_k \} $ of $ \mathcal{F} $ in $ L^2(Q), $ and the points $ \lambda_m := \min \{ 1, m \varepsilon / b \}, \, 0 \leq m \leq \lceil b / \varepsilon \rceil, $ so that every $ \lambda \in [0,1] $ is within $ \varepsilon / (2b) $ of some $ \lambda_m. $ Then $ \| \lambda f - \lambda_m g_k \|_{L^2(Q)} \leq | \lambda - \lambda_m | \, b + \lambda_m \| f - g_k \|_{L^2(Q)} \leq \varepsilon, $ so $ N(\varepsilon, \mathrm{star} \, \mathcal{F}, L^2(Q)) \leq (2 + b / \varepsilon) N(\varepsilon / 2, \mathcal{F}, L^2(Q)). $ Since $ 2 + b/\varepsilon \leq 3 \mathfrak{l} / \varepsilon $ and $ \log (e \mathfrak{l} / \varepsilon) \geq 1, $ we get $ \log (3 \mathfrak{l} / \varepsilon) \leq 2 \log (e \mathfrak{l}/\varepsilon) $ and $ V \log (2 e \mathfrak{l} / \varepsilon) \leq V (1 + \log 2) \log (e \mathfrak{l} / \varepsilon), $ and $ 2 + 1.7 V \leq 4 V. $
\end{proof}

\begin{lemma}(Critical radius from entropy) \label{lem:critical_radius}
	Let $ \mathcal{F} $ be a separable class with $ 0 \in \mathcal{F}, $ $ \sup_{f \in \mathcal{F}} \| f \|_\infty \leq b $ for some $ b > 0, $ satisfying \eqref{eq:uniform_entropy_generic} for some $ V \geq 1 $ and $ \mathfrak{l} \geq b, $ and put $ \mathrm{L} := \log (2 e \mathfrak{l} n / b). $ Then for every $ r > 0 $
	\begin{equation} \label{eq:critical_radius}
	\mathbb{E} \mathfrak{R}_n \{ f \in \mathcal{F}: P f^2 \leq r \} \leq C_1 \left( \sqrt{\frac{r V \mathrm{L}}{n}} + \frac{b V \mathrm{L}}{n} \right),
	\end{equation}
	with $ C_1 $ absolute. Consequently, if $ \beta_0 \geq b, $ then $ \mathfrak{s}(r) := \beta_0 C_1 ( \sqrt{r V \mathrm{L} / n} + b V \mathrm{L} / n ) $ is a sub-root function dominating $ \beta_0 \mathbb{E} \mathfrak{R}_n \{ f \in \mathcal{F} : Pf^2 \leq r \} $ for all $ r > 0, $ and its fixed point satisfies $ r^* \leq 4 C_1^2 \beta_0^2 V \mathrm{L} / n. $
\end{lemma}
\begin{proof}
    Fix $ r > 0, $ let $ \mathcal{F}_r := \{ f \in \mathcal{F} : Pf^2 \leq r \}, $ which contains $ 0, $ and write $ \mathfrak{r} := \mathbb{E} \mathfrak{R}_n (\mathcal{F}_r) \leq b $ and $ \hat{\sigma}^2 := \sup_{f \in \mathcal{F}_r} P_n f^2 \leq b^2. $ Conditionally on the sample, $ X_f := n^{-1/2} \sum_i \varpi_i f(Z_i) $ is a centred sub-Gaussian process with respect to $ \mathrm{d}(f, g) := \| f - g \|_{L^2(P_n)}, $ and $ | X_f - X_g | \leq \sqrt{n} \, \mathrm{d}(f,g). $ The diameter of $ \mathcal{F}_r $ is at most $ 2 \hat{\sigma}. $ Since $ X_0 = 0, $ Theorem 5.22 of \citep{wainwright2019} with $ \delta := 4 b / n $ (or, when $ \delta $ exceeds the diameter $ D_r $ of $ \mathcal{F}_r, $ the trivial bound $ \sup (X_f - X_g) \leq \sqrt{n} D_r < \sqrt{n} \delta $) gives
    \begin{equation*}
        \mathbb{E}_\varpi \mathfrak{R}_n (\mathcal{F}_r) \leq \frac{1}{\sqrt{n}} \mathbb{E}_\varpi \sup_{f, g \in \mathcal{F}_r} (X_f - X_g) \leq \frac{8 b}{n} + \frac{32}{\sqrt{n}} \int_{b/n}^{2 \hat{\sigma}} \sqrt{\log N_{\mathrm{int}} (u, \mathcal{F}_r, L^2(P_n))} \, du,
    \end{equation*}
    with the integral read as $ 0 $ if $ 2 \hat{\sigma} \leq b/n, $ where $ N_{\mathrm{int}} $ counts nets with centres in $ \mathcal{F}_r. $ For $ u \leq 2 \hat{\sigma} \leq 2 \mathfrak{l}, $ choosing a point of $ \mathcal{F}_r $ in each ball of an external $ (u/2) $-net shows $ \log N_{\mathrm{int}} (u, \mathcal{F}_r, L^2(P_n)) \leq \log N(u/2, \mathcal{F}, L^2(P_n)) \leq V \log (2 e \mathfrak{l} / u) \leq V \mathrm{L} $ for $ u \geq b/n. $ Hence $ \mathbb{E}_\varpi \mathfrak{R}_n (\mathcal{F}_r) \leq 8b/n + 64 \hat{\sigma} \sqrt{V \mathrm{L} / n}, $ and taking expectations, with Jensen, $ \mathfrak{r} \leq 8 b / n + 64 \sqrt{V \mathrm{L} / n} \, ( \mathbb{E} \hat{\sigma}^2 )^{1/2}. $ Next, $ \hat{\sigma}^2 \leq r + \sup_{f \in \mathcal{F}_r} | P_n f^2 - P f^2 |. $ By symmetrization (Lemma 2.3.1 of \citep{vdvaart1996}) and the contraction principle (Theorem 4.12 of \citep{ledoux1991}) applied to $ \phi(v) := \min \{ v^2, b^2 \} / (2b), $ which is a contraction with $ \phi(0) = 0 $ and agrees with $ v^2 / (2b) $ on $ [-b, b], $
    \begin{equation*}
        \mathbb{E} \sup_{f \in \mathcal{F}_r} | P_n f^2 - P f^2 | \leq 2 \mathbb{E} \mathfrak{R}_n \{ f^2 : f \in \mathcal{F}_r \} = 4 b \, \mathbb{E} \mathfrak{R}_n \{ \phi \circ f : f \in \mathcal{F}_r \} \leq 8 b \, \mathfrak{r}.
    \end{equation*}
    Therefore $ \mathfrak{r} \leq 8 b/n + 64 \sqrt{V \mathrm{L}/n} \, ( \sqrt{r} + \sqrt{8 b \mathfrak{r}} ). $ By $ x y \leq x^2 / 2 + y^2 / 2 $ with $ x := \sqrt{\mathfrak{r}}, $ $ 64 \sqrt{8 b V \mathrm{L} / n} \sqrt{\mathfrak{r}} \leq \mathfrak{r} / 2 + 16384 \, b V \mathrm{L} / n, $ hence $ \mathfrak{r} \leq 16 b / n + 128 \sqrt{r V \mathrm{L} / n} + 32768 \, b V \mathrm{L} / n, $ which is \eqref{eq:critical_radius} with $ C_1 := 32784, $ because $ V \mathrm{L} \geq 1. $

    The function $ \mathfrak{s} $ is nonnegative, nondecreasing, and $ \mathfrak{s}(r) / \sqrt{r} $ is nonincreasing, so it is sub-root and has a unique positive fixed point. At it, $ r^* = X + Y $ with $ X := \beta_0 C_1 \sqrt{r^* V \mathrm{L} / n} $ and $ Y := \beta_0 C_1 b V \mathrm{L} / n, $ so $ r^* \leq 2 \max(X, Y). $ If $ X \geq Y, $ dividing $ r^* \leq 2 X $ by $ \sqrt{r^*} $ and squaring gives $ r^* \leq 4 \beta_0^2 C_1^2 V \mathrm{L} / n. $ If $ Y > X, $ then $ r^* \leq 2 Y \leq 2 C_1 \beta_0^2 V \mathrm{L} / n $ by $ b \leq \beta_0. $ Both are at most $ 4 C_1^2 \beta_0^2 V \mathrm{L} / n. $
\end{proof}

\begin{theorem}(Deviation bound) \label{thm:dev_bound}
	Grant Assumptions \ref{ass:bounded}--\ref{ass:compact}, let $ \mathcal{H}_N, \Psi_N $ be as in Definition \ref{def:trial_test_classes} with $ R $ as in \eqref{eq:R_K_choice}, and fix $ \zeta \in (0, 1). $ Put
	\begin{equation} \label{eq:delta_n}
		\delta_n^2 := C_0 \sigma^{-c} \frac{(\mathbb{V}_N + \log (1 / \zeta)) \log^2 (n / \zeta)}{n},
	\end{equation}
	where $ C_0 \sigma^{-c} $ is a sufficiently large constant of the form fixed in Appendix \ref{app:notation}. Then with probability at least $ 1 - \zeta, $ simultaneously for all $ h \in \mathcal{H}_N $ and all $ \psi \in \Psi_N: $
	\begin{enumerate}[label=(\roman*), leftmargin=2.2em]
	\item $ | \int_I \left[ \hat{a}_t - a_t \right] (h, \psi) \nu (dt) | \leq \delta_n \| \psi \|_{\mathcal{V}} + \delta_n^2 \| \psi \|_\bullet \leq \sqrt{2} \delta_n \| \psi \|_{\mathcal{V}, \delta_n}, $
	\item $ \frac{1}{2} \mathcal{D}(h) - \delta_n^2 \leq \hat{\mathcal{D}}(h) - \hat{\mathcal{D}}(h^*) \leq 2 \mathcal{D}(h) + \delta_n^2, $
	\item $ \frac{1}{2} \| \psi \|^2_{\mathcal{V}} - \delta_n^2 \| \psi \|^2_\bullet \leq \| \psi \|^2_{\hat{\mathcal{V}}} \leq 2 \| \psi \|^2_{\mathcal{V}} + \delta_n^2 \| \psi \|^2_\bullet, $ and consequently
	\begin{equation} \label{eq:psi_v_emp_bernst}
	\frac{1}{4} \| \psi \|^2_{\mathcal{V}, \delta_n} \leq \| \psi \|^2_{\hat{\mathcal{V}}, \delta_n} \leq 2 \| \psi \|^2_{\mathcal{V}, \delta_n},
	\end{equation}
	a purely multiplicative comparison.
	\end{enumerate}
    For $ \zeta \geq n^{-C'} $ one has $ \log (n/\zeta) \leq (C' + 1) \log n, $ so that $ \delta_n^2 \leq C'' \sigma^{-c} (\mathbb{V}_N + \log (1/\zeta)) \log^2 n / n. $
\end{theorem}

\begin{proof}
	Constants. Write $ \eta_n := \delta_n / 3 $ and fix an absolute $ C_T \geq 1 $ with $ c C_T^2 \geq 2 $ and $ C_T \sqrt{\log 8} \geq C, $ where $ c, C $ are the absolute constants of Lemma \ref{lem:envelopes_and_variances}(iii). Put $ \mathcal{T} := C_T \sqrt{d} \sigma_*^{-2} \sqrt{\log (8n/\zeta)} $ and $ \Sigma^2 := 3 (G_t^2 + G^2 \mathcal{T}^2 + G^2), $ and let $ \mathfrak{b}, \mathfrak{l} $ be as in \eqref{eq:frak_b_l} at this $ \mathcal{T}. $ From \eqref{eq:three_constants}, \eqref{eq:R_K_choice}, $ n \geq 2 $ and $ \log (8n/\zeta) \leq 4 \log (n / \zeta), $ with $ \sigma^{-c} $ as in Appendix \ref{app:notation}:
	\begin{equation} \label{eq:dev_constants}
	    \begin{aligned}
	    & \Gamma^2 + G_t^2 \leq \sigma^{-c} \log n, \quad \Sigma^2 + \mathfrak{b} \leq \sigma^{-c} \log (n/\zeta), \quad \mathfrak{l} \leq \sigma^{-c} \log^2 (n/\zeta), \\
	    & \log (2 e \mathfrak{l} n / b) \leq \sigma^{-c} \log (n / \zeta) \quad \text{for every } b \geq 12 \bar{B}^2.
	    \end{aligned}
	\end{equation}
	Choosing $ C_0 \geq 9 / \log^2 2 $ gives $ \eta_n^2 \geq 1/n, $ hence $ J := \max \{ 0, \lceil \log_2 (\Gamma / \eta_n) \rceil \} $ obeys $ J \leq \log_2 (\Gamma \sqrt{n}) + 1 $ and $ \log (J + 1) \leq \sigma^{-c} \log n. $ Every inequality ``$ \leq \delta_n^2 $'' or ``$ \leq \eta_n^2 $'' below holds once $ C_0 $ is large enough, because its left-hand side is at most $ \sigma^{-c} (\mathbb{V}_N + \log (1/\zeta)) \log^2 (n/\zeta) / n $ by \eqref{eq:dev_constants} and $ \mathbb{V}_N \geq 1. $ By Lemma \ref{lem:entropy_index_classes}, the three index classes are separable, bounded by $ \mathfrak{b} \leq \mathfrak{l}, $ and satisfy \eqref{eq:uniform_entropy_generic} with $ V = 6 \mathbb{V}_N. $ Put $ u := \log (8 / \zeta). $

	Step 1 (value class). $ \mathcal{F}_1 := \mathrm{star} \, \mathcal{D}_\bullet $ is separable, star-shaped, bounded by $ b_1 := 12 \bar{B}^2, $ and satisfies \eqref{eq:uniform_entropy_generic} with $ V = 24 \mathbb{V}_N $ by Lemma \ref{lem:entropy_closure}. Since $ P D_h = \mathcal{D}(h) \geq 0 $ and $ P D_h^2 \leq 16 \bar{B}^2 P D_h $ (Lemma \ref{lem:envelopes_and_variances}(iv)), $ P (\lambda D_h)^2 \leq \lambda \beta_1 P(\lambda D_h) \leq \beta_1 P (\lambda D_h) $ for $ \lambda \in [0,1] $ with $ \beta_1 := 16 \bar{B}^2 \geq b_1. $ Lemma \ref{lem:critical_radius} gives $ r_1^* \leq 96 C_1^2 \beta_1^2 \mathbb{V}_N \mathrm{L}_1 / n $ with $ \mathrm{L}_1 := \log (2 e \mathfrak{l} n / b_1), $ and Theorem \ref{thm:localization_bbm} gives, on an event $ E_1 $ of probability at least $ 1 - 2 e^{-u} = 1 - \zeta/4, $ for all $ h \in \mathcal{H}_N, $
	\begin{equation*}
	\frac{1}{2} \mathcal{D}(h) - \Delta_1 \leq \hat{\mathcal{D}}(h) - \hat{\mathcal{D}}(h^*) \leq 2 \mathcal{D}(h) + \Delta_1, \qquad \Delta_1 := c_1 \Big( 96 C_1^2 \beta_1 \frac{\mathbb{V}_N \mathrm{L}_1}{n} + \frac{28 \bar{B}^2 u}{n} \Big) \leq \delta_n^2,
	\end{equation*}
	using $ P D_h = \mathcal{D}(h) $ and $ P_n D_h = \hat{\mathcal{D}}(h) - \hat{\mathcal{D}}(h^*). $ This is (ii).

	Step 2 (norm class, generator scale). $ \mathcal{F}_2 := \mathrm{star} \, \mathcal{Q}_1 $ is separable, star-shaped, bounded by $ b_2 := \Gamma^2, $ satisfies \eqref{eq:uniform_entropy_generic} with $ V = 24 \mathbb{V}_N, $ and $ P f^2 \leq \Gamma^2 P f $ on it by Lemma \ref{lem:envelopes_and_variances}(v). With $ \beta_2 := \Gamma^2 $ and $ \mathrm{L}_2 := \log (2 e \mathfrak{l} n / \Gamma^2), $ Lemma \ref{lem:critical_radius} and Theorem \ref{thm:localization_bbm} give an event $ E_2 $ of probability at least $ 1 - \zeta / 4 $ on which, for all $ \psi \in \Psi^\bullet_N, $
	\begin{equation*}
		\frac{1}{2} \| \psi \|^2_{\mathcal{V}} - \eta_n^2 \leq \| \psi \|^2_{\hat{\mathcal{V}}} \leq 2 \| \psi \|_{\mathcal{V}}^2 + \eta_n^2,
	\end{equation*}
	because $ P Q_\psi = \| \psi \|^2_{\mathcal{V}}, $ $ P_n Q_\psi = \| \psi \|^2_{\hat{\mathcal{V}}} $ and $ c_1 ( 96 C_1^2 \Gamma^2 \mathbb{V}_N \mathrm{L}_2 + 2 \Gamma^2 u ) / n \leq \eta_n^2 $ by \eqref{eq:dev_constants}.

	Step 3 (bilinear class, generator scale). (a) Truncation. By Lemma \ref{lem:envelopes_and_variances}(iii) and $ c C_T^2 \geq 2, $ the event $ E_{\mathcal{T}} := \{ \max_{i \leq n} \bar{V}(Z_i) \leq \mathcal{T} \} $ has probability at least $ 1 - 2n (\zeta / 8n)^{c C_T^2} \geq 1 - \zeta/4, $ and on it $ P_n G_{h, \psi} = P_n G^{\mathcal{T}}_{h,\psi} $ for all $ (h, \psi). $ Since $ \mathcal{T} \geq C \sqrt{d} \sigma_*^{-2}, $ the same lemma gives, for $ \psi \in \Psi^\bullet_N, $ $ | P G_{h,\psi} - P G^{\mathcal{T}}_{h, \psi} | \leq C \mathcal{T} \Gamma G (\zeta / 8n)^{2} \leq \sigma^{-c} \log (8n/\zeta) (\zeta/8n)^2 \leq \sigma^{-c} / n \leq \eta_n^2, $ using $ y^{-2} \log y \leq y^{-1} $ for $ y \geq 1. $

	(b) Shells. For $ \psi \in \Psi^\bullet_N $ we have $ \| \psi \|_{\mathcal{V}} \leq \Gamma_{\mathcal{V}} \leq \Gamma. $ For $ 0 \leq j \leq J $ put $ r_j := 2^{-j} \Gamma $ and $ \mathcal{G}_j := \{ G^{\mathcal{T}}_{h, \psi} : h \in \mathcal{H}_N, \, \psi \in \Psi^\bullet_N, \, \| \psi \|_{\mathcal{V}} \leq r_j \}. $ Each $ \mathcal{G}_j $ is a separable subclass of $ \mathcal{G}^{\mathcal{T}}_1 $ containing $ 0, $ bounded by $ \mathfrak{b}, $ satisfying \eqref{eq:uniform_entropy_generic} with $ V = 6 \mathbb{V}_N, $ and with $ \Var(f) \leq P f^2 \leq \varsigma_j^2 := \Sigma^2 r_j^2 $ by \eqref{eq:P_G_T_2}. Symmetrization and Lemma \ref{lem:critical_radius} with $ \mathrm{L}_3 := \log (2 e \mathfrak{l} n / \mathfrak{b}) $ give
	\begin{equation*}
	    \mathbb{E} \sup_{f \in \mathcal{G}_j} | P_n f - P f | \leq 2 \mathbb{E} \mathfrak{R}_n \{ f \in \mathcal{G}^{\mathcal{T}}_1 : P f^2 \leq \varsigma_j^2 \} \leq 2 C_1 \Big( \varsigma_j \sqrt{\frac{6 \mathbb{V}_N \mathrm{L}_3}{n}} + \frac{6 \mathfrak{b} \mathbb{V}_N \mathrm{L}_3}{n} \Big).
	\end{equation*}
	Theorem \ref{thm:localization_adamczak} with the constant envelope $ \mathfrak{b} $ (so $ \| \max_i \mathfrak{b} \|_{\psi_1} = \mathfrak{b} / \log 2 $) and $ u_3 := \log (4 (J+1) / \zeta) \geq 1 $ gives, with probability at least $ 1 - \zeta / (4 (J+1)), $
	\begin{equation*}
	    \sup_{f \in \mathcal{G}_j} | P_n f - P f | \leq \varsigma_j \frac{4 C_1 \sqrt{6 \mathbb{V}_N \mathrm{L}_3} + c_2 \sqrt{u_3}}{\sqrt{n}} + \frac{24 C_1 \mathfrak{b} \mathbb{V}_N \mathrm{L}_3 + 2 c_2 \mathfrak{b} u_3}{n} \leq r_j \eta_n + \eta_n^2,
	\end{equation*}
	where the last step uses $ \varsigma_j = \Sigma r_j, $ $ u_3 \leq \sigma^{-c} \log (n / \zeta) $ and \eqref{eq:dev_constants}, so that $ \Sigma^2 ( \mathbb{V}_N \mathrm{L}_3 + u_3 ) \leq \sigma^{-c} (\mathbb{V}_N + \log (1/\zeta)) \log^2 (n/\zeta) $ and $ \mathfrak{b} ( \mathbb{V}_N \mathrm{L}_3 + u_3 ) \leq \sigma^{-c} (\mathbb{V}_N + \log(1/\zeta)) \log^2 (n / \zeta). $ Let $ E_3 $ be the intersection of these $ J + 1 $ events, $ \mathbb{P}(E_3) \geq 1 - \zeta / 4. $

	(c) Combination. On $ E_{\mathcal{T}} \cap E_3, $ let $ h \in \mathcal{H}_N, $ $ \psi \in \Psi^\bullet_N $ and $ j_* := \max \{ j \leq J : \| \psi \|_{\mathcal{V}} \leq r_j \}. $ If $ j_* < J $ then $ \| \psi \|_{\mathcal{V}} > r_{j_* + 1} = r_{j_*}/2, $ if $ j_* = J $ then $ r_J \leq \eta_n. $ Hence $ r_{j_*} \leq 2 \| \psi \|_{\mathcal{V}} + \eta_n $ and, by (a) and (b),
	\begin{equation} \label{eq:step_3_bound}
	| P_n G_{h, \psi} - P G_{h, \psi} | \leq r_{j_*} \eta_n + \eta_n^2 + \eta_n^2 \leq 2 \eta_n \| \psi \|_{\mathcal{V}} + 3 \eta_n^2, \quad \psi \in \Psi^\bullet_N.
	\end{equation}

	Step 4 (from generator scale to the cone). Work on $ E_1 \cap E_2 \cap E_{\mathcal{T}} \cap E_3, $ which has probability at least $ 1 - \zeta. $ Fix $ h $ and $ \psi \in \Psi_N $ and let $ \psi = c \, \omega \phi^\chi $ be any representation as in \eqref{eq:psi_n_class}, then $ \upsilon := \omega \phi^\chi \in \Psi^\bullet_N. $ Since $ \psi \mapsto G_{h, \psi} $ is linear, \eqref{eq:step_3_bound} applied to $ \upsilon $ gives
	\begin{equation*}
		| P_n G_{h, \psi} - P G_{h, \psi} | = | c | \, | P_n G_{h, \upsilon} - P G_{h, \upsilon} | \leq 2 \eta_n \| \psi \|_{\mathcal{V}} + 3 \eta_n^2 | c |.
	\end{equation*}
	The left-hand side does not depend on the representation, so taking the infimum over representations replaces $ |c| $ by $ \| \psi \|_\bullet. $ As $ 2 \eta_n \leq \delta_n $ and $ 3 \eta_n^2 \leq \delta_n^2, $ and by \eqref{eq:three_empirical_things}, this is the first inequality of (i), the second is $ a + b \leq \sqrt{2} (a^2 + b^2)^{1/2}. $ Likewise $ Q_\psi = c^2 Q_\upsilon, $ and Step 2 applied to $ \upsilon $ gives $ \frac{1}{2} \| \psi \|^2_{\mathcal{V}} - \eta_n^2 c^2 \leq \| \psi \|^2_{\hat{\mathcal{V}}} \leq 2 \| \psi \|^2_{\mathcal{V}} + \eta_n^2 c^2 $ for every representation, whence (iii) after taking the infimum and using $ \eta_n \leq \delta_n. $ Adding $ \delta_n^2 \| \psi \|^2_\bullet $ to (iii) yields $ \| \psi \|^2_{\hat{\mathcal{V}}, \delta_n} \geq \frac{1}{2} \| \psi \|^2_{\mathcal{V}} $ and $ \| \psi \|^2_{\hat{\mathcal{V}}, \delta_n} \leq 2 \| \psi \|^2_{\mathcal{V}, \delta_n}, $ since also $ \| \psi \|^2_{\hat{\mathcal{V}}, \delta_n} \geq \delta_n^2 \| \psi \|^2_\bullet, $ averaging the two lower bounds gives \eqref{eq:psi_v_emp_bernst}. Statement (ii) is Step 1.
\end{proof}

\begin{theorem}(Guidance rate) \label{thm:guidance_rate}
	Grant Assumptions \ref{ass:bounded}--\ref{ass:compact}, fix $ \zeta \in (0,1) $ and $ \kappa \geq 1, $ let $ R, K, Q_n $ be as in \eqref{eq:R_K_choice}--\eqref{eq:Q_n_def}, and choose
	\begin{equation} \label{eq:N_mu_alpha_b_choice}
		N := \Big\lceil \left( n \mathrm{A}_\beta^2 \right)^{\frac{1}{2 (\beta - 1) + d}} \Big\rceil, \quad \mu, \alpha \in [\kappa^{-1}, \kappa], \quad \mathrm{b} := (\varepsilon_n / \Xi_\mathrm{s})^{\frac{1}{\mathrm{s} + 4}} \land \bar{B} \ \ (\mathrm{b} := B \text{ if } \mathrm{s} = \infty),
	\end{equation}
    \begin{equation} \label{eq:rate_vareps_n}
        \varepsilon_n := \mathrm{A}_\beta^{\frac{2 d}{2 (\beta - 1) + d}} n^{- \frac{2 (\beta - 1)}{2 (\beta - 1) + d}}.
    \end{equation}
    Let $ \mathcal{H}_N, \Psi_N $ be as in Definition \ref{def:trial_test_classes} with budgets equal to the right-hand sides of \eqref{eq:budgets}, and let $ \hat{h}, \hat{g} $ be given by \eqref{eq:h_empirical} with $ \delta_n $ as in \eqref{eq:delta_n}. There is $ n_0, $ depending only on $ (d, \beta, \bar{B}, \ell, \sigma_*, C_t, \kappa_\Theta), $ such that for $ n \geq n_0, $ on the event of Theorem \ref{thm:dev_bound},
	\begin{equation} \label{eq:g_rate}
		\| \hat{g} - g^* \|^2_{L^2 (\nu \otimes \rho; I^\circ)} \leq \sigma^{-c} \log^8 (e n / \zeta) \, \Xi_\mathrm{s}^{\frac{4}{\mathrm{s} + 4}} \varepsilon_n^{\frac{\mathrm{s}}{\mathrm{s} + 4}},
	\end{equation}
	where $ \sigma^{-c} $ may also depend on $ \kappa, $ and where at $ \mathrm{s} = \infty $ the right-hand side is read as $ \sigma^{-c} \log^8 (e n / \zeta) B^{-4} \varepsilon_n. $ At fixed $ \beta $ the exponent of $ n $ in \eqref{eq:rate_vareps_n} is $ 2 (\beta - 1) / (2 (\beta - 1) + d), $ so \eqref{eq:g_rate} is $ n^{-2 / (d+2)} $ up to logarithmic factors at $ \beta = 2 $ and $ \mathrm{s} = \infty. $ The bound contains no regularization-gap term, for any fixed $ \kappa. $
\end{theorem}
\begin{proof}
	Write $ D(h) := \mathcal{D}(h), $ $ \Pi_\Psi (h) := \mathcal{N}_{\Psi_N, \delta_n} (h)^2 $ and $ \mathfrak{D} := 2 (\beta - 1) + d, $ and let $ \hat{F}(h) $ denote the supremum in \eqref{eq:h_empirical}. 

	Step 0 (regime and approximant). Put $ x_n := (n \mathrm{A}_\beta^2)^{1/\mathfrak{D}} \geq 1, $ so $ x_n \leq N \leq 2 x_n. $ Then $ \mathrm{A}_\beta^2 N^{-2 (\beta - 1)} \leq \mathrm{A}_\beta^2 x_n^{-2(\beta - 1)} = \varepsilon_n $ and $ N^d / n \leq 2^d x_n^d / n = 2^d \varepsilon_n. $ Since $ \mathrm{A}_\beta $ and $ Q_n $ grow only polylogarithmically in $ n $ while $ x_n $ grows polynomially, enlarging $ n_0 $ ensures \eqref{eq:approx_regime} and the conditions of Lemma \ref{lem:space_time_approx} for $ n \geq n_0. $ That lemma gives $ h_N \in \mathcal{H}_N $ satisfying \eqref{eq:approx_error_eN} with $ e_N := 8 Q_n \mathrm{A}_\beta N^{-(\beta - 1)}, $ so $ e_N^2 \leq 64 Q_n^2 \varepsilon_n \leq \sigma^{-c} \log^8 (n) \, \varepsilon_n $ by \eqref{eq:Q_n_def}. By Lemma \ref{lem:weak_res_approx} and \eqref{eq:proj_res}, $ \Pi_\Psi (h_N) \leq \mathcal{P}^w (h_N) \leq 4 (1 + M_4)^2 e_N^2 $ and $ D (h_N) \leq \sup_t \| h_N - h^* \|_{\rho_t}^2 \leq e_N^2. $ With the budgets fixed, Lemma \ref{lem:complexity}\ref{cond:eff_dim} (whose side conditions $ W_N \geq d + 1 $ and $ P_0 \leq P_{\mathcal{H}} $ hold once $ C^\sharp \geq d + 1, $ because $ L, W_N \geq 2 $ give $ (L - 1) W_N^2 \geq C^\sharp K (N + 1)^d L \geq P_0 $) gives $ \mathbb{V}_N \leq \sigma^{-c} N^d (\log n)^{7/2}, $ because $ \bar{P}_0 \leq \sigma^{-c} K N^d $ (as $ L \leq (N+1)^d $ for $ n \geq n_0 $), $ K \leq \sigma^{-c} (\log n)^{3/2}, $ and the bracket in \eqref{eq:eff_dim_bound} and $ \log (1 + 4 \Gamma) $ are at most $ \sigma^{-c} \log^2 n. $ Hence, using $ \mathbb{V}_N \geq 1 $ and $ 1 + \log (1/\zeta) \leq \log (e n / \zeta), $
	\begin{gather*}
	    \delta_n^2 \leq \sigma^{-c} \frac{N^d}{n} (\log n)^{7/2} \big( 1 + \log (1/\zeta) \big) \log^2 (n/\zeta) \leq \sigma^{-c} \log^{8} (e n / \zeta) \, \varepsilon_n, \\
	    \bar{\varepsilon}_n := e_N^2 + \delta_n^2 \leq \sigma^{-c} \log^8 (en/\zeta) \, \varepsilon_n.
	\end{gather*}

	Throughout the remaining steps we work on the event of Theorem \ref{thm:dev_bound}, on which the adversary norms $ \| \cdot \|_{\hat{\mathcal{V}}, \delta_n} $ and $ \| \cdot \|_{\mathcal{V}, \delta_n} $ are equivalent within the factors $ \frac{1}{4} $ and $ 2 $ by \eqref{eq:psi_v_emp_bernst}.
	
	Step 1 (the inner maximum). For every $ \psi \in \Psi_N, $ the definition \eqref{eq:proj_res} and Theorem \ref{thm:dev_bound}(i) give
	\begin{equation}
		\int_I \hat{a}_t (h, \psi) \nu (dt) \leq A(h, \psi) + \sqrt{2} \delta_n \| \psi \|_{\mathcal{V}, \delta_n} \leq \left( \sqrt{\Pi_\Psi (h)} + \sqrt{2} \delta_n \right) \| \psi \|_{\mathcal{V}, \delta_n},
	\end{equation}
	and $ \| \psi \|_{\mathcal{V}, \delta_n} \leq 2 \| \psi \|_{\hat{\mathcal{V}}, \delta_n} $ by \eqref{eq:psi_v_emp_bernst}. Writing $ r := \| \psi \|_{\hat{\mathcal{V}}, \delta_n}, $ the bracket in \eqref{eq:h_empirical} is at most $ 2 ( \sqrt{\Pi_\Psi} + \sqrt{2} \delta_n ) r - \frac{\alpha}{2} r^2, $ whose maximum over $ r \geq 0 $ is $ 2 ( \sqrt{\Pi_\Psi} + \sqrt{2} \delta_n )^2 / \alpha, $ so
	\begin{equation} \label{eq:F_hat_upper_bound}
		\hat{F}(h) \leq \frac{4 \Pi_\Psi (h) + 8 \delta_n^2}{\alpha}.
	\end{equation}
	In the other direction, fix $ \epsilon > 0 $ and pick $ \hat{\psi} \in \Psi_N $ with $ \| \hat{\psi} \|_{\mathcal{V}, \delta_n} = 1 $ and $ A(h, \hat{\psi}) \geq \sqrt{\Pi_\Psi (h)} - \epsilon, $ possible because $ \Psi_N $ is a symmetric cone. Evaluating the bracket at $ \psi = r \hat{\psi}, $ using Theorem \ref{thm:dev_bound}(i) in the other direction and $ \| r \hat{\psi} \|^2_{\hat{\mathcal{V}}, \delta_n} \leq 2 r^2, $ and letting $ \epsilon \downarrow 0, $
	\begin{equation} \label{eq:F_hat_lower_bound}
		\hat{F}(h) \geq \underset{r \geq 0}{\sup} \left[ r (\sqrt{\Pi_\Psi (h)} - \sqrt{2} \delta_n ) - \alpha r^2 \right] = \frac{( \sqrt{\Pi_\Psi (h)} - \sqrt{2} \delta_n )^2_+}{4 \alpha} \geq \frac{\Pi_\Psi (h)}{8 \alpha} - \frac{\delta_n^2}{2 \alpha},
	\end{equation}
	the last step by $ (x - y)^2_+ \geq \frac{1}{2} x^2 - y^2 $ for $ x, y \geq 0. $
	
	Step 2 (main inequality). Since $ \hat{h} $ minimizes and $ h_N \in \mathcal{H}_N, $ Theorem \ref{thm:dev_bound}(ii) gives
	\begin{equation*}
	\begin{aligned}
	\tfrac{1}{2} D (\hat{h}) - \delta_n^2 + \mu \hat{F}(\hat{h}) &\leq \hat{\mathcal{D}}(\hat{h}) - \hat{\mathcal{D}}(h^*) + \mu \hat{F} (\hat{h}) \leq \hat{\mathcal{D}}(h_N) - \hat{\mathcal{D}}(h^*) + \mu \hat{F} (h_N) \\
	&\leq 2 D(h_N) + \delta_n^2 + \mu \hat{F} (h_N).
	\end{aligned}
	\end{equation*}
	Inserting \eqref{eq:F_hat_lower_bound} on the left and \eqref{eq:F_hat_upper_bound} on the right, with $ q := \mu / \alpha \in [\kappa^{-2}, \kappa^2], $ gives $ \min \{ \frac{1}{2}, \frac{q}{8} \} ( D(\hat{h}) + \Pi_\Psi (\hat{h}) ) \leq \max \{ 2, 4 q \} ( D(h_N) + \Pi_\Psi (h_N) ) + (2 + 8.5 q) \delta_n^2, $ hence
	\begin{equation} \label{eq:D_h_N_bounds}
	D (\hat{h}) + \Pi_\Psi (\hat{h}) \leq C_\kappa \big( D (h_N) + \Pi_\Psi (h_N) + \delta_n^2 \big) \leq \sigma^{-c} \bar{\varepsilon}_n,
	\end{equation}
	with $ C_\kappa $ depending only on $ \kappa, $ using Step 0 and $ (1 + M_4)^2 \leq C d \sigma_*^{-4}. $ Only the ratio $ \mu / \alpha $ enters, and no choice of it trades bias against variance.
	
	Step 3 (coercivity). By Theorem \ref{thm:coercivity} for the classes of Definition \ref{def:trial_test_classes} ($ C_0 = 1 $), applied to $ \hat{h} $ with $ \tau = \delta_n $ and $ \delta = 1/2, $ together with \eqref{eq:D_h_N_bounds}, $ \Lambda_* \lesssim d^2 \sigma_*^{-4} $ and $ \ell \leq 1, $
	\begin{equation*}
	    \int_{I^\circ} \| \nabla \hat{\varphi} \|^2_{\rho_t} \nu (dt) + D(\hat{h}) \leq \sigma^{-c} \bar{\varepsilon}_n \leq \sigma^{-c} \log^8 (e n / \zeta) \, \varepsilon_n, \qquad \hat{\varphi} := \hat{h} - h^*.
	\end{equation*}
	
	Step 4 (clipping). Corollary \ref{cor:guidance_control}, applicable since $ \hat{h} \in \mathcal{H}_N $ and $ \mathrm{b} \in (0, \bar{B}], $ with $ \mathrm{b}^{-2} \leq \bar{B}^2 \mathrm{b}^{-4}, $ gives
	\begin{equation*}
	    \| \hat{g} - g^* \|^2_{L^2 (\nu \otimes \rho; I^\circ)} \leq \sigma^{-c} \log^8 (e n / \zeta) \, \mathrm{b}^{-4} \varepsilon_n + 4 (G^*)^2 \, \mathcal{E}_{\mathrm{b}},
	\end{equation*}
	where, as in its proof, $ \mathcal{E}_{\mathrm{b}} := \int_{I^\circ} \rho_t (h^* < \mathrm{b}) \nu (dt) \leq \Xi_\mathrm{s} \mathrm{b}^\mathrm{s} $ for $ \mathrm{s} < \infty. $
	Let $ \mathrm{s} < \infty, $ $ \mathrm{b}_0 := (\varepsilon_n / \Xi_\mathrm{s})^{1/(\mathrm{s} + 4)} $ and $ \mathcal{V}_n := \Xi_\mathrm{s}^{4/(\mathrm{s}+4)} \varepsilon_n^{\mathrm{s}/(\mathrm{s}+4)} = \mathrm{b}_0^{-4} \varepsilon_n = \Xi_\mathrm{s} \mathrm{b}_0^\mathrm{s}. $ If $ \mathrm{b}_0 \leq \bar{B}, $ then $ \mathrm{b} = \mathrm{b}_0 $ and both terms are at most $ \sigma^{-c} \log^8 (en/\zeta) \mathcal{V}_n, $ as $ G^* \leq \sigma^{-c}. $ If $ \mathrm{b}_0 > \bar{B}, $ then, because $ w \leq \bar{B} $ on $ \mathrm{supp} (p_0) $ gives $ \Xi_\mathrm{s} \bar{B}^\mathrm{s} = \mathbb{E}_{p_0} (\bar{B} / w)^\mathrm{s} \geq 1, $ we have $ \mathcal{V}_n = \Xi_\mathrm{s} \mathrm{b}_0^\mathrm{s} > \Xi_\mathrm{s} \bar{B}^\mathrm{s} \geq 1; $ since $ | \hat{g} | \leq G^* $ by construction and $ | g^* | \leq G^* $ on $ I $ by Lemma \ref{lem:bounded_guidance}, $ \| \hat{g} - g^* \|^2_{L^2 (\nu \otimes \rho; I^\circ)} \leq 4 (G^*)^2 \leq 4 (G^*)^2 \mathcal{V}_n. $ In both cases \eqref{eq:g_rate} follows. If $ \mathrm{s} = \infty, $ then $ w \geq B $ on $ \mathrm{supp}(p_0) $ gives $ h^* \geq B, $ so $ \mathcal{E}_B = 0 $ and the first term with $ \mathrm{b} = B \leq \bar{B} $ gives the stated bound.
\end{proof}

\begin{remark}(No variance amplification) \label{rem:no_amplification}
    Heuristically, the gap is a factor $N^2$ of variance: a resolution-$N$ value estimator has variance $N^d/n$, differentiating it amplifies this to $N^{d+2}/n$, and balancing against the gradient bias $N^{-2(\beta-1)}$ then forces $N^{2\beta+d} \asymp n$. Our estimator also differentiates its fit, $\hat g = \Pi_{G^\ast}(\nabla \hat h / (\hat h \vee \mathrm{b}))$, but $\hat h$ is not a pure value estimate: by \eqref{eq:weak_main} each test function supplies a linear moment condition on the first derivatives of $h^\ast$, so the gradient of the fit is controlled at complexity $\mathbb{V}_N \asymp N^d$ with no amplification, and the balance is $N^{2(\beta-1)+d} \asymp n$, as in Step 4.
\end{remark}

\section{From guidance error to sampling error} \label{app:sampling}

\begin{lemma}(The controlled marginals are a bounded reweighting) \label{lem:marginals_bounded}
    Under Assumptions \ref{ass:bounded}--\ref{ass:compact}, define 
    \begin{equation}
        \rho^*_t (dx) := \frac{h^* (t,x)}{Z} \rho_t (dx), \quad Z = \mathbb{E}_{p_0} [w] \in (0, \bar{B}].
    \end{equation}
    Then each $ \rho_t^* $ is a probability measure, $ \rho_T^* = q_0, $ and $ ( \rho_t^*)_{t \in [0, T]} $ is the marginal flow of the exactly guided SDE \eqref{eq:rewersed_process} started from $ \rho_0^*. $ Moreover 
    \begin{equation} \label{eq:upsilon_def}
        \Upsilon := \underset{t \in [0, T]}{\sup} \left\| \frac{d \rho_t^*}{d \rho_t} \right\|_\infty \leq \frac{\bar{B}}{Z} \leq \bar{B} \, \Xi_{\mathrm{s}}^{1 / \mathrm{s}},
    \end{equation}
    the last expression read as $ \bar{B} / B $ when $ \mathrm{s} = \infty. $
\end{lemma}
\begin{proof}
    Total mass: $ \int h^* (t, \cdot) \rho_t = \mathbb{E} [h^*(t, \overleftarrow{X}_t)] = \mathbb{E}[w (X_0)] = Z, $ by Proposition \ref{prop:martingale} and the tower property, so $ \rho_t^* $ is a probability measure. At $ t = T $ it is $ w p_0 / Z = q_0. $

    For the flow statement, $ L_t := h^* (t, \overleftarrow{X}_t) $ is a positive $\mathbb{P}$-martingale with $ d L_t = \sqrt{2} \nabla h^* \cdot d B_t $ by Theorem \ref{thm:harmonicity}, so $ d L_t / L_t = \sqrt{2} \nabla \log h^* \cdot d B_t $ and $ L_t / L_0 = \mathcal{E} ( \int \sqrt{2} \nabla \log h^* \cdot d B)_t. $ Put $ d \mathbb{Q} / d \mathbb{P} |_{\mathcal{F}_t} := L_t / Z. $ Girsanov's theorem makes $ \tilde{B}_t := B_t - \int_0^t \sqrt{2} \nabla \log h^* ds $ as a $\mathbb{Q}$-Brownian motion, so under $\mathbb{Q}$
    \begin{equation}
        d \overleftarrow{X}_t = (b + \sqrt{2} \cdot \sqrt{2} \nabla \log h^*) dt + \sqrt{2} d \tilde{B}_t = (b + 2 \nabla \log h^*) dt + \sqrt{2} d \tilde{B}_t,
    \end{equation}
    which is \eqref{eq:rewersed_process}, and the $\mathbb{Q}$-marginals are $ \mathbb{E}_{\mathbb{P}} [ f (\overleftarrow{X}_t) L_t ] / Z = \int f h^* \rho_t / Z, $ that is $ \rho_t^*. $ Note that the tilt acts on the initial condition as well: the exactly guided process must be started from $ \rho_0^* = h^* (0, \cdot) p_T / Z $ rather than from $ p_T. $ Initialization error is not addressed here. 

    For \eqref{eq:upsilon_def}, $ h^* \leq \bar{B} $ by Assumption \ref{ass:bounded}, while Jensen applied to the convex map $ u \mapsto u^{-\mathrm{s}} $ gives $ \mathbb{E}_{p_0} [w^{-\mathrm{s}}] \geq ( \mathbb{E}_{p_0} [w] )^{-\mathrm{s}}, $ i.e. $ Z \geq \Xi_{\mathrm{s}}^{-1 / \mathrm{s}}. $ At $ \mathrm{s} = \infty $ one has $ Z \geq B $ directly.
\end{proof}

\begin{corollary}(Sampling guarantee in the controlled metric) \label{cor:sampling_guarantee}
    Let $ J := [t_0, \bar{t}] \subseteq I^\circ, $ and let $ \mathbb{Q}^* $ and $ \hat{\mathbb{Q}} $ be the laws on $ C(J; \mathbb{R}^d) $ of 
    \begin{equation}
        d \overleftarrow{Z}_t = (b + 2 g^*) (t, \overleftarrow{Z}_t)dt + \sqrt{2} d B_t, \quad d \overleftarrow{Y}_t = (b + 2 \hat{g})(t, \overleftarrow{Y}_t) dt + \sqrt{2} d B_t,
    \end{equation}
    both started at time $ t_0 $ from $ \rho^*_{t_0}. $ Then $ \hat{\mathbb{Q}} $ and $ \mathbb{Q}^* $ are equivalent,
    \begin{equation} \label{eq:KL_equiv}
        \mathrm{KL} (\mathbb{Q}^* || \hat{\mathbb{Q}}) = \int_J \| \hat{g} - g^* \|^2_{\rho^*_t} dt \leq \ell \, \Upsilon \, \| \hat{g} - g^* \|^2_{L^2 (\nu \otimes \rho; J)},
    \end{equation}
    and, writing $ \hat{\rho}_{\bar{t}} := \mathrm{Law}( \overleftarrow{Y}_{\bar{t}} ) $ for the time-$\bar{t}$ marginal of the estimated sampler,
    \begin{equation} \label{eq:tv_bound}
        \mathrm{TV} (\rho^*_{\bar{t}}, \hat{\rho}_{\bar{t}})^2 \leq \frac{1}{2} \ell \, \Upsilon \, \| \hat{g} - g^* \|^2_{L^2 (\nu \otimes \rho; J)}.
    \end{equation}
    In particular, under the hypotheses of Theorem \ref{thm:guidance_rate}, for $ n \geq n_0 $ and on the event of Theorem \ref{thm:dev_bound},
    \begin{equation} \label{eq:particular_tv_bound}
        \mathrm{TV} (\rho^*_{\bar{t}}, \hat{\rho}_{\bar{t}})^2 \leq \sigma^{-c} \log^8 (e n / \zeta) \, \ell \bar{B} \, \Xi_{\mathrm{s}}^{\frac{1}{\mathrm{s}} + \frac{4}{\mathrm{s} + 4}} \varepsilon_n^{\frac{\mathrm{s}}{\mathrm{s} + 4}},
    \end{equation}
    which at $ \mathrm{s} = \infty $ reads $ \sigma^{-c} \log^8 (e n / \zeta) \, \ell \bar{B} B^{-5} \varepsilon_n, $ of order $ n^{-2/(d+2)} $ up to logarithmic factors at $ \beta = 2. $
\end{corollary}
\begin{proof}
    The drift difference is $ 2 (\hat{g} - g^*), $ of modulus at most $ 4 G^* $ since $ | \hat{g} | \leq G^* $ by construction \eqref{eq:proj_guidance} and $ | g^* | \leq G^* $ by Lemma \ref{lem:bounded_guidance}. With $ \theta := \sqrt{2} (\hat{g} - g^*), $ Novikov's condition holds trivially
    \begin{equation}
        \mathbb{E}_{\mathbb{Q}^*} \exp \left( \frac{1}{2} \int_J | \theta |^2 dt \right) \leq \exp (4 (G^*)^2 \ell) < \infty,
    \end{equation}
    so the two path laws are equivalent and 
    \begin{equation}
        \mathrm{KL}(\mathbb{Q}^* || \hat{\mathbb{Q}}) = \mathbb{E}_{\mathbb{Q}^*} \left[ \frac{1}{2} \int_J | \theta |^2 dt \right] = \int_J \mathbb{E}_{\rho_t^*} | \hat{g} - g^* |^2 dt,
    \end{equation}
    which is the identity in \eqref{eq:KL_equiv}. This is where the projection $ \Pi_{G^*} $ earns its keep beyond Corollary \ref{cor:guidance_control}. The unprojected ratio obeys only $ | \nabla \hat{h} / (\hat{h} \lor \mathrm{b}) | \leq G / \mathrm{b}, $ and $ \mathrm{b} \asymp (\varepsilon_n / \Xi_{\mathrm{s}})^{1 / (\mathrm{s} + 4)} \downarrow 0 $ by \eqref{eq:N_mu_alpha_b_choice}, so that bound degrades with $ n $ and the Novikov constant with it, whereas $ G^* = 2 \sqrt{d} \sigma_*^{-2} $ is independent of $ n. $

    The inequality in  \eqref{eq:KL_equiv} is Lemma \ref{lem:marginals_bounded}: $ \int_J \| \cdot \|^2_{\rho_t^*} dt \leq \Upsilon \int_J \| \cdot \|^2_{\rho_t} dt = \ell \, \Upsilon \int_J \| \cdot \|^2_{\rho_t}  \nu(dt), $ using $ dt = \ell \, \nu(dt). $ Then \eqref{eq:tv_bound} follows from Pinsker's inequality and the data-processing inequality for the coordinate map $ \pi_{\bar{t}}: C(J; \mathbb{R}^d) \rightarrow \mathbb{R}^d, $ and \eqref{eq:particular_tv_bound} from Theorem \ref{thm:guidance_rate} together with $ \| \cdot \|_{L^2 (\nu \otimes \rho; J)} \leq \| \cdot \|_{L^2 (\nu \otimes \rho; I^\circ)} $ and \eqref{eq:upsilon_def}.
\end{proof}

\begin{remark}(The window $ I^\circ $) \label{rem:sampling_window}
    Corollary \ref{cor:coercivity} and Theorem \ref{thm:coercivity} deliver the gradient bound on the half-window $ I^\circ $ and not on $ I. $ The restriction is intrinsic to the endpoint-averaging device and not an artifact of taking $ \nu $ uniform: for any probability measure $ \nu $ on $ I $ without an atom at $ t_1 $ the weight $ \omega (t) = \nu ([t, t_1]) $ is a survival function vanishing at $ t_1. $ The argument leaves $ \| \varphi(t_1, \cdot) \|^2_{\rho_{t_1}} $ uncontrolled near the right endpoint, i.e. a value error at a single instant, since the statistical analysis controls only the $\nu$-average $ \mathcal{D}(h), $ and no reweighting of the spatial marginal manufactures an instantaneous quantity out of a time-average.

    At the same time, given a target sampling window $ J = [t_0, \bar{t}] $ and any $ \delta > 0, $ we can run the whole construction on the enlarged estimation window $ \tilde{I} := [t_0, \bar{t} + \delta], $ for which $ J = \tilde{I}_{\delta / \tilde{\ell}} $ in the notation of \eqref{eq:delta_coercivity}, with $ \tilde{\ell} := \bar{t} + \delta - t_0, $ provided $ \tilde{\ell} \leq 1 $ and $ \bar{t} + \delta < T. $ Every statement from Section \ref{sec:coercivity} onwards holds verbatim with $ \ell $ replaced by $ \tilde{\ell}, $ with the constant 2 of \eqref{eq:coercivity_on_half} replaced by $ \tilde{\ell} / \delta, $ and with $ \sigma_* $ replaced by $ \sigma_{T - \bar{t} - \delta}. $ Taking $ \delta = \bar{t} - t_0 $ doubles the window and keeps the constant at 2. So, the estimator must be run closer to the singular end than the sampler is, and the $ \sigma^{-c} $ prefactors of Theorem \ref{thm:guidance_rate} and Corollary \ref{cor:sampling_guarantee} are then evaluated at the deeper cutoff. Nothing else in the argument changes, and in particular Corollary \ref{cor:sampling_guarantee} applies with $ J = \tilde{I}_{\delta / \tilde{\ell}} $ in place of $ I^\circ. $
\end{remark}

\begin{remark}[From the half-window to $ I_\delta $] \label{rem:delta_window}
    Several statements above are proved on $ I^\circ = [t_0, (t_0 + t_1)/2], $ where $ \omega \geq 1/2. $ Fix $ \delta \in (0,1) $ and set $ I_\delta := [t_0, t_1 - \delta \ell], $ so that $ \omega (t) = (t_1 - t)/\ell \geq \delta $ on $ I_\delta $ and $ I_{1/2} = I^\circ. $ The weight enters only when $ E_\omega $ is converted into an unweighted integral, and there $ \omega \geq \delta $ gives
    \begin{equation} \label{eq:delta_window_coercivity}
        \int_{I_\delta} \| \nabla \varphi \|^2_{\rho_t} \nu (dt) \leq \frac{1}{\delta} E_\omega,
    \end{equation}
    so the left-hand side is at most $ \delta^{-1} $ times the right-hand side of \eqref{eq:coercivity} at $ \theta = \ell, $ equivalently $ (2 \delta)^{-1} $ times the right-hand side of \eqref{eq:coercivity_on_half}, while Theorem \ref{thm:coercivity} is already stated on $ I_\delta. $ Corollary \ref{cor:guidance_control} holds verbatim with $ I^\circ $ replaced by $ I_\delta, $ its proof being unchanged. Hence Theorem \ref{thm:guidance_rate} and Corollary \ref{cor:sampling_guarantee} hold on $ I_\delta, $ resp.\ for $ J \subseteq I_\delta, $ with an extra factor $ \delta^{-1}: $ apply Theorem \ref{thm:coercivity} with this $ \delta $ in Step 3 of the proof of Theorem \ref{thm:guidance_rate}, and use $ | \hat{g} - g^* | \leq 2 G^* $ on $ I_\delta \subseteq I $ in the capped case of Step 4. What cannot be reached is $ t_1 $ itself, where $ \omega(t_1) = 0. $
\end{remark}

\section{Further details of the numerical experiments} \label{app:experiments}

This section records how the experiments of Section \ref{sec:experiments} are implemented, and gives the
full sweeps behind the numbers quoted there.

\subsection{General implementation details} \label{app:real_setup}

\textbf{Base models.} Both are public unconditional DDPM checkpoints with $\epsilon$-prediction and
$1000$ training steps, frozen throughout and sampled in $100$ DDIM steps at $\eta = 0$, that is
deterministically (Table \ref{tab:real_models})
\citep{ddpmcifar10ckpt,godoy2023}. Throughout this section $\eta$ is the DDIM noise level, unrelated to
the width of the analytic strip carrying the same name in the approximation lemmas. The target is the superclass of the six animal classes on
CIFAR-10 \citep{krizhevsky2009} and of the odd digits on MNIST \citep{lecun1998}. MNIST model was fine-tuned
with random horizontal flips and draws some digits mirrored, so the MNIST classifier is applied to an
image and to its mirror and keeps the more confident prediction.

\begin{table}[ht]
\centering
\caption{Base models. ``Share'' is the fraction of base samples the held-out head assigns to the target
superclass, ``SIR-32'' the same after $32$-fold importance resampling with weights $w$, and FID is to the
model's own reference sample (Appendix \ref{app:metrics}).}
\label{tab:real_models}
\small
\begin{tabular}{llcccc}
\hline
 & checkpoint & parameters & share & SIR-32 & FID \\
\hline
CIFAR-10 & \texttt{google/ddpm-cifar10-32} & $35.7$M & $0.653$ & $0.854$ & $24.5$ \\
MNIST & \texttt{dvgodoy/ddpm-cifar10-32-mnist} & $35.7$M & $0.523$ & $0.630$ & $7.4$ \\
\hline
\end{tabular}
\end{table}

\textbf{Weight and heads.} $r(x_0) = \log \sum_{k \in C} p_\varphi(k \mid x_0)$ and $w = \epsilon +
(1-\epsilon)\,\mathrm{sigmoid}((r - r_{\mathrm{ref}})/\alpha)$, with $r_{\mathrm{ref}}$ the median of $r$
on the reference sample and $\alpha$ half its interquartile range, so $\epsilon \leq w \leq 1$ and
Assumption \ref{ass:bounded} holds with $\mathrm{s} = \infty$. Two heads of one architecture are trained
on disjoint halves of the labelled set: the tilt head defines $w$ and is used by every estimator,
the held-out head scores every reported number and is used by none.

\textbf{Reference sample.} The $n$ reference points are samples of the base model itself, so $\rho_t$ are
its own marginals and $s$ in \eqref{eq:weak_main} is its network, they are noised at $12$ levels in the
window $\sigma_{T-t} \in [0.35, 0.95]$. The guidance is zero above the window, and below it $\hat h$ is
evaluated at the window's lower edge.

\textbf{Fitted classes.} Those of Definition \ref{def:trial_test_classes}. The trial class is a
convolutional network of three blocks, width $32$ doubling with depth, with the rectified cubic
$\varrho(u) = (u \vee 0)^3$ under a smooth saturation at the activation budget $\mathcal{A} = 3$, the
retraction $\Theta$ on the output, the spatial clip $\chi_R$ on the input, and the weight budget
$\Lambda = 5$ imposed by projection after every step, a convolution is a sparse linear map, so such a
network is a member at these budgets in $d = 3072$, where a dense map is not. The test class is the cone
over differences of two trial networks, which by Lemma \ref{lem:networks_difference} is the network of
twice the width and one more layer with $\Theta$ as the activation of the appended layer, so
\ref{property:absorption} holds as an identity, the ratio measured before each fit is $1.000000$. Four
things depart from the letter of the definition: the sup-norm constraints are read on the batch rather
than over $I \times \mathbb{R}^d$ (measured, $\sup|\nabla h| = 0.24$ against $G = 1.8 \cdot 10^3$ and
$\sup|\partial_t h| = 2.9$ against $G_t = 8.2 \cdot 10^5$, so they never bind), the sparsity budget is
whatever the convolutions realise, a group normalisation and a smooth time embedding are kept for
trainability, both $C^\infty$ so the class stays $C^{1,2}$, neither of the form $\varrho(Ax+b),$ and the
penalty is a plug-in square of a batch mean, with a positive bias of order $1/\text{batch}$.

\textbf{Sampling and cost.} The guided sampler replaces $\hat\epsilon$ by $\hat\epsilon -
\sigma\,c\,\nabla \log \hat h$, capped at $10\%$ of $\|\hat\epsilon\|$ for every guided method. Each
configuration draws $3$ seeds of $2000$ samples with the initial noise of sample $i$ shared across
methods. The CIFAR-10 sweep took about three hours on one RTX 4090, the MNIST sweep about two. The code runs from a single entry point with one configuration file per base model, in PyTorch
on top of \texttt{diffusers}, with FID from \texttt{torchmetrics}.

\subsection{The empirical objective} \label{app:objective}

With $\omega(u) = (u - u_{\mathrm{lo}})/(u_{\mathrm{hi}} - u_{\mathrm{lo}})$, $(t_i, X_i)_{i \leq m}$ the
noised reference points, $w_i = w(X_0^i)$ and $s = -\hat\epsilon/\sigma$, the estimator minimises over
$\Hcal$
\[
\frac{1}{m}\sum_i \big( h(t_i, X_i) - w_i \big)^2 + \frac{\mu}{2} \sup_{0 \neq \psi \in \Psi}
\frac{\hat a(h, \psi)^2}{\|\psi\|^2_{\hat{\mathcal{V}}}} ,
\quad
\hat a(h, \psi) = \frac{1}{m}\sum_i \Big[ \big( \partial_t h + (x + s)\cdot\nabla h \big) \psi - \nabla
h \cdot \nabla \psi \Big](t_i, X_i) ,
\]
the second term being $\mathcal{N}_{\Psi}(h)^2/2$: $\Psi$ carries a free real scalar and the quotient is
invariant under it, so the radial part is solved in closed form and only the direction is optimised.
Reverse time runs backwards in $u$, so $\partial_t h = -\partial_u h$ in the coordinate the reference
points are indexed by.

Trial and test networks are trained at two time scales by Adam, two ascent steps per descent step, with a
warmup on the learning rate and on $\mu$ and a projection of every weight onto $\Lambda$ after each step,
the scale of the test network is held near one, which does not move the maximiser. Each batch draws eight
of the twelve levels. The first tenth of the steps runs at $\mu = 0$, after which $\mu$ is set so that the
penalty equals the value term at that point rather than tuned: on CIFAR-10 this gives $\mu = 0.0599$, and
the penalty then falls by more than three orders of magnitude.

The inner supremum over a network class is nonconvex, so the penalty the fit reports is a lower bound on
$\mathcal{N}_{\Psi}(\hat h)$. It is therefore also measured against cones the fit never saw, $\omega \cdot
\mathrm{span}\,\phi$ over $Q = 1024$ random cosine features of pooled pixels, where the supremum is a
closed-form quadratic and the plug-in bias is removed by cross-fitting over two halves of $24{,}576$
points: $\mathcal{N}^2 = 2.1 \cdot 10^{-3}$ at $8 \times 8$ pooling, $1.8 \cdot 10^{-3}$ against a second
cone with independent frequencies, $1.1 \cdot 10^{-3}$ at $16 \times 16$. A value-only least-squares fit
gives $5.7 \cdot 10^{-3}$ on the same cones and a linear instance of the estimator, whose inner supremum
is exact, gives $5.1 \cdot 10^{-4}$.

\subsection{Metrics used in the experiments} \label{app:metrics}
\begin{enumerate}[leftmargin=*]
\item \emph{Percentage of the exact tilt's shift}: $100$ times the paired gain in held-out reward over
the base model, divided by $\E_{q_0} r - \E_{p_0} r$ computed by self-normalized importance sampling on
the $6000$ base samples, values above $100\%$ are possible for $c > 1$. Standard errors are over the
$6000$ paired samples, the reported standard deviations over the three seeds, and are the larger of the two.
\item \emph{Share}: the fraction of samples the held-out head assigns to the target superclass.
\item \emph{FID} \citep{heusel2017}: Fr\'echet distance between $2048$-dimensional Inception
\citep{szegedy2016} features of the $2000$ samples of a seed and of $2000$ points of the model's own
reference sample, averaged over the $3$ seeds, so it measures departure from the base distribution.
\item \emph{Transfer ratio}: the paired gain on the tilt head divided by that on the held-out head, a
check that a reported gain is not specific to the head defining $w$.
\end{enumerate}

\subsection{The synthetic instance} \label{app:synthetic}

\textbf{The instance.} $p_0$ is atomic on $M = 512$ points drawn uniformly from the ball of radius $0.8$,
so Assumption \ref{ass:compact} holds verbatim, and the tilt obeys $0.5 \leq w \leq 3$, so Assumption
\ref{ass:bounded} holds with $\bar{B} = 3$ at the boundary case $\mathrm{s} = \infty$. On such a $p_0$,
$h^\ast$ and its derivatives, the score $s$ and $\nabla \cdot s$ are softmax averages over the atoms,
closed-form at any $d,$ the score is never learned. With $T = 1$ the window is $I = [0.40, 0.85]$ and
errors are read on $I_\delta$ at $\delta = 1/2$, relative to $\| g^\ast \|^2_{L^2 (\nu \otimes \rho;
I_\delta)}$, by Monte Carlo on an $8000$-point cloud. A replicate resamples atoms, masses, tilt, knots,
sample and evaluation cloud, bars are $\pm 1$ standard deviation over $8$ replicates.

\textbf{Trial class and sweeps.} Tensor B-splines times a Chebyshev basis in time, cubic because
\eqref{eq:wei_main} needs $C^{1,2}$, hence $\beta = 4,$ here \eqref{eq:h_empirical_main} is a single
linear solve, the inner supremum being a closed-form quadratic in the coefficients. Seven sample sizes
over three decades, resolution chosen by oracle on a grid of nine values and the penalty weight over
eleven decades, slopes are fitted on the seed mean and bootstrapped over seeds. A fitted exponent is the
rate only if the oracle's resolution stays interior to its grid, which happens at $d = 1$ but not at
$d = 2, 3$, where the oracle sits at or near the coarsest cell, the exponent is therefore read at $d = 1$. The comparison
with least squares does not depend on the exponent being readable, and is reported at all three
dimensions in Table \ref{tab:rate}: the weak residual is ahead in every one of the $20$ cells swept, by
a margin that grows with $d$.

\textbf{Penalty ablation.} Four estimators are fitted on one sample and read on one evaluation cloud, each
at its own oracle $(N, \mu)$ over the same grids: least squares on the values ($\mu = 0$), the Sobolev
penalty $\int_I \| \nabla h \|^2_{\rho_t} \nu (dt)$, the strong residual $\| \mathcal{R}[h]
\|^2_{L^2(\nu \otimes \rho)}$ and the weak residual of \eqref{eq:h_empirical_main}. The strong residual
carries the reverse drift, $\partial_t + (x + 2s) \cdot \nabla + \Delta$, the form with $(x+s)$ and
$-\Delta$ being its integration by parts against $\rho_t$ and valid only under the integral, the weak
penalty is a quadratic form of a sample mean, whose plug-in estimate carries a bias of order $P/n$, so it
is read across two independent halves, as the image fits do.

\begin{table}[ht]
\centering
\caption{Penalty ablation at $d = 1$ in the cubic class ($\beta = 4$), $8$ seeds, each penalty at its own oracle
$(N, \mu)$: relative guidance error at the smallest and the largest sample size, against the exponent the
theory predicts for it. A rate is predicted only for the weak residual and, from the classical
derivative-recovery bound, for least squares, the Sobolev penalty carries a regularization bias and the
strong residual an approximation floor (Table \ref{tab:penalties}), and neither is analysed here.
``Value LS'' is the $\mu = 0$ baseline of Figure \ref{fig:rate}.}
\label{tab:ablation}
\small
\begin{tabular}{lccc}
\hline
penalty & error at $n = 10^3$ & error at $n = 10^6$ & predicted exponent \\
\hline
value LS ($\mu = 0$) & $6.46 \times 10^{-2}$ & $9.71 \times 10^{-5}$ & $0.667$ \\
Sobolev & $4.41 \times 10^{-2}$ & $1.40 \times 10^{-4}$ & --- \\
strong residual & $4.46 \times 10^{-3}$ & $2.34 \times 10^{-5}$ & --- \\
weak residual (ours) & $6.51 \times 10^{-3}$ & $1.64 \times 10^{-5}$ & $\mathbf{0.857}$ \\
\hline
\end{tabular}
\end{table}

\begin{figure}[ht]
\centering
\includegraphics[width=0.55\textwidth]{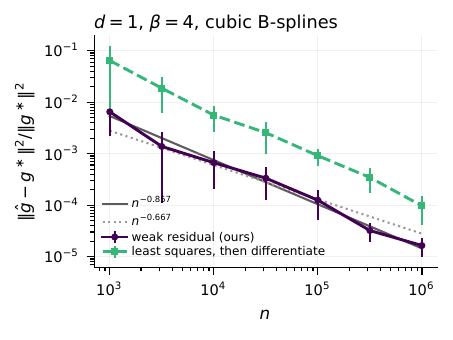}
\caption{Guidance error against $n$ at $d=1$ in the tensor cubic B-spline class ($\beta=4$), for the weak-residual estimator and for least squares on the values followed by differentiation, against the predicted slopes $0.857$ and $0.667$, intercepts fitted only. Bars are $\pm1$ standard deviation over $8$ seeds. The weak residual's fitted exponent, $0.842 \pm 0.026$, matches its prediction, least squares runs above its own, $0.909 \pm 0.038$ against $0.667$, so over this range it is measured far from its asymptote and its steeper slope is a climb out of a worse level rather than a faster rate.}
\label{fig:rate}
\end{figure}

\begin{table}[ht]
\centering
\caption{The two estimators at equal budget, $8$ seeds, each at its own oracle $(N, \mu)$ on the same
grids. ``Differentiate'' is least squares on the values, then differentiated. Entries are relative
squared guidance error, the largest $n$ is $10^6$, and $3.2 \cdot 10^5$ at $d = 3$.}
\label{tab:rate}
\small
\IfFileExists{plots/tab_rate.tex}{
\begin{tabular}{@{}l cc cc cc@{}}
\hline
 & \multicolumn{2}{c}{weak residual (ours)} & \multicolumn{2}{c}{differentiate} & \multicolumn{2}{c}{ratio} \\
\cline{2-3}\cline{4-5}\cline{6-7}
$d$ & $n = 10^3$ & largest $n$ & $n = 10^3$ & largest $n$ & smallest & largest \\
\hline
$1$ & $6.5 \times 10^{-3}$ & $1.6 \times 10^{-5}$ & $6.5 \times 10^{-2}$ & $9.7 \times 10^{-5}$ & $5.9\times$ & $13.5\times$ \\
$2$ & $2.3 \times 10^{-2}$ & $6.9 \times 10^{-5}$ & $3.5 \times 10^{-1}$ & $4.2 \times 10^{-4}$ & $5.4\times$ & $15.3\times$ \\
$3$ & $8.3 \times 10^{-2}$ & $3.7 \times 10^{-4}$ & $2.5$ & $3.6 \times 10^{-3}$ & $9.7\times$ & $30.2\times$ \\
\hline
\end{tabular}
}{\textsc{[run \texttt{figures.paperFigures}]}}
\end{table}

\subsection{Evaluation on CIFAR-10} \label{app:cifar}

\textbf{Estimators compared.} Both baselines are implemented against their authors' public code rather
than from the papers alone: DEFT from \texttt{github.com/alexdenker/DEFT} (commit \texttt{2495d46}),
DOIT from \texttt{github.com/liamyzq/Doob\_training\_free\_adaptation} (commit \texttt{9cbc541}), where
code and paper differ we follow the code. DEFT \citep{denker2024} trains a head for the $h$-transform on
the same reference sample and weight, in its authors' parametrisation, their equation (12): the head
reads the noisy state, the Tweedie estimate $\hat x_0$ and the gradient $\nabla_{\hat x_0} \log w(\hat
x_0)$, and returns a network output plus a learned time-dependent multiple of that gradient. It shares
the scale $c$ and the cap with every other method, and the deterministic sampler is its authors' own
setting rather than a liberty we take: their objective is denoising score matching against the forward
process, with no sampler entering it, and the configuration shipped with their code sets the DDIM $\eta$
to $0$.

DOIT \citep{zhu2026doit} is training-free, estimating $h$ and $\nabla h$ at each step by Monte Carlo over
$M$ one-step look-ahead samples. Two of its choices are not ours to make. Its terminal condition is the
exponential tilt $h(\cdot, 0) \propto \exp(r/\tau)$ rather than the bounded $w$ of Assumption
\ref{ass:bounded}, and has to be: the estimator resolves only the contrast between the $M$ weights, and
forcing the shared $w$ on it leaves them too close together to steer, the field then having cosine
$0.01$ with the reward gradient where the tilt gives $0.16$--$0.20$. And what it differentiates is the
sampler's own transition density, so it is defined only for a stochastic sampler, its reference
implementation drawing the look-ahead cloud from the kernel it samples with. Our headline protocol is
deterministic, where that density is a point mass, so DOIT is reported at $\eta = 0.4$ as well, the value
its authors use, with ours resampled there, the reported DEFT head was not, so that protocol carries ours
and DOIT only. Under $\eta = 0$ the look-ahead spread has to be supplied by hand, a departure from the
method which we report as such. We sweep DOIT over both its parameters, the temperature $\tau$ and the
strength $\gamma$, the latter playing the role of $c$ under the same cap, and below the shared sweep as
well, its authors' own range being $\gamma \leq 0.5$. That range is the only one in which the method is
competitive at all, and only on one of the two datasets: on MNIST $\gamma = 0.01$ buys $88.5\%$ of the
exact tilt's shift at FID $14.1$ against a base of $7.3$, its one point on a usable frontier, where we
reach $176\%$ at FID $13.8,$ on CIFAR-10 the same $\gamma$ moves the held-out reward the wrong way,
$-43.8\%$, and the best small-$\gamma$ point is $116\%$ at FID $222$.

\begin{table}[ht]
\centering
\caption{CIFAR-10, animals, along the guidance scale ($\gamma$ for DOIT): held-out reward shift as \% of the exact tilt's, and FID to the model's own reference sample. Base FID $24.5 \pm 0.2$, share $0.653$. Bold marks, at each scale, the largest shift and the smallest FID, a row holds $c$ fixed, which leaves the three methods at different FIDs, so the operative reading is the frontier of Figure \ref{fig:frontier_all} and not the row. DOIT is entered at the three $\gamma$ of the shared sweep, where its cap binds on every guided step so that the three agree to $0.2$ points, and at the small $\gamma$ its authors' own setting calls for. Daggered rows are that setting --- $\eta = 0.4$ and no cap, read against the base of the same sampler --- and are not comparable with the rest of the row they sit in, the wider grid over $\gamma$, $\tau$ and the sampler is in Appendix \ref{app:cap}. Here and in Table \ref{tab:mnist}, $\pm$ is one standard deviation over the three seeds, which dominates the standard error over the $6000$ paired samples (at most $2.3$ for the shift).}
\label{tab:cifar}
\small
\begin{tabular}{ccccccc}
\hline
 & \multicolumn{2}{c}{ours} & \multicolumn{2}{c}{DEFT} & \multicolumn{2}{c}{DOIT} \\
$c$ & \% & FID & \% & FID & \% & FID \\
\hline
$0.01^\dagger$ & --- & --- & --- & --- & $-43.8 \pm 5.1$ & $43.1 \pm 0.7$ \\
$0.05^\dagger$ & --- & --- & --- & --- & $38.7 \pm 2.0$ & $87.0 \pm 1.2$ \\
$0.1^\dagger$ & --- & --- & --- & --- & $116.0 \pm 2.5$ & $221.6 \pm 1.0$ \\
$0.25$ & $3.2 \pm 0.3$ & $24.5 \pm 0.2$ & --- & --- & --- & --- \\
$0.5$ & $6.1 \pm 0.4$ & $24.6 \pm 0.2$ & --- & --- & --- & --- \\
$1$ & $11.7 \pm 0.8$ & $24.7 \pm 0.2$ & $43.0 \pm 3.0$ & $25.8 \pm 0.4$ & $70.5 \pm 8.0$ & $116.7 \pm 0.9$ \\
$2$ & $21.4 \pm 1.6$ & $25.0 \pm 0.3$ & $78.2 \pm 5.2$ & $29.7 \pm 0.9$ & --- & --- \\
$3$ & $29.7 \pm 2.3$ & $25.3 \pm 0.3$ & $101.8 \pm 6.1$ & $36.2 \pm 1.0$ & --- & --- \\
$5$ & $44.3 \pm 3.3$ & $26.1 \pm 0.4$ & $122.6 \pm 7.7$ & $53.8 \pm 0.4$ & $70.7 \pm 7.8$ & $116.7 \pm 0.8$ \\
$10$ & $69.3 \pm 3.1$ & $28.7 \pm 0.5$ & $125.1 \pm 7.3$ & $83.4 \pm 0.5$ & --- & --- \\
$20$ & $94.0 \pm 5.1$ & $32.9 \pm 0.7$ & $121.8 \pm 7.4$ & $97.7 \pm 0.8$ & $70.6 \pm 7.8$ & $116.6 \pm 0.9$ \\
\hline
\end{tabular}
\end{table}

Transfer ratios fall towards one along the sweep, $1.44 \to 1.19$ for ours and $1.37 \to 1.15$ for DEFT.
The sweeps are not aligned, so we read them at a matched shift as the frontiers are read, there the two
are indistinguishable, $1.23$ against $1.26$ at $78\%$ of the exact tilt's shift and $1.19$ against
$1.20$ at $94\%$, and we claim no difference, having no standard error for the ratio. For DOIT the ratio
is $-0.46$ at every $\gamma$, and its sign is the reading that matters: the held-out head gains $0.63$
while the head defining $w$, from which DOIT's own terminal condition is built, loses $0.29$. At FID
$117$ the samples are far enough outside the model's distribution that the two classifiers disagree
about which way the reward moved, so the $70.6\%$ in the table is not an alignment gain. The pairs of Figure \ref{fig:main}c are the five largest rises in the held-out probability among
the $64$ samples of seed $0$ at $c = 20$, a rule rather than a hand-picked set, the same $64$ are shown
uncurated in Figure \ref{fig:cifar_samples}. Over all $6000$ pairs that probability rises from below $0.3$
to above $0.7$ in $13.4\%$ of them and falls the other way in $0.07\%$.

\begin{figure}[ht]
\centering
\includegraphics[width=\textwidth]{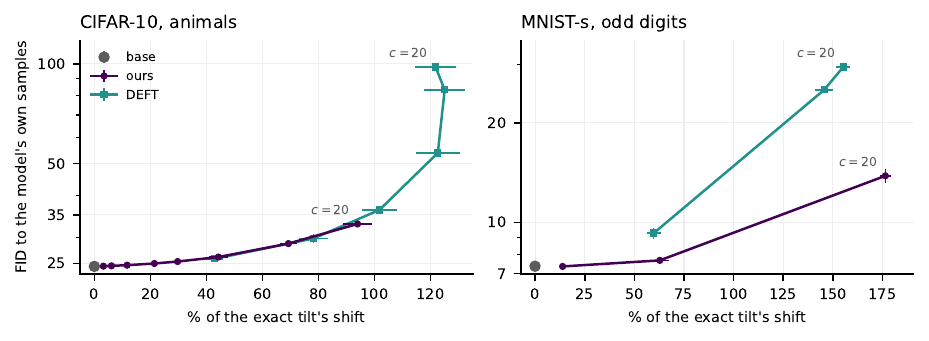}
\caption{What the alignment costs. Horizontal is the held-out reward shift as a percentage of the exact
tilt's, vertical the FID to the model's own reference sample, so the grey point is the base model and a
lower curve is cheaper, bars are $\pm1$ s.d.\ over three seeds on both axes. Left: CIFAR-10,
$c \in [0.25, 20]$. Right: MNIST, $c \in \{1, 5, 20\}$. DOIT is off both panels, at FID $117$ and
$153$ against bases of $24.5$ and $7.4,$ drawing it would compress the region the figure exists to
show, and its numbers are in Tables \ref{tab:cifar} and \ref{tab:mnist}.}
\label{fig:frontier_all}
\end{figure}

\begin{figure}[ht]
\centering
\includegraphics[width=\textwidth]{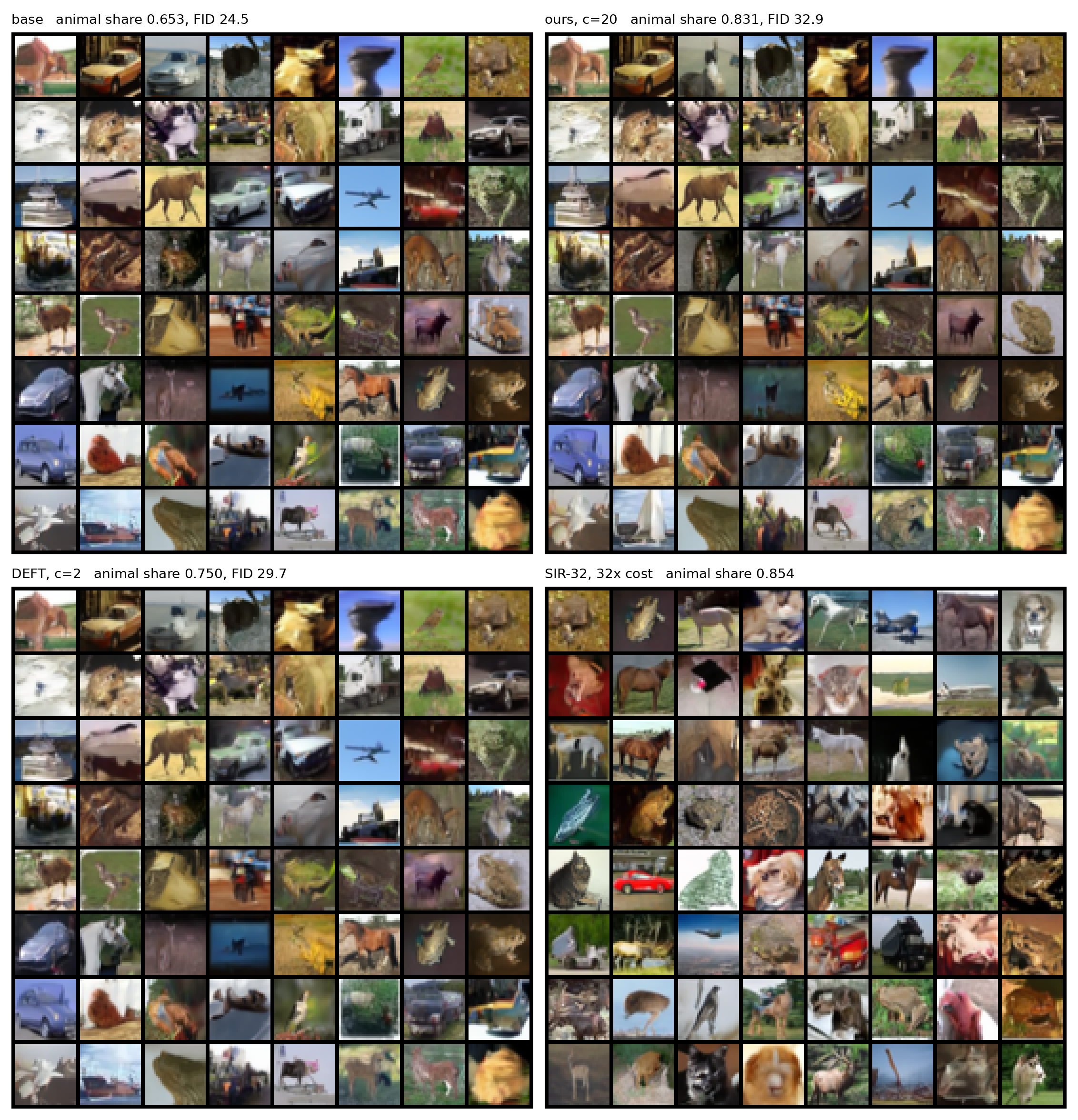}
\caption{CIFAR-10, animals, the first $64$ samples of seed $0$, uncurated. Base and ours at $c = 20$ above, DEFT and SIR-32 below, each panel labelled with the share of animals it carries and with its FID. DEFT is drawn at $c = 2$, the scale of its sweep whose FID is closest to the one our own panel is drawn at, because a matched departure from the base distribution is the reading under which the two frontiers are compared, at a matched $c$ it would instead sit at FID $98$. The guided panels share the initial noise with the base panel, so a sample can be compared with its counterpart in the same position, SIR-32 does not, by construction.}
\label{fig:cifar_samples}
\end{figure}

\subsection{Evaluation on MNIST} \label{app:mnist}

\begin{table}[ht]
\centering
\caption{MNIST, odd digits, tilted under the recipe of Appendix \ref{app:real_setup}: held-out reward
shift as \% of the exact tilt's and FID to the model's own reference sample. Base share $0.523$, FID
$7.4 \pm 0.1,$ SIR-32 share $0.630$. Bold marks, at each scale, the largest shift and the smallest FID,
and as in Table \ref{tab:cifar} the row, which holds $c$ fixed, is not the operative comparison. Daggered
rows are DOIT at the small $\gamma$ and in the setting its authors use, $\eta = 0.4$ and no cap, read
against the base of that sampler, $\gamma = 0.01$ is the one point at which DOIT reaches a usable
frontier at all. $\pm$ is one standard deviation over the three seeds.}
\label{tab:mnist}
\small
\begin{tabular}{ccccccc}
\hline
 & \multicolumn{2}{c}{ours} & \multicolumn{2}{c}{DEFT} & \multicolumn{2}{c}{DOIT} \\
$c$ & \% & FID & \% & FID & \% & FID \\
\hline
$0.01^\dagger$ & --- & --- & --- & --- & $88.5 \pm 3.7$ & $14.1 \pm 0.8$ \\
$0.05^\dagger$ & --- & --- & --- & --- & $280.0 \pm 7.9$ & $99.4 \pm 2.2$ \\
$0.1^\dagger$ & --- & --- & --- & --- & $283.4 \pm 7.9$ & $108.0 \pm 1.4$ \\
$1$ & $13.9 \pm 1.0$ & $7.4 \pm 0.1$ & $59.8 \pm 3.5$ & $9.3 \pm 0.3$ & $296.4 \pm 4.1$ & $153.3 \pm 3.6$ \\
$5$ & $62.8 \pm 4.7$ & $7.7 \pm 0.1$ & $145.6 \pm 4.6$ & $25.2 \pm 0.2$ & $296.4 \pm 4.0$ & $153.1 \pm 3.5$ \\
$20$ & $176.4 \pm 2.8$ & $13.8 \pm 0.7$ & $155.2 \pm 3.5$ & $29.5 \pm 0.7$ & $296.3 \pm 4.0$ & $152.7 \pm 3.5$ \\
\hline
\end{tabular}
\end{table}

The shift exceeds that of the exact tilt from $c = 5$ on, as a scale $c > 1$ allows. Read at a fixed
departure from the base distribution rather than at a fixed $c$, the ordering is the one of CIFAR-10: at
the FID DEFT reaches at $c = 1$ our shift is about $92\%$ against its $59.8\%$, and at our $c = 20$ it is
$176.4\%$ against about $84\%$ for DEFT interpolated to the same FID. Transfer ratios grow with the scale
here, $1.37 \to 1.54$ for ours and $1.56 \to 1.66$ for DEFT, so ours is the smaller of the two at every
scale. DOIT reports $296\%$ at FID $153$, twenty times the base, here, unlike on CIFAR-10, both heads
move the same way, ratio $1.50$, but they are reading a sample the cap could not hold together.

\begin{figure}[ht]
\centering
\includegraphics[width=\textwidth]{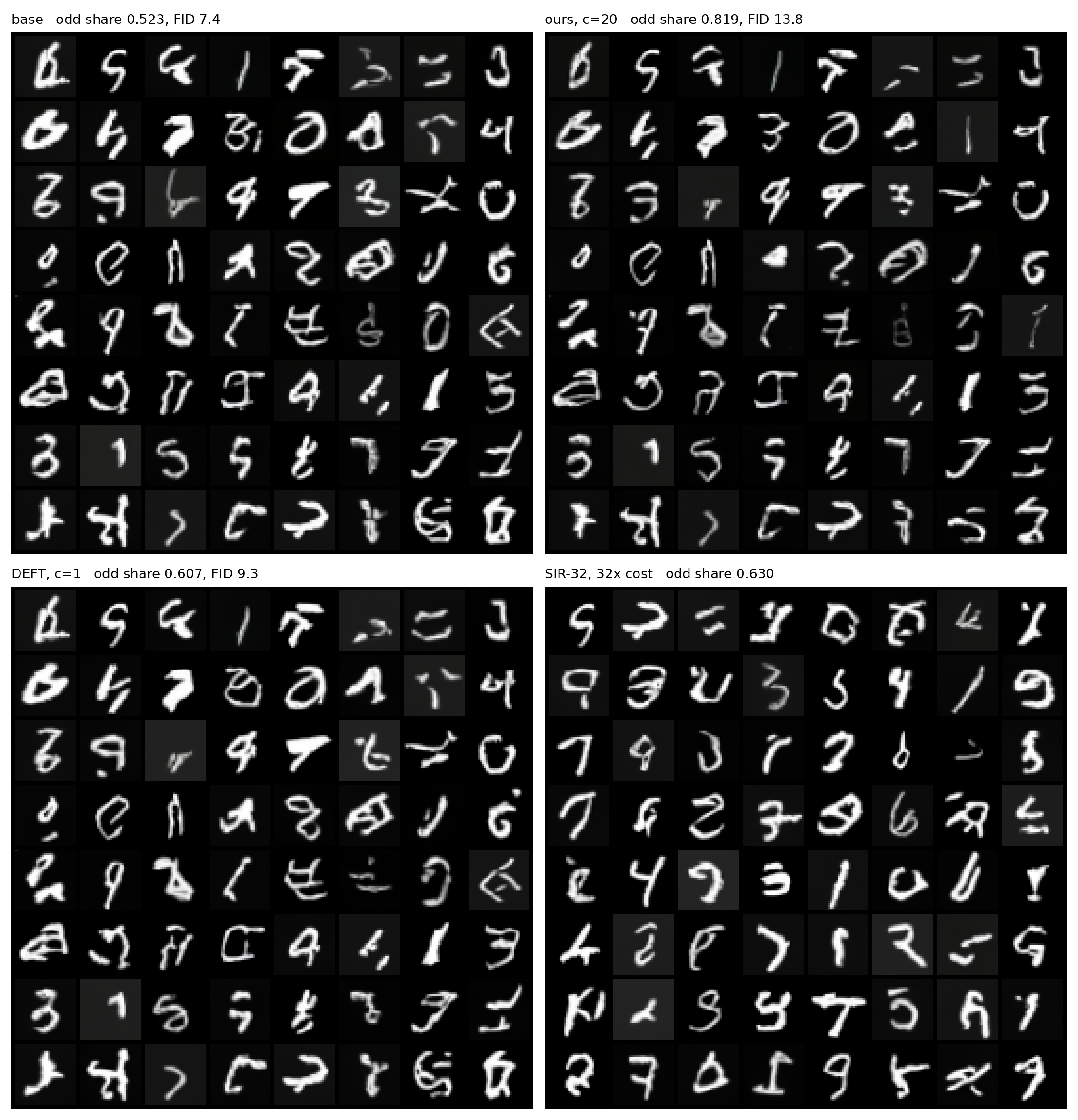}
\caption{MNIST, odd digits, the first $64$ samples of seed $0$, uncurated. Base and ours at $c = 20$
above, DEFT at $c = 1$ and SIR-32 below, each panel labelled with the share of odd digits and with its
FID, as in Figure \ref{fig:cifar_samples}, DEFT is drawn at the scale of its sweep whose FID is closest
to ours. The guided panels share the initial noise with the base panel. The base model draws some digits
mirrored.}
\label{fig:mnist_s_samples}
\end{figure}

\subsection{Final hyperparameter values} \label{app:hyperparams}
Shared by both base models unless stated otherwise.
\begin{itemize}[leftmargin=*]
\item \emph{Weight and reference sample}: $n = 20{,}000$ samples from the base model, $12$ noise levels per
sample in $\sigma_{T-t} \in [0.35, 0.95]$, $\epsilon = 0.05$ in $w$. Classifier heads: a five-layer
convolutional network, the MNIST heads work on a canonical $28\times28$ grey image, with label smoothing
$0.1$ and noise augmentation.
\item \emph{Our estimator}: trial network of three blocks, width $32$ doubling with depth, activation
budget $\mathcal{A} = 3$, weight budget $\Lambda = 5,$ test network the difference of two such networks.
Minimax: $6000$ descent steps at batch $128$, two ascent steps per descent step, Adam at $10^{-4}$ and
$4 \cdot 10^{-4}$ with $\beta = (0, 0.9)$, warmup over the first tenth and a cosine decay after it, an
exponential moving average of the trial weights at $0.999$, gradient norm clipped at $5$, $\mu$ calibrated
so that the penalty equals the value term at the end of the warmup.
\item \emph{DEFT}: the head of their equation (12) --- a U-Net trunk on the concatenation of the noisy
state, the Tweedie estimate and the log-weight gradient, $48$ base channels, $3$ downsampling stages on
CIFAR-10 and $2$ on MNIST, its last layer initialised at zero, plus a two-layer time embedding returning
the multiple of the gradient, initialised at one, Adam \citep{kingma2015}, learning rate
$2\cdot10^{-4}$, batch $128$, $8000$ steps, gradient norm clipped at $5$.
\item \emph{DOIT}: $M = 32$ antithetic look-ahead samples at $\eta = 0.4$, correction over steps
$1 \dots \lfloor 0.67 L \rfloor$ of the $L = 100$, $h(\cdot, 0) \propto \exp(r/\tau)$ with $\tau$ swept,
these are the reference implementation's values.
\item \emph{Sampling}: $100$ DDIM steps at $\eta = 0$, and at $\eta = 0.4$ where a second regime is
reported, correction capped at $10\%$ of $\|\hat\epsilon\|$ for every guided
method, $c \in \{0.25, 0.5, 1, 2, 3, 5, 10, 20\}$ on CIFAR-10 and $\{1, 5, 20\}$ on MNIST,
$3$ seeds of $2000$ samples.
\end{itemize}
\subsection{What the guidance cap does} \label{app:cap}

Every guided method runs under one shared constraint, a correction capped at $10\%$ of
$\|\hat\epsilon\|$ per sample and step. A constraint that binds unevenly would shape the comparison, so
each run records how hard it works.

\begin{table}[ht]
\centering
\caption{The cap at $c = 20$ on CIFAR-10 ($\gamma = 1$ for DOIT, and $\gamma = 0.1$, the small-$\gamma$
setting of its authors), where it binds at all: the mean
uncapped correction against the cap of $0.1$, the share of guided steps it binds on, and what removing
it does. At $c = 5$ and below it binds on at most $1\%$ of the steps of either trained method. The
uncapped point was not sampled for the reported DEFT head, nor was that head run under the stochastic
sampler, so those cells are empty.}
\label{tab:cap}
\small
\begin{tabular}{llccc}
\hline
sampler & method & $\overline{|\delta|}$ & binds & capped $\to$ uncapped \\
\hline
$\eta = 0$ & ours & $0.031$ & $5.3\%$ & $0.831 / 32.9 \to 0.739 / 43.4$ \\
 & DEFT & $0.158$ & $77.9\%$ & $0.941 / 97.7 \to$ --- \\
 & DOIT & $2.4$ & $100\%$ & $0.689 / 116.7 \to 0.995 / 463.2$ \\
$\eta = 0.4$ & ours & $0.031$ & $5.0\%$ & $0.854 / 34.6 \to 0.805 / 46.2$ \\
 & DOIT, $\gamma = 1$ & $2.3$ & $100\%$ & $0.464 / 70.2 \to 0.996 / 437.1$ \\
 & DOIT, $\gamma = 0.1$ & $0.23$ & $95.2\%$ & $0.462 / 70.0 \to 0.889 / 221.6$ \\
\hline
\end{tabular}
\end{table}

The cap does different work for each. Ours asks for less than a third of it on average and removing it
costs nine points of share and ten of FID, DEFT asks for more than the cap itself at $c = 20$ and is
clipped on four fifths of its steps, for DOIT the cap is the only thing between the run and divergence,
and at $\gamma = 20$ the uncapped run collapses to FID $719$ with no animals left. What matters for the
paper is that the comparison does not rest on it: the two frontiers meet between FID $28$ and $30$,
which is DEFT's $c \in \{1, 2\}$ and our $c \in \{10, 20\}$, and there the cap binds on none of DEFT's
steps and at most a twentieth of ours.

\subsection{Approximation floors at the operative smoothness} \label{app:floors}

Table \ref{tab:penalties} claims the strong residual carries a first-order approximation floor,
non-informative at $\beta = 2$. A floor belongs to the class, not the estimator, so it is measured
directly, at the member of the class approximating $h^\ast$ in value, gradient and time derivative at
once. Orders $3$, $2$, $1$ give $\beta = 4$, $3$, $2,$ order $1$, being only $C^0$, is not an admissible
trial class, but its floor is defined.

\begin{table}[ht]
\centering
\caption{Approximation floors against resolution, $d = 1$, $4$ replicates, slopes over
$N \in [2, 24]$, predictions $-2(\beta-1)$ and $-2(\beta-2)$.}
\label{tab:floors}
\small
\begin{tabular}{cccc}
\hline
$\beta$ (order) & weak floor & strong floor & strong\,/\,weak at $N = 24$ \\
\hline
$4$ ($3$) & $N^{-1.44}$ & $N^{-1.89}$ & $7.4$ \\
$3$ ($2$) & $N^{-1.81}$ & $N^{-1.00}$ & $311$ \\
$2$ ($1$) & $N^{-2.46}$ & $\mathbf{N^{+0.29}}$ & $8500$ \\
\hline
\end{tabular}
\end{table}

At $\beta = 2$ the strong floor stops decaying while the weak one keeps falling, the gap widening from
$7$ through $311$ to $8500$. The fitted slopes are not the predicted ones, the grid binding before the
smoothness does at $\beta = 4$ and $3$, and $h^\ast$ being smoother than the $W^{2,\infty}$ the
prediction assumes at $\beta = 2,$ the table supports the ordering and the flattening, not the exponents.

\subsection{The finite-\texorpdfstring{$\mathrm{s}$}{s} regime} \label{app:finiteS}

The rate experiment runs at the boundary case $\mathrm{s} = \infty$, where the clipping never binds. What
Assumption \ref{ass:bounded} buys beyond it is the factor $\Xi_{\mathrm{s}}^{4/(\mathrm{s}+4)}
\varepsilon_n^{\mathrm{s}/(\mathrm{s}+4)}$, so what is worth testing is whether the error degrades where
that envelope says it should. A fraction $0.8$ of the atoms of Appendix \ref{app:synthetic} is therefore
given a weight $w_{\mathrm{low}}$ descending the ladder $1, 0.3, 0.1, 0.03, 3 \cdot 10^{-3}$, the rest
keeping $3$, with class, sample sizes, oracle grids and four replicates unchanged. The window is not: on
$I = [0.40, 0.85]$ the noise is comparable to the radius of the data, the smoothing averages the
low-weight region away, and $h^\ast$ never falls below $\mathrm{b}$, however small $w_{\mathrm{low}}$ is.
The ladder is read on $I = [0.96, 0.995]$ instead, with $r = 2$ and $N \leq 32$ for the smaller $\sigma,$
$\mathrm{b}$ is the level the theorem prescribes at the minimising $\mathrm{s}$, exact here because $p_0$
is atomic.

\begin{table}[ht]
\centering
\caption{The finite-$\mathrm{s}$ ladder, $d = 1$, $4$ replicates, each cell at its own oracle
$(N, \mu),$ slopes over the seven sample sizes, ``clip'' the mass on which $h^\ast < \mathrm{b}$ at
$n = 10^6$.}
\label{tab:finites}
\small
\begin{tabular}{cccccc}
\hline
$\min w$ & $n = 10^3$ & $n = 10^6$ & measured & envelope & clip \\
\hline
$0.5$ (base) & $3.0 \times 10^{-3}$ & $1.0 \times 10^{-5}$ & $0.787$ & $0.779$ & $0.00$ \\
$1$ & $1.2 \times 10^{-2}$ & $2.2 \times 10^{-5}$ & $0.919$ & $0.779$ & $0.00$ \\
$0.3$ & $2.6 \times 10^{-2}$ & $2.0 \times 10^{-5}$ & $1.014$ & $0.779$ & $0.00$ \\
$0.1$ & $4.3 \times 10^{-2}$ & $2.2 \times 10^{-5}$ & $1.076$ & $0.337$ & $0.00$ \\
$0.03$ & $1.5 \times 10^{-1}$ & $2.9 \times 10^{-3}$ & $0.547$ & $0.004$ & $0.40$ \\
$3 \times 10^{-3}$ & $2.6 \times 10^{-1}$ & $\mathbf{2.1 \times 10^{-1}}$ & $\mathbf{0.028}$ & $0.002$ & $0.44$ \\
\hline
\end{tabular}
\end{table}

The clipping mass separates the two regimes. Down to $\min w = 0.1$ it does not bind at $n = 10^6$ and
the error there is the reference rung's $2 \times 10^{-5}$: a small $w$ on a set the smoothing averages
away costs a constant, not a rate. At $0.03$ it binds at every sample size, the guaranteed exponent
collapses and the error is two orders of magnitude worse, at $3 \cdot 10^{-3}$ there is no convergence at
all over three decades of $n$. The envelope being an upper bound, the measured slope need only be at
least as steep, which it is at every rung, what the table shows is that the bound goes flat exactly where
the estimator stops converging.

\end{document}